\documentclass[aap]{imsart}

\usepackage{graphicx} 
\usepackage{wrapfig}
\usepackage[numbers]{natbib}

\usepackage{amsmath,amssymb,amsfonts,amsthm}
\usepackage{mathtools}
\usepackage{bm}
\usepackage{bbm}
\usepackage{psfrag}
\usepackage{color}
\usepackage{caption,subcaption}

\usepackage{enumitem}
\setlist{leftmargin=1.6em}

\usepackage[hidelinks]{hyperref} 

\newtheorem{theorem}{Theorem}
\newtheorem{proposition}{Proposition}
\newtheorem{assumption}{Assumption}
\newtheorem{lemma}{Lemma} 
\newtheorem{corollary}{Corollary}

\theoremstyle{definition}

\theoremstyle{remark}
\newtheorem{remark}{Remark}

\usepackage{graphicx}
\usepackage{xcolor}

\usepackage{algorithm}

\newcommand{\vh}[0]{\bm {h}}

\newcommand{\vp}[0]{\bm {p}}
\newcommand{\vv}[0]{\bm {v}}
\newcommand{\vx}[0]{\bm {x}}
\newcommand{\vy}[0]{\bm {y}}
\newcommand{\vz}[0]{\bm {z}}
\newcommand{\vu}[0]{\bm {u}}

\newcommand{\vW}[0]{\bm {W}}
\newcommand{\vX}[0]{\bm {X}}
\newcommand{\vY}[0]{\bm {Y}}

\newcommand{\vT}[0]{\bm{T}}

\newcommand{\vw}[0]{\bm{w}}

\newcommand{\vzero}[0]{\bm {0}}
\newcommand{\vone}[0]{\bm {1}}

\newcommand{\Cov}[0]{\mathrm{Cov}}
\newcommand{\Var}[0]{\mathrm{Var}}

\newcommand{\calF}[0]{\mathcal{F}}

\newcommand{\calH}[0]{\mathcal{H}}

\newcommand{\calX}[0]{\mathcal{X}}

\newcommand{\calY}[0]{\mathcal{Y}}

\newcommand{\E}[0]{\mathbb{E}}

\newcommand{\R}[0]{\mathbb{R}}

\newcommand{\Prob}[0]{\mathbb{P}}

\newcommand{\vertiii}[1]{{\left\vert\kern-0.25ex\left\vert\kern-0.25ex\left\vert #1 
    \right\vert\kern-0.25ex\right\vert\kern-0.25ex\right\vert}}

\newcommand{\cA}{\mathcal{A}}

\newcommand{\cC}{\mathcal{C}}

\newcommand{\cH}{\mathcal{H}}

\newcommand{\cN}{\mathcal{N}}

\newcommand{\cQ}{\mathcal{Q}}

\newcommand{\cT}{\mathcal{T}}

\newcommand{\cX}{\mathcal{X}}
\newcommand{\cY}{\mathcal{Y}}
\newcommand{\cZ}{\mathcal{Z}}

\newcommand{\BB}{\mathbb{B}}

\newcommand{\EE}{\mathbb{E}}

\newcommand{\GG}{\mathbb{G}}

\newcommand{\NN}{\mathbb{N}}

\newcommand{\PP}{\mathbb{P}}

\newcommand{\RR}{\mathbb{R}}

\newcommand{\vasti}{\bBigg@{3.5 }}
\newcommand{\vast}{\bBigg@{4}}
\newcommand{\Vast}{\bBigg@{5}}
\newcommand{\Vastt}{\bBigg@{7}}

\makeatother

\newcommand{\be}{\begin{equation}}
\newcommand{\ee}{\end{equation}}
\newcommand{\ba}{\begin{align}}
\newcommand{\ea}{\end{align}}
\newcommand{\baa}{\begin{align*}}
\newcommand{\eaa}{\end{align*}}

\newcommand{\argmin}{\mathop{\mathrm{argmin}}}

\DeclareMathOperator{\interior}{int}
\DeclareMathOperator{\diam}{diam}
\DeclareMathOperator{\inte}{int}

\DeclareMathOperator{\dist}{dist}

\begin{document}

\begin{frontmatter}
\title{Stability and statistical inference for \\semidiscrete optimal transport maps}
\runtitle{Stability and inference for semidiscrete OT maps}

\begin{aug}
\author[A]{\fnms{Ritwik}~\snm{Sadhu}\ead[label=e1]{rs2526@cornell.edu}},
\author[B]{\fnms{Ziv}~\snm{Goldfeld}\ead[label=e2]{goldfeld@cornell.edu}},
\and
\author[A]{\fnms{Kengo}~\snm{Kato}\ead[label=e3]{kk976@cornell.edu}}
\address[A]{Department of Statistics and Data Science,
Cornell University\printead[presep={,\ }]{e1,e3}}

\address[B]{School of Electrical and Computer Engineering, Cornell University\printead[presep={,\ }]{e2}}
\end{aug}

\begin{abstract}
We study statistical inference for the optimal transport (OT) map (also known as the Brenier map) from a known absolutely continuous reference distribution onto an unknown finitely discrete target distribution. We derive limit distributions for the $L^p$-error with arbitrary $p \in [1,\infty)$ and for linear functionals of the empirical OT map, together with their moment convergence. The former has a non-Gaussian limit, whose explicit density is derived, while the latter attains asymptotic normality. 
For both cases, we also establish consistency of the nonparametric bootstrap. The derivation of our limit theorems relies on new stability estimates of functionals of the OT map with respect to the dual potential vector, which may be of independent interest.  We also discuss applications of our limit theorems to the construction of confidence sets for the OT map and inference for a  maximum tail correlation.  \textcolor{black}{Finally, we show that, while the empirical OT map does not possess nontrivial weak limits in the $L^2$ space, it satisfies a central limit theorem in a dual H\"{o}lder space, and the Gaussian limit law attains the asymptotic efficiency bound.}
\end{abstract}



\end{frontmatter}

\section{Introduction}

\subsection{Overview}

Optimal transport (OT) provides a versatile framework to compare probability measures and has seen a surge of applications in statistics, machine learning, and applied mathematics; see, e.g., \cite{panaretos2020invitation,peyre2019computational,santambrogio2017euclidean}, among many others. We refer to \cite{ambrosio2005,villani2008optimal,santambrogio15} as standard references on OT. For Borel probability measures $R$ and $P$ on $\R^d$ with finite second moments, the Kantorovich OT problem with quadratic cost reads as
\begin{equation}
\min_{\pi \in \Pi (R,P)}  \int \frac{1}{2} \| \vy-\vx \|^2 \, d \pi(\vy,\vx),
\label{eq: Kantorovich}
\end{equation}
where $\Pi (R,P)$ denotes the collection of couplings of $R$ and $P$ (i.e., each $\pi \in \Pi(R,P)$ is a probability measure on the product space with marginals $R$ and $P$). 
A central object of interest in OT theory is the OT map, or the Brenier map \cite{brenier1991polar}. 
For absolutely continuous $R$, the unique optimal solution $\pi^*$ to \eqref{eq: Kantorovich} exists and  concentrates on the graph of a deterministic map $T^*$, called the OT map. The OT map has been applied to transfer learning and domain adaptation, among many others. Also, the OT map can be seen as a multivariate extension of the quantile function \cite{carlier2016vector,chernozhukov2017monge, ghosal2019multivariate,
hallin2021distribution} and has been recently applied to causal inference \cite{torous2021optimal}.

Motivated by these applications, the statistical analysis of the OT map has seen increased interest, mostly focusing on 
estimation error rates \cite{chernozhukov2017monge, hutter2021minimax, ghosal2019multivariate,deb2021rates, manole2021plugin, pooladian2021entropic, pooladian2022debiaser, divol2022optimal, pooladian2023minimax}; see a literature review below for a further discussion on these references. Still, there is much to be desired on statistical \textit{inference} for the OT map, such as valid testing and the construction of confidence sets, both of which hinge on a limit distribution theory. 
 Alas, to the best of our knowledge, there have been no limit distribution results for the OT map, except for the $d=1$ case where the OT map agrees with the composition of the quantile and distribution functions (cf. Chapter 2 in \cite{santambrogio15}).\footnote{\textcolor{black}{One exception is the recent preprint \cite{manole2023central}, which appeared on the arXiv after the initial submission of this manuscript; see the discussion at the end of Section \ref{sec: literature}.}}
Indeed, deriving distributional limits of OT maps poses a significant challenge from a probabilistic perspective. This is because the OT map is implicitly defined via the gradient of a dual potential function that is an optimal solution to a certain optimization problem, and hence its dependence on the data is highly complicated. 

The present paper tackles this challenge and derives several limit distribution results for the empirical OT map under the \textit{semidiscrete} setting. Semidiscrete OT refers to the case where the input distribution $R$ is absolutely continuous while the target $P$ is (finitely) discrete. Univariate quantiles for discrete outcomes have been investigated by  \cite{machado2005quantiles,chernozhukov2020generic}, and semidiscrete OT maps can be seen as a multivariate analog of discrete quantile functions. Discrete outcomes appear in many applications, including count data, ordinal data, discrete duration data, and data rounded to a finite number of discrete values. Additionally, semidiscrete OT has been applied to computer graphics \cite{de2012black,levy2015numerical}, fluid dynamics \cite{de2015power,gallouet2018lagrangian}, and resource allocation problems \cite{hartmann2020semi}.

For known reference measure $R$, which is natural when viewing the OT map as a vector quantile function, we derive limit distributions for the $L^p$-error $\| \hat{\vT}_n-\vT^* \|_{L^p(R)}^p$ with arbitrary $p \in [1,\infty)$ and for linear functionals of the form $\langle \bm{\varphi},\hat{\vT}_n \rangle_{L^2(R)} = \int \langle \bm{\varphi},\hat{\vT}_n \rangle \, dR$. Here $\hat{\vT}_n$ is the empirical OT map transporting $R$ onto the empirical distribution, and $\bm{\varphi}$ is a suitable Borel vector field. 
The limit distribution for the $L^p$-error is non-Gaussian, while that for the linear functional is a centered Gaussian. In addition to distributional convergence, we establish moment convergence of any polynomials of the $L^p$-error and linear functional, which in particular leads to an asymptotic expansion of the squared $L^2(R)$-risk of $\hat{\vT}_n$.  Furthermore, we derive an explicit form of Lebesgue density of the non-Gaussian limit law for the $L^p$-error functional. Finally, we establish  consistency of the nonparametric bootstrap for both cases. 

These statistical results enable us to perform various inference tasks for the OT map, such as the construction of $L^p$-confidence sets for $\vT^*$ and confidence intervals for $\langle \bm{\varphi},\vT^*\rangle_{L^2(R)}$.
One drawback of $L^p$-confidence sets is that they are difficult to visualize compared to $L^\infty$-confidence bands. To address this, we discuss a method to construct a confidence band derived from an $L^p$-confidence set that satisfies a certain relaxed coverage guarantee (cf. \cite[Section 5.8]{wasserman2006all}). Also, we develop a method to construct confidence intervals for a version of the maximum tail correlation \cite{beirlant2020center}, which serves as a risk measure for multivariate risky assets and is an instance of a linear functional of the OT map. Small-scale simulation experiments confirm that the bootstrap works reasonably well for $L^1$-confidence sets and inference for the maximum tail correlation. 

It is important to note that the semidiscrete OT map  is a discrete mapping that is piecewise constant over the partition of $\R^d$ defined by the \textit{Laguerre cells} (\cite{aurenhammer1987power,aurenhammer1998minkowski}; Chapter 5 in \cite{peyre2019computational}). 
\textcolor{black}{As such, the empirical OT map at a fixed interior point of each Laguerre cell exactly coincides with the population OT map as soon as the sample size is large enough, which effectively means that no nondegenerate \textit{pointwise} limit distribution exists for the empirical OT map.} Nevertheless, this observation does not contradict our limit theorems since both functionals aggregate the contributions of the empirical OT map near the boundaries of the Laguerre cells, which add up to nontrivial weak limits. \textcolor{black}{See the discussion at the end of Section \ref{sec: semidiscrete OT} for more details.}

The proof of our limit theorems relies on new stability estimates of the OT map with respect to the dual potential vector---this result may be of independent interest. Specifically, we show that the $L^p$-error functional is Hadamard \textit{directionally} differentiable with a nonlinear derivative as a mapping of the dual potential vector. For the linear functional, we establish (Fr\'{e}chet) differentiability and characterize the derivative. The derivation of these stability estimates, which relies on the careful analysis of the facial structures of the Laguerre cells, is the main technical contribution of the present paper. Given the stability results, the limit distributions follow by combining the extended delta method \cite{romisch2004} and a central limit theorem (CLT) for the empirical dual potential vector  \cite{del2022central}.
The bootstrap consistency readily follows for the linear functional, as it is Fr\'{e}chet differentiable with respect to the dual potential vector. For the $L^p$-error functional, while the derivative is nonlinear, the bootstrap is still shown to be consistent, thanks to the specific structure of the derivative; see the discussion after Theorem \ref{prop: bootstrap consistency}.

\textcolor{black}{
Having established limit theorems for the $L^p$-error and linear functionals of the empirical OT map, the question remains whether the empirical OT map has a weak limit in a function space. A natural function space to work in would be $L^2(R)$, but it turns out that the empirical OT map does not possess nontrivial weak limits in that space, under reasonable regularity conditions on the reference measure. Still, we establish that the empirical OT map satisfies a CLT in the dual space of the $\alpha$-H\"{o}lder space with any $\alpha \in (0,1]$. The idea to consider weak limits of function-valued estimators in normed spaces endowed with a weaker norm than $L^2$ is partially inspired by \cite{castillo2013nonparametric,castillo2014bernstein}, where this approach was used to derive nonparametric Bernstein-von Mises theorems. Additionally, we show that the empirical OT map is asymptotically efficient in the H\'{a}jek-Le Cam sense (cf. Chapter 3.11 in \cite{van1996weak}) under the said dual H\"{o}lder space setting. To the best of our knowledge, this is the first paper to derive asymptotic efficiency results for OT map estimation in an infinite-dimensional setting.
}

\subsection{Literature review}
\label{sec: literature}

The literature on statistical OT has rapidly expanded in recent years, so we confine ourselves to discussing only references directly related to OT map estimation and semidiscrete OT.
The literature on OT map estimation has mostly focused on continuous targets. \cite{chernozhukov2017monge} propose viewing the OT map as a multivariate extension of the quantile function and establish local uniform consistency of the empirical OT map; see also \cite{carlier2016vector,ghosal2019multivariate,hallin2021distribution}. \cite{hutter2021minimax} derive minimax rates for estimating the OT map and analyze wavelet-based estimators. For continuous targets, the results of \cite{hutter2021minimax} suggest that the minimax rate (for the $L^2$-loss) would be $n^{-1/d}$ when no further smoothness assumptions are imposed on the dual potentials (although the $n^{-1/d}$-minimax rate is formally a conjecture). Thus, estimation of the OT map suffers from the `curse of dimensionality', similarly to the OT cost itself \cite{dudley1969speed,fournier2015rate,weed2019sharp,niles2019estimation}. 
See also \cite{deb2021rates,manole2021plugin,pooladian2021entropic,divol2022optimal} for other contributions on estimation of the OT map. None of the above references contain limit distribution results for the OT map.

An object related to the OT map is an entropic OT (EOT) map \cite{pooladian2021entropic}, which is defined by the barycentric projection of the optimal coupling for the regularized OT problem with entropic penalty \cite{cuturi2013lightspeed,peyre2019computational}. The EOT map approximates the standard OT map as the regularization parameter approaches zero \cite{mikami2004monge,carlier2022convergence}. For a \textit{fixed} regularization parameter, \cite{gonzalez2022weak,rigollet2022sample,gonzalez2022weak,goldfeld2022EOT} derive the parametric convergence rate for the empirical EOT map. Furthermore,  \cite{goldfeld2022EOT}  establish a limiting Gaussian distribution, bootstrap consistency, and asymptotic efficiency for the empirical EOT map; see also \cite{gonzalez2022weak}.  However, the results of \cite{goldfeld2022EOT} and \cite{gonzalez2022weak} do not extend to vanishing regularization parameters and hence do not cover standard OT map estimation.  

For semidiscrete OT, \cite{kitagawa2019convergence} and \cite{bansil2022quantitative} derive several important structural results, including regularity of the dual objective function. \cite{altschuler2022asymptotics} derive an asymptotic expansion of the EOT cost in the semidiscrete case when the regularization parameter tends to zero, showing faster convergence than the continuous-to-continuous case. In the statistics literature, there are two recent papers related to ours.  The first is \cite{del2022central}, which derives limit distributions for the OT cost and dual potential vector in the semidiscrete setting. The derivation of our limit theorems builds on their work, but as noted before, the bulk of our effort is devoted to establishing (directional) derivatives of the $L^p$-error and linear functionals of the OT map with respect to the dual potential vector; these are not covered by \cite{del2022central} and require substantial work.\footnote{The most recent update of \cite{del2022central} available at \texttt{https://hal.science/hal-03232450v3} added a remark (Remark 4.9) showing that $\sqrt{n}\| \hat{\vT}_n - \vT^* \|_{L^2(R)}^2$ is stochastically bounded, but leaves the problem of finding weak limits of $\hat{\vT}_n-\vT^*$ open.  } Also, in contrast to \cite{del2022central}, we allow the support of the reference measure $R$ to be unbounded, which requires additional work to derive a CLT for the dual potential vector in our setting.
The second related work is \cite{pooladian2023minimax}, which shows that empirical EOT maps with vanishing regularization parameters achieve the parametric convergence rate toward the standard OT map (under the squared $L^2$-loss) in the semidiscrete setting, showing that OT map estimation in the semidiscrete case is free of the curse of dimensionality. No limit distribution results are derived in that paper.

\textcolor{black}{
Finally, we comment on the recent preprint \cite{manole2023central}, which was posted on the arXiv after the initial submission of this manuscript. In \cite{manole2023central}, the authors establish pointwise CLTs for kernel estimators of smooth OT maps for absolutely continuous distributions supported on the flat torus. Their scope and proof techniques are substantially different.
}

\subsection{Organization}
The rest of this paper is organized as follows. 
\textcolor{black}{In Section \ref{sec: preliminaries}, we present background material of semidiscrete OT, Hadamard differentiability, and the extended delta method.}
Section \ref{sec: main} collects the stability results and limit theorems for the two functionals of interest. 
 Section \ref{sec: applications} presents applications and simulation results. \textcolor{black}{In Section \ref{sec: dual holder}, we show that the empirical OT map does not possess nontrivial weak limits in $L^2(R)$, but nonetheless satisfies a CLT in a dual H\"{o}lder space; that section also shows that the empirical OT map is asymptotically efficient in the said Banach space setting.} All the proofs are gathered in Section \ref{sec: proofs}. Section \ref{sec: conclusion} leaves some concluding remarks.

\subsection{Notation}
For $a,b \in \R$, we use the notation $a \vee b = \max \{a,b \}$ and $a \wedge b = \min \{ a,b \}$.
We use $\| \cdot \|$ and $\langle \cdot, \cdot \rangle$ to denote the Euclidean norm and inner product, respectively. Vectors \textcolor{black}{and matrices} are written in boldface letters. Let $\vzero$ and $\vone$ denote the vectors of all zeros and ones, respectively; their dimensions should be understood from the context. For $r > 0$, let $B_r$ denote the closed ball with center $\vzero$ and radius $r$. 
For a subset $A$ of a Euclidean space, the boundary and interior are denoted by $\partial A$ and $\inte (A)$, respectively. Also, define $\dist (\vy,A) \coloneqq \inf \{ \| \vy-\vy' \| : \vy' \in A \}$.
\textcolor{black}{For $d \in \NN$ and $0 \le r \le d$, $\cH^{r}$ denotes the $r$-dimensional Hausdorff measure on $\R^d$,
\[
\cH^{r}(A) = \lim_{\delta \to 0} \inf \Big \{ \frac{\pi^{r/2}}{\Gamma(\frac{r}{2}+1)} \sum_{j=1}^\infty \Big ( \frac{\diam C_j}{2}\Big)^r :  A \subset \bigcup_{j=1}^\infty C_j, \ \diam C_j \le \delta \Big\}, \quad A \subset \R^d,
\]
where $\diam C_j$ is the diameter of $C_j$ and $\Gamma (\cdot)$ is the gamma function.}
See \cite{evans1991measure} for a textbook treatment of Hausdorff measures.

\section{Preliminaries}
\label{sec: preliminaries}

\subsection{Semidiscrete optimal transport}
\label{sec: semidiscrete OT}
We consider a semidiscrete OT problem under quadratic cost. Let $R$ be a Borel probability measure on $\R^d$ with finite second moment and $P$ be a finitely discrete distribution on $\R^d$ with support $\cX = \{ \vx_1,\dots,\vx_N \}$, where $\vx_1,\dots,\vx_N$ are all distinct ($N \ge 2$). Consider the Kantorovich problem (\ref{eq: Kantorovich}). 
 Assuming that $R$ is absolutely continuous with respect to the Lebesgue measure, Brenier's theorem \cite{brenier1991polar} yields that the Kantorovich problem (\ref{eq: Kantorovich}) admits a unique optimal solution (coupling) $\pi^*$. Furthermore, the coupling $\pi^*$ is induced by an $R$-a.s. unique map $\vT^*: \R^d \to \cX$, i.e.,  $\pi^*$ agrees with the joint law of $\big(\vY,\vT^*(\vY)\big)$ for $\vY \sim R$. We refer to $\vT^*$ as the \textit{OT map} transporting $R$ onto $P$.

\textcolor{black}{The duality theory for the OT problem (cf. \cite{ambrosio2005, villani2008optimal,santambrogio15}) plays a central role in the derivation of the limit theorems for the empirical OT map. Regarding $P$ as a probability measure on $\calX$, the (semi)dual problem for (\ref{eq: Kantorovich}) reads as 
\begin{equation}
\label{eq: Kantorovich dual}
 \max_{\psi \in L^1(P)}  \int_{\calX} \psi \, dP +\int_{\R^d} \psi^c \, dR,
\end{equation}
and the maximum is attained; here, $\psi^c(\vy) = \min_{\vx \in \calX}(\|\vy-\vx\|^2/2-\psi(\vx))$ is the \textit{$c$-transform} of $\psi$ for the cost $c(\vx,\vy) = \| \vx-\vy\|^2/2$ (cf. Theorem 5.9 in \cite{villani2008optimal}; Theorem 6.1.5 in \cite{ambrosio2005}). 
Setting $\vz = (\psi(\vx_1), \dots, \psi(\vx_N))^\intercal \in \R^N$, the dual problem (\ref{eq: Kantorovich dual}) reduces to
\begin{equation}
\label{eq: semi_dual}
\max_{\vz\in\R^N} \, \langle \vz,\vp \rangle + \int \min_{1 \leq i\leq N} \left ( \frac{1}{2} \| \vy-\vx_i \|^2 - z_i\right )\,dR(\vy),
\end{equation}
where $\vp = (p_1,\dots,p_N)^{\intercal} \coloneqq (P(\{ \vx_1 \}),\dots,P(\{\vx_N\}))^{\intercal}$ is the simplex (frequency) vector corresponding to $P$. 
We shall call any optimal solution $\vz^*$ to the dual problem (\ref{eq: semi_dual}) a \textit{dual potential vector}. Now, for $\psi^c(\vy) = \min_{1 \le i \le N}(\|\vy-\vx_i\|^2/2-z_i^*)$, the OT map $\vT^*$ transporting $R$ onto $P$ agrees with $\vy - \nabla_{\vy}\psi^c(\vy)$ (cf. Theorem 1.22 in \cite{santambrogio15} and its proof; note that $\psi^c$ is an optimal potential from $R$ to $P$), which simplifies to 
\[
\vT^*(\vy) 
=\argmin_{\vx_i: 1 \le i \le N} \left ( \frac{1}{2} \|\vy-\vx_i\|^2 - z_i^* \right )
\]
for $R$-a.e. $\vy$. For $R$-a.e. $\vy$, the argmin on the right-hand side is a singleton.}

For $\vz \in \R^N$, define the \textit{Laguerre cells} $\{ C_i(\vz) \}_{i=1}^N$ as 
\[
\begin{split}
  C_i(\vz) &\coloneqq \bigcap_{\substack{j \ne i \\ 1 \le j \le N}} \left \{\vy \in \R^d : \frac{1}{2}\|\vx_i - \vy\|^2 
 - z_i \leq \frac{1}{2} \|\vx_j - \vy\|^2 - z_j\right \} \\
 &= \bigcap_{\substack{j \ne i \\ 1 \le j \le N}}\big\{\vy \in \R^d: \langle \vx_i-\vx_j,\vy \rangle \geq b_{ij} (\vz) \big \},
\end{split}
\]
with $b_{ij}(\vz) \coloneqq  b_{ij}(z_i,z_j) \coloneqq (\|\vx_i\|^2 - \|\vx_j\|^2)/2 -z_i + z_j$. \textcolor{black}{By definition, Laguerre cells agree with Voronoi cells when $\vz$ is constant, i.e., $z_1=\cdots=z_N$.} Each Laguerre cell $C_i(\vy)$ is a polyhedral set defined by the intersection of $N-1$ half-spaces. Then, we see that
\begin{equation}
\vT^*(\vy) = \vx_i \ \  \text{for} \ \  \vy \in \inte \big(C_i(\vz^*)\big) \ \ \text{and} \ \ i \in \{1,\dots,N \}.
\label{eq: OT map}
\end{equation}
Since the Laguerre cells form a partition of $\R^d$ up to Lebesgue negligible sets, the expression (\ref{eq: OT map}) defines an $R$-a.e. defined map with values in $\cX$. Furthermore, since $\vT^*$ is a transport map, i.e., $\vT^*(\vY) \sim P$ for $\vY \sim R$, we have 
\[
\begin{split}
R\big (C_i(\vz^*)\big) =R\big ( \inte (C_i(\vz^*)) \big ) = \Prob \big(\vT^*(\vY) = \vx_i \big) = P(\{ \vx_i \}) = p_i > 0
\end{split}
\] 
for every $i \in \{ 1,\dots, N \}$. \textcolor{black}{We refer the reader to \cite{aurenhammer1987power,aurenhammer1998minkowski} and Chapter 5 in \cite{peyre2019computational} for more details about Laguerre cells.}

In the present paper, we assume that $R$ is a \textit{known} reference measure and are interested in making statistical inference on $\vT^*$ for unknown $P$. Given an i.i.d. sample $\vX_1,\dots,\vX_n$ from $P$, a natural estimator for $\vT^*$ is the empirical OT map $\hat{\vT}_n$ transporting $R$ onto the empirical distribution $\hat{P}_n = n^{-1}\sum_{i=1}^n \delta_{\vX_i}$.  
Set 
\[
\hat{\vp}_n = (\hat{p}_{n,1},\dots,\hat{p}_{n,N})^{\intercal} = \big(\hat{P}_n(\{\vx_1\}),\dots,\hat{P}_n(\{\vx_N\})\big)^\intercal
\]
as the empirical frequency vector.
Then, the empirical OT map $\hat{\vT}_n$ admits the expression
\[
\hat{\vT}_n (\vy) = \argmin_{\vx_i: 1 \le i \le N} \left ( \frac{1}{2} \|\vy-\vx_i\|^2 - \hat{z}_{n,i} \right ),
\]
where $\hat{\vz}_n =(\hat{z}_{n,1},\dots,\hat{z}_{n,N})^{\intercal}$ is an optimal solution to the dual problem (\ref{eq: semi_dual}) with $\vp$ replaced by $\hat{\vp}_n$.

We shall study limit theorems for the empirical OT map. The problem has a certain subtlety which we shall discuss here. 
In our semidiscrete setup, OT maps are piecewise constant functions with values in the discrete set $\cX$. Hence, it is not hard to see that, for every $\vy \in \inte(C_i(\vz^*))$, the empirical OT map $\hat{\vT}_n(\vy)$ \textit{exactly} coincides with $\vT^*(\vy)$ as $n \to \infty$. This effectively means that finding pointwise limit distributions is a vacuous endeavour; see Proposition \ref{prop: pointwise} ahead for the precise statement. However, this `super consistency' result has limited statistical values since (i) the population Laguerre cells are unknown;  (ii) there is no guarantee that a chosen reference point $\vy$ lies in the interior of one of the population Laguerre cells; (iii) the sample size needed to guarantee $\hat{\vT}_n(\vy) = \vT^*(\vy)$ relies on how close $\vy$ is to the boundary of $C_i(\vz^*)$; and (iv)  the super consistency result cannot capture the behavior of the OT map near the boundaries of the Laguerre cells. 

Because of these reasons, we shall analyze functionals of the empirical OT map other than pointwise ones. 
The preceding observation does not preclude the possibility of finding nontrivial weak limits for certain functionals of the empirical OT map, because the contributions from the behaviors of the empirical OT map near the boundaries of the (population) Laguerre cells may pile up and lead to nondegenerate limits.  
Specifically, we shall focus on the following functionals\footnote{We use $L^s$ instead of $L^p$ as $p$ might be confused with the probability simplex vector or its elements.}:
\begin{itemize}
\item $L^s$-error with arbitrary $s \in [1,\infty)$: 
\[
\| \hat{\vT}_n - \vT^* \|_{L^s(R)}^s= \int \| \hat{\vT}_n-\vT^* \|^s dR; 
\]
\item  Linear functional: 
\[\langle \bm{\varphi}, \hat{\vT}_n \rangle_{L^2(R)} = \int \langle \bm{\varphi}, \hat{\vT}_n \rangle dR\] for a suitable Borel vector field $\bm{\varphi}:\R^d \to \R^d$.
\end{itemize}
We will establish (nondegenerate) weak limits for those functionals (for the linear functional case, nondegeneracy of the weak limit relies on the choice of $\bm{\varphi}$). We also establish consistency of the nonparametric bootstrap for both functionals. These results enable performing various inference tasks for the OT map. See Section \ref{sec: applications} ahead for more details.

\begin{remark}[Known $R$ assumption]
Our assumption of a known reference measure $R$ is natural when we view the OT map as a multivariate extension of the quantile function \cite{chernozhukov2017monge}. Indeed, when $d=1$ and $R= \text{Unif}[0,1]$, the OT map $T^*$ agrees with the quantile function of $P$. In general, the OT map shares two important properties of the quantile function. (i) For $\vY \sim R$, $\vT^*(\vY)$ recovers the target distribution $P$, $\vT^*(\vY) \sim P$; and (ii) $\vT^*$ is a \textit{monotone}, i.e., $\langle \vT^*(\vy) - \vT^*(\vy'), \vy-\vy' \rangle \ge 0$, which follows because $\vT^*$ is cyclically monotone, i.e., it agrees with the gradient of a convex function. Common choices of the reference measure $R$ include the uniform distribution over the unit cube $[0,1]^d$ and the unit ball $B_1$. Another possible reference measure would be the standard Gaussian distribution.
\end{remark}

\begin{remark}[Computational aspects]
The decomposition into Laguerre cells is known as a `power diagram' in computational geometry \cite{aurenhammer1987power}, for which efficient algorithms are available \cite{bowyer1981computing,watson1981computing}. They can be implemented in several programming libraries, such as CGAL \cite{cgal1996cgal} and GEOGRAM \cite{levy2015geogram}. Efficient algorithms for computing dual potential vectors were proposed by \cite{merigot2011multiscale,kitagawa2019convergence}. Among others, \cite{kitagawa2019convergence} established linear convergence of a damped Newton algorithm for solving the dual problem \eqref{eq: semi_dual}. We refer to \cite{levy2018notions} and \cite[Chapter 5]{peyre2019computational} for a review of computational aspects of semidiscrete OT. 
\end{remark}

\subsection{Hadamard differentiability and extended delta method}
\label{sec: functional delta}
\textcolor{black}{
We briefly review Hadamard differentiability and the extended delta method. The reader is referred to \cite{van1996weak,vanderVaart1998asymptotic, romisch2004} for more details.
Let $\mathfrak{D}, \mathfrak{E}$ be normed spaces and consider a map $\phi : \Theta \to \mathfrak{E}$, where $\Theta$ is a nonepty subset of $\mathfrak{D}$. 
We say that $\phi$ is \textit{Hadamard directionally differentiable}  at $\theta \in \Theta$  if there exists a continuous map $\phi'_{\theta}: \mathfrak{D} \to \mathfrak{E}$ such that 
\begin{equation}
\lim_{n \to \infty} \frac{\phi(\theta_{n}) - \phi(\theta)}{t_n} = \phi_{\theta}'(h)
\label{eq: H derivative}
\end{equation}
for every sequence of positive reals $t_n \downarrow 0$ and every sequence $\theta_n \in \Theta$ with $t_n^{-1}(\theta_{n} -\theta) \to h$ as $n \to \infty$.
The derivative $\phi_{\theta}'$ need not be linear.  If \eqref{eq: H derivative} only holds for $h \in \mathfrak{D}_0$ for a subset $\mathfrak{D}_0 \subset \mathfrak{D}$, then we say that $\phi$ is 
Hadamard directionally differentiable  at $\theta \in \Theta$ \textit{tangentially} to $\mathfrak{D}_0$. In that case, the derivative $\phi_{\theta}'$ is defined only on $\mathfrak{D}_0$.
If the derivative $\phi_{\theta}'$ is linear, then we say that $\phi$ is \textit{Hadamard differentiable}  at $\theta$. When $\mathfrak{D}$ is finite-dimensional and $\Theta$ is open, Hadamard differentiability is equivalent to Fr\'{e}chet differentiability; cf. Example 3.9.2 in \cite{van1996weak}. Recall that $\phi$ is Fr\'{e}chet differentiable at $\theta$ if there exists a continuous linear map $\phi'_{\theta}: \mathfrak{D} \to \mathfrak{E}$ such that $\phi (\theta+h) - \phi(\theta) = \phi'_{\theta}(h) + o(\|h\|)$ as $\| h \| \to 0$. }

\textcolor{black}{
The extended delta method enables deriving limit theorems for Hadamard directionally differentiable functionals of convergent (in distribution) sequences of random elements. 
}

\begin{lemma}[Extended delta method; \cite{romisch2004}]
\label{lem: functional delta method}
\textcolor{black}{Consider the above setting. Let $T_n: \Omega \to \Theta$ be maps such that $r_n (T_n - \theta) \stackrel{d}{\to} T$ for some norning sequence $r_n \to \infty$ and Borel measurable map $T: \Omega \to \mathfrak{D}$ with values in a separable subset of $\mathfrak{D}$. Then, we have  $r_n \big(\phi(T_n) - \phi(\theta)\big) - \phi_{\theta}'(r_n(T_n-\theta)) \to 0$ in outer probability. In particular, $r_n \big(\phi(T_n) - \phi(\theta)\big) \stackrel{d}{\to} \phi_{\theta}'(T)$.}
\end{lemma}

\section{Stability and limit theorems for two functionals}
\label{sec: main}

Our approach to finding limit distributions for the preceding functionals relies on establishing (directional) differentiability with respect to the dual potential vector. Indeed, the bulk of our effort is devoted to proving those stability estimates. 
The desired limit distributions follow by combining a limit distribution result for the empirical dual potential vector (cf. \cite{del2022central}) and the extended delta method. After discussing regularity conditions on the reference measure $R$ and the Borel vector field $\bm{\varphi}$, we present key stability results, which would be of independent interest. Then, we move on to discuss the limit theorems.

\subsection{Assumptions}

Throughout, we maintain the following assumption on the reference measure $R$. 
\begin{assumption}[Regularity of $R$]
\label{asp: cont}
The reference measure $R$ has finite second moment and is absolutely continuous with Lebesgue density $\rho$ continuous $\cH^{d-1}$-a.e. on a closed set $\cY$ containing the support of $R$.
Furthermore, 
for every affine subspace $H$ of $\R^d$ with dimension $r \in \{ d-1,d-2 \}$ and for some sufficiently small $t_0 > 0$, we have
\begin{equation}
\int_{H} \sup_{\substack{\vv \in (H-\vy_0)^{\bot} \\ \| \vv \| \le t_0}} \rho (\vy + \vv) \, d\cH^{r}(\vy) < \infty,
\label{eq: DCT}
  \end{equation}
where $\vy_0$ is any fixed point in $H$ and $(H-\vy_0)^{\bot} = \{ \vv : \langle \vv, \vy-\vy_0 \rangle = 0, \forall \vy \in H \}$.  
Finally, the set $\cY$ satisfies either of the following: (i) $\cY$ is a convex polyhedral set ($\cY=\R^d$ is allowed), or (ii) $\cY$ has Lipschitz boundary with $\cH^{d-1}(\partial \cY \cap H) = 0$ for every hyperplane $H$ in $\R^d$.
\end{assumption}

The density $\rho$ need not be globally continuous on $\R^d$. Without loss of generality, we set $\rho = 0$ on $\cY^c$. Condition (\ref{eq: DCT}) holds if $\cY$ is compact and $\rho$ is bounded, in which case the left-hand side on (\ref{eq: DCT}) can be bounded by $\| \rho \|_{\infty} \cH^{r}(H \cap \cY^{t_0}) < \infty$, where $\cY^{t_0}$ is the $t_0$-blow-up of $\cY$, i.e., $\cY^{t_0} = \{ \vy : \dist (\vy,\cY) \le t_0 \}$. In particular, Assumption \ref{asp: cont} allows for the uniform distributions over the unit cube and ball.  Also, (nondegenerate) Gaussian distributions on $\R^d$ clearly satisfy Assumption \ref{asp: cont}.

For the limit theorems, we need an additional assumption on  the regularity of $R$, stated as follows.

\begin{assumption}[$L^1$-Poincar\'{e} inequality]
\label{asp: poincare}
The reference measure $R$ satisfies an $L^1$-Poincar\'{e} inequality, i.e., there exists a finite constant $\mathsf{C}_{\mathrm{P}}$ such that for $\vY \sim R$,
\[
\E\big [ |f(\vY) - \E[f(\vY)]| \big] \le \mathsf{C}_{\mathrm{P}} \E\big[\| \nabla f (\vY) \|\big],
\]
for every locally Lipschitz function $f$ on $\R^d$, where $\| \nabla f (\vy) \| \coloneqq \limsup_{\vx \to \vy} \frac{|f(\vx) - f(\vy)|}{\|\vx - \vy\|}$. 
\end{assumption}

Assumption \ref{asp: poincare} is not needed for the stability results in the next section, but needed to establish a limit distribution for $\hat{\vz}_n$, whose derivation relies on (minor extensions of) the results of \cite{kitagawa2019convergence,bansil2022quantitative}. In \cite{kitagawa2019convergence,bansil2022quantitative}, the $L^1$-Poincar\'{e} inequality is used to guarantee strict concavity of the dual objective function in nontrivial directions, which ensures the uniqueness of the dual potential vector subject to proper normalization. 
The $L^1$-Poincar\'{e} inequality is  equivalent to Cheeger's isoperimetric inequality,
\begin{equation}
\liminf_{\delta \downarrow 0} \frac{R\big( A^\delta \setminus A \big)}{\delta} \ge \mathsf{C}_{\mathrm{P}}^{-1} \min \{ R(A), 1-R(A) \}
\label{eq: cheeger}
\end{equation}
for every Borel set $A \subset \R^d$, where $A^\delta = \{ \vy : d(\vy,A) \le  \delta \}$. See \cite[Lemma 2.2]{milman2007role} and \cite{bobkov1997iso}. For regular $\rho$ and $A$, the left-hand side on (\ref{eq: cheeger}), called the Minkowski content, agrees with (our definition of) the surface measure $R^+(A)$. Our reference density $\rho$ may have discontinuities, but our proof only requires Cheeger's inequality to hold for Laguerre cells $C_i (\vz)$ with positive $R$-measure, for which the Minkowski content agrees with $R^+(C_i(\vz))$ under Assumption \ref{asp: cont} from our proof of Theorem \ref{thm: Hadamard_diff} below. 

The $L^1$-Poincar\'{e} inequality is known to hold for  every \textit{log-concave distribution}, i.e., a distribution $Q$ of the form $dQ= e^{-\psi} \,d\vy$ for some convex function $\psi: \R^d \to (-\infty,\infty]$; see \cite{kannan1995isoperimetric,bobkov1999isoperimetric}.
The uniform distributions over the unit cube and ball and nondegenerate Gaussian distributions all satisfy Assumption \ref{asp: poincare} as they are log-concave.

When we consider inference for linear functionals of the form $\langle \bm{\varphi}, \vT^* \rangle_{L^2(R)}$, we make the following assumption on the Borel vector field $\bm{\varphi}$. Denote by $\mathfrak{D}_{\bm{\varphi}}$ the set of discontinuities of $\bm{\varphi}$ on $\cY$. 

\begin{assumption}[Regularity of $\bm{\varphi}$]
\label{asp: text function}
\textcolor{black}{The Borel vector field $\bm{\varphi}: \R^d \to \R^d$ is $R$-integrable (i.e., each coordinate of $\bm{\varphi}$ is $R$-integrable) and satisfies that $\cH^{d-1}(\mathfrak{D}_{\bm{\varphi}} \cap H) = 0$ for every hyperplane $H$ in $\R^d$. Furthermore, Condition (\ref{eq: DCT}) with $\rho$ replaced by $\| \bm{\varphi} \| \cdot \rho$ holds for every affine subspace $H$ of $\R^d$ with dimension $r \in \{ d-1,d-2\}$ and for some sufficiently small $t_0 > 0$.}
\end{assumption}

The assumption essentially guarantees that the discontinuities of $\bm{\varphi}$ have $\calH^{d-1}$-measure zero around the boundaries of the Laguerre cells, which is needed for its Fr\'{e}chet differentiability. Assumption \ref{asp: text function} will be discussed further after Theorem \ref{thm: Hadamard_diff} below. We allow the test function $\bm{\varphi}$ to have discontinuities to cover maximum tail correlations, where $\bm{\varphi} (\vy)= \alpha^{-1} \vy \mathbbm{1}_{B_r^c}(\vy)$ for some $\alpha \in (0,1)$ and $r > 0$; see Section \ref{sec: applications} ahead. 
\begin{remark}[Unbounded reference densities]
Assumptions \ref{asp: cont} and \ref{asp: poincare} allow for unbounded reference densities. For example, assume $d > 2$ and consider 
\[
\rho (\vy) = \frac{c_\alpha}{\|\vy\|^{\alpha}} \mathbbm{1}_{B_1 \setminus \{ \vzero \}} (\vy), \ 0 < \alpha < d-2, 
\]
where $c_{\alpha}$ is the normalizing constant. Note that $\vY \sim \rho$ corresponds to independently generating $\frac{\vY}{\| \vY \|}$ from the uniform distribution over the unit sphere and $\| \vY \|$ from the $Beta~(d-\alpha,1)$ distribution. In this case, we may take $\calY = B_1$, for which Assumption \ref{asp: cont} (ii) is verified, and $\rho$ is continuous on $\calY$ except at $\vy = \vzero$. Condition (\ref{eq: DCT}) trivially holds if the affine space $H$ does not contain the origin since in that case $\rho$ is bounded on $H$. Suppose that $H$ contains the origin and take $\vy_0 = \vzero$. For every $\vv \in H^{\bot}$, $\| \vy + \vv \|^{-\alpha} = (\| \vy \|^2+\|\vv\|^2)^{-\alpha/2} \le \| \vy \|^{-\alpha}$, 
so that for every $t_0 > 0$, 
\[
\sup_{\vv \in H^{\bot}, \| \vv \| \le t_0} \rho (\vy + \vv) \le \frac{c_\alpha}{\| \vy \|^{\alpha}} \mathbbm{1}_{B_{1+t_0}} (\vy).
\]
Since $\int_{H \cap B_{1+t_0}} \| \vy \|^{-\alpha} \, \calH^{r}(d\vy) < \infty$ whenever $\alpha < r$, Condition (\ref{eq: DCT}) is verified. The $L^1$-Poincar\'{e} inequality (Assumption \ref{asp: poincare}) follows from a small adaptation to the proof of Proposition A.1 in \cite{kitagawa2019convergence}, upon observing that the density $\bar{\rho} (r) \propto r^{d-1-\alpha}$ on $[0,1]$ is log-concave and hence satisfies the $L^1$-Poincar\'{e} inequality. However, Assumption \ref{asp: cont} excludes the spherical uniform distribution that corresponds to $\alpha = d-1$. Indeed, Theorem \ref{thm: Hadamard_diff} below does not hold in general for the spherical uniform distribution; see Remark \ref{rem: spherical uniform} for further discussion. 
\end{remark}

\subsection{Stability results}
For $\vz \in \R^N$, define a map $\vT_{\vz}$ with values in $\cX$ by
\[
\vT_{\vz}(\vy) = \argmin_{\vx_i: 1 \le i \le N} \left ( \frac{1}{2} \|\vy-\vx_i\|^2 - z_{i} \right ),
\]
which is well-defined $R$-a.e. Note that, since $\bm{T}_{\vz}$ agrees with the gradient of the convex function $\max_{1 \le i \le N} \big \{ \langle \cdot,\vx_i \rangle - \big (\frac{\|\vx_i\|^2}{2} - z_i \big) \big \}$ up to Lebesgue negligible sets, by the Knott-Smith theorem \cite{knott1984optimal}, $\bm{T}_{\vz}$ is the OT map transporting $R$ onto $\sum_{i=1}^N R_i\big (C_i(\vz)\big) \delta_{\vx_i}$. 
In this section, we establish (directional) differentiability of the following functions:
\[
\begin{split}
\delta_s(\vz_1,\vz_2) &\coloneqq \|\vT_{\vz_1} - \vT_{\vz_2}\|_{L^s(R)}^s,\quad \vz_1,\vz_2 \in \R^N,\\
\gamma_{\bm{\varphi}} (\vz) &\coloneqq \langle \bm{\varphi}, \vT_{\vz} \rangle_{L^2(R)} = \int \langle \bm{\varphi},\vT_{\vz} \rangle \, dR, \quad \vz \in \R^N.
\end{split}
\]
Observe that $\| \hat{\vT}_n - \vT^* \|_{L^s(R)}^s = \delta_s(\hat{\vz}_n,\vz^*)$ and $\langle \bm{\varphi},\hat{\vT}_n - \vT^* \rangle_{L^2(R)}= \gamma_{\bm{\varphi}}(\hat{\vz}_n) - \gamma_{\bm{\varphi}}(\vz^*)$.
Combined with a limit distribution result for $\hat{\vz}_n$ (which will be discussed in the next section), the limit distributions for these functionals follow via the extended delta method.
It turns out that the $\delta_s$ functional is not (Fr\'{e}chet) differentiable at $(\vz^*,\vz^*)$, but \textit{Hadamard directionally differentiable}, which is enough to invoke the extended delta method.
Note that to find a limit distribution for $\| \hat{\vT}_n-\vT^*\|_{L^s(R)}^s$, we only need to derive a Hadamard directional derivative of a simpler function $\vz \mapsto \| \vT_{\vz} - \vT_{\vz^*} \|_{L^s(R)}^s$ at $\vz^*$.  However, to study the bootstrap  for the $L^s$-functional, we need to analyze the two-variable mapping $(\vz_1,\vz_2) \mapsto \| \vT_{\vz_1} - \vT_{\vz_2} \|_{L^s(R)}^s$.

\begin{figure}
    \centering \includegraphics[width=0.54\textwidth]{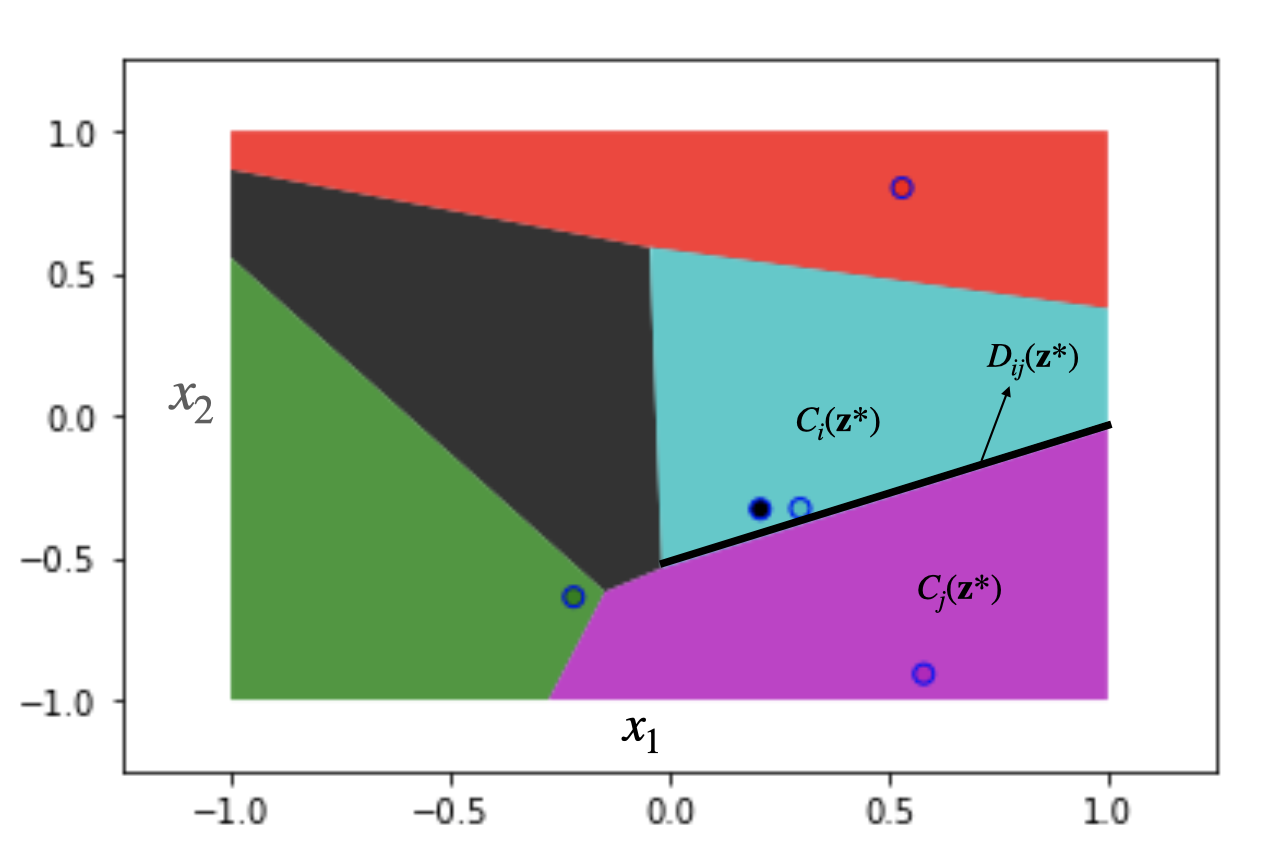}
    \caption{\small Laguerre cells $C_i(\vz^*)$ and $C_j(\vz^*)$ and their intersection $D_{ij}(\vz^*)$ in $d = 2$. Here, $R = \mathrm{Unif} ([-1,1]^2)$ and $P = \mathrm{Unif}(\{\vx_1,\dots,\vx_5\})$ for 5 points $\vx_1,\dots,\vx_5$ chosen randomly on $[-1,1]^2$.}
    \label{fig:laguerre_1}
\end{figure}

To state our stability results, we need additional notations. 
For a subset $D$ of a hyperplane in $\R^d$, the $R$-surface measure of $D$ is defined by
\[
R^{+}(D) = \int_D \rho \,d\cH^{d-1}.
\]
Recall the Laguerre cells $\{ C_i(\vz) \}_{i=1}^N$ defined in the previous section and set $b_{ij}(\vz) = (\|\vx_i\|^2-\|\vx_j\|^2)/2-z_i+z_j$. When evaluated at $\vz = \vz^*$  (an optimal solution to (\ref{eq: semi_dual}) for $(R,P)$), we often omit the dependence on $\vz^*$, i.e., we write $b_{ij} = b_{ij}(\vz^*)$ and $C_i = C_i(\vz^*)$ (here $\vz^*$ is any optimal solution to the dual problem \eqref{eq: semi_dual}).
Denote the boundary between $C_i$ and $C_j$ by
\[
D_{ij} \coloneqq C_i \cap C_j = C_i  \cap \big\{\vy: \langle \vx_i-\vx_j,\vy \rangle = b_{ij}\big\} = C_j  \cap \big\{\vy: \langle \vx_i-\vx_j,\vy \rangle = b_{ij}\big\}.
\]
Observe that $b_{ij}$ is anti-symmetric in $(i,j)$, $b_{ij}  = -b_{ji}$, and that it satisfies $b_{ij} + b_{jk} = b_{ik}$, while $D_{ij}$ is symmetric in $(i,j)$, $D_{ij} = D_{ji}$. The collection $\{ D_{ij} \}_{j \ne i}$ forms the boundary of $C_i$. Figure \ref{fig:laguerre_1} displays Laguerre cells and their boundaries in $d=2$.

We are now ready to state the main result of this section. 
\begin{theorem}[Stability]
\label{thm: Hadamard_diff}
Under  Assumption \ref{asp: cont}, the following hold. 
\begin{enumerate}
    \item[(i)] For arbitrary $s \in [1,\infty)$, the map $\R^{2N} \ni (\vz_1,\vz_2) \mapsto \delta_s(\vz_1, \vz_2) = \|\vT_{\vz_1} - \vT_{\vz_2}\|_{L^s(R)}^s$ is Hadamard directionally differentiable at $(\vz^*, \vz^*)$  with derivative 
\begin{equation}
\label{eq: h_deriv}
[ \delta_s ]'_{(\vz^*, \vz^*)}(\vh_1, \vh_2) = \sum_{1 \leq i < j \leq N} \|\vx_i - \vx_j\|^{s-1} \, R^{+}(D_{ij})\, |h_{2,j} - h_{2,i} - h_{1,j} + h_{1,i}| 
\end{equation}
for $\vh_{k} = (h_{k,1},\dots,h_{k,N})^{\intercal} \in \R^N$ and $k=1,2$. 
\item[(ii)] Suppose in addition that Assumption \ref{asp: text function} holds. Then, the map $\R^N \ni \vz \mapsto \gamma_{\bm{\varphi}}(\vz) = \langle \bm{\varphi},\vT_{\vz} \rangle_{L^2(R)}$ is (Fr\'{e}chet) differentiable at $\vz^*$ with derivative
\begin{equation}
[\gamma_{\bm{\varphi}}]_{\vz^*}'(\vh) = \sum_{1\leq i < j \leq N} \frac{h_i - h_j}{\|\vx_i - \vx_j\|}\int_{D_{ij}}  \langle \vx_i - \vx_j,  \bm{\varphi}(\vy) \rangle  \rho (\vy) \, d\cH^{d-1}(\vy)
\label{eq: h-deriv2}
\end{equation}
for $\vh = (h_1,\dots,h_N)^{\intercal}\in \R^N$.
\end{enumerate}
\end{theorem}

The derivative of $\delta_s$ in (\ref{eq: h_deriv}) is nonlinear in $(\vh_1,\vh_2)$, so $\delta_s$ is only directionally differentiable. On the other hand, $\gamma_{\bm{\varphi}}$ is Fr\'{e}chet differentiable and the derivative in (\ref{eq: h-deriv2}) is linear in $\vh$. The derivative $[\gamma_{\bm{\varphi}}]_{\vz^*}'$ may vanish depending on the choice of $\bm{\varphi}$. For example, if $\bm{\varphi}$ vanishes on the boundaries of the Laguerre cells $\{ C_i \}_{i=1}^N$, then $[\gamma_{\bm{\varphi}}]_{\vz^*}' \equiv 0$. However, the derivative may be nonvanishing if the support of $\bm{\varphi}$ intersects the boundaries of the Laguerre cells. 

The proof of Theorem \ref{thm: Hadamard_diff} is lengthy and constitutes the majority of Section \ref{sec: proofs}. For Part (i), we observe that 
\[
\delta_s(\vz^* + t\vh_1,\vz^*+t\vh_2) = \sum_{1 \le i \ne j \le N} \|\vx_i-\vx_j\|^s R\big( C_i(\vz^*+t\vh_1) \cap C_j(\vz^*+t\vh_2) \big).
\]
The proof proceeds by carefully analyzing the facial structures of polyhedral sets of the form $C_i(\vz^*+t\vh_1) \cap C_j(\vz^*+t\vh_2)$, which gives rise to a Gateaux directional derivative. To lift the Gateaux differentiability to the Hadamard one, we establish local Lipshitz continuity of $\delta_s$ (cf. \cite{shapiro1990concepts}). The proof of Part (ii) is similar upon observing that
\[
\gamma_{\bm{\varphi}}(\vz^* + t\vh) - \gamma_{\bm{\varphi}}(\vz^*) = \sum_{1 \le i \neq j \le N} \int_{C_i(\vz^*) \cap C_j(\vz^* + t\vh)} \langle \vx_j - \vx_i, \bm{\varphi}(\vy) \rangle \ dR(\vy).
\]
Assumption \ref{asp: text function} is essential to guarantee that the discontinuities of $\bm{\varphi}$ do not affect the local behavior of $\bm{\varphi}$ around each $D_{ij} = C_i \cap C_j$ with respect to $\cH^{d-1}$. 
Indeed, if the set of discontinuities of $\bm{\varphi}$ has positive $\calH^{d-1}$-measure, then the function $\gamma_{\bm{\varphi}}$ is in general not Fr\'{e}chet differentiable. To see this, consider $\bm{\varphi} = \vT^*$, for which $\gamma_{\bm{\varphi}} (\vz) - \gamma_{\bm{\varphi}}(\vz^*)= -\| \vT_{\vz} - \vT^* \|_{L^2(R)}^2/2 + (\| \vT_{\vz}\|^2_{L^2(R)} -  \| \vT^*\|^2_{L^2(R)})/2$. The function $\vz \mapsto \| \vT_{\vz}  \|_{L^2(R)}^2 = \sum_{i=1}^N R(C_i(\vz)) \| \vx_i \|^2$ is Fr\'{e}chet differentiable at $\vz^*$ from the proof of Theorem \ref{thm: Hadamard_diff}, while $\vz \mapsto \| \vT_{\vz} - \vT^* \|_{L^2(R)}^2$ is not by Theorem \ref{thm: Hadamard_diff} (i), so $\gamma_{\bm{\varphi}}$ is not Fr\'{e}chet differentiable at $\vz^*$.

\begin{remark}[Spherical uniform distribution]
\label{rem: spherical uniform}
One of the common reference measures used in the literature is the spherical uniform distribution \cite{hallin2021distribution},
\[
\rho (\vy) = \frac{1}{\|\vy\|^{d-1}\omega_{d-1}} \mathbbm{1}_{B_1 \setminus \{ \vzero \}}(\vy),
\]
where $\omega_{d-1}$ is the volume of the unit sphere in $\R^d$. 
However, the results of Theorem~\ref{thm: Hadamard_diff} do not hold for the spherical uniform distribution. Consider the case where $\vx_1 = (1,0)^{\intercal}, \vx_2 = (-1,0)^{\intercal}$, and $P(\{ \vx_1 \}) = P(\{ \vx_2 \}) = 1/2$, for which $C_1(\vz^*)=\{\vy : y_1 \ge 0 \}$ and $C_2(\vz^*) = \{ \vy : y_1 \le 0 \}$. 
Since $R^+(D_{12}) = R^+ (\{ \vy : y_1 = 0 \}) = \omega_1^{-1}\int_{-1}^1 |y_2|^{-1} dy_2 = \infty$, the derivative formulas in Theorem \ref{thm: Hadamard_diff} do not hold in general. 

Nevertheless, an inspection of the proof shows that the results of Theorem \ref{thm: Hadamard_diff} hold for the spherical uniform distribution if the hyperplane containing each $D_{ij}$ does not intersect the origin (the $L^1$-Poincar\'e inequality follows by the argument in the proof of Proposition A.1 in \cite{kitagawa2019convergence}); however, the said assumption seems difficult to verify in practice. 
\end{remark}

\subsection{Limit theorems}
\label{sec: limit}
Recall that $\vX_1,\dots,\vX_n$ is an i.i.d. sample from $P$ with empirical distribution $\hat{P}_n = n^{-1}\sum_{i=1}^n \delta_{\vX_i}$. The frequency vector for $\hat{P}_n$ is $\hat{\vp}_n = (\hat{p}_{n,1},\dots,\hat{p}_{n,N})^{\intercal} = \big(\hat{P}_n(\{\vx_1\}),\dots,\hat{P}_n(\{\vx_N \})\big)^{\intercal}$. Before deriving the limit theorems for the (functionals of interest of the) semidiscrete OT map, we briefly rederive the CLT for the empirical dual potential vector. The latter result was originally established in \cite{del2022central} for compact  $\cY$ by combining the results from \cite{kitagawa2019convergence} and the $M$-estimation (or $Z$-estimation) machinery. We present a slightly different and  direct derivation that relies on the delta method---an approach that lends better for the subsequent bootstrap consistency analysis. Also, in contrast to \cite{del2022central}, our derivation allows for unbounded $\cY$, which requires additional work. Asymptotic efficiency of the empirical estimator $\langle \bm{\varphi},\hat{\vT}_n \rangle_{L^2(R)}$ follows as a direct byproduct of our approach.

We first need to guarantee the uniqueness (identification) of an optimal solution to the dual problem (\ref{eq: semi_dual}) subject to proper normalization. Obviously, if $\vz$ is optimal for (\ref{eq: semi_dual}), then $\vz + \mathsf{c}\vone$ for any $\mathsf{c} \in \R$ is optimal as well, so we make the normalization that $\langle \vz,\vone \rangle = 0$. Proposition \ref{lem: kitagawa} below shows the uniqueness of the dual potential vector under this normalization as well as its Hadamard differentiability. An application of the delta method then yields a limit distribution for $\hat{\vz}_n$.

To state the proposition, we need additional notations. Let 
\[
\cQ= \Big\{ \bm{q} \in\RR^{N}_{\geq 0} : 
\langle \bm{q},\vone \rangle = 1 \Big \}
\]
denote the set of probability simplex vectors.
We may identify $\bm{q} \in \cQ$ with the probability measure $\sum_{i=1}^{N}q_i \delta_{\vx_i}$. 
 Also, define 
\[
\cQ_{+} = \cQ \cap \R_{>0}^N.
\]
Set the objective function in the dual problem as
\[
\Phi(\vz,\bm{q}) = \langle \vz,\bm{q} \rangle+ \int \min_{1 \leq i\leq N} \left ( \frac{1}{2} \| \vy-\vx_i \|^2 - z_i\right )\,dR(\vy)
\]
for $\vz \in \R^N$ and $\bm{q} \in \cQ$. We are ready to state Proposition \ref{lem: kitagawa}. The proof essentially relies on the results from \cite{kitagawa2019convergence, bansil2022quantitative}, but requires additional work to cover the case where the set $\cY$ is unbounded and $\rho$ has discontinuities on $\cY$. Recall that we have assumed that $\vp \in \cQ_{+}$.

\begin{proposition}[Uniqueness and differentiability of dual potential vector]
\label{lem: kitagawa}
\textcolor{black}{
Under Assumptions \ref{asp: cont} and \ref{asp: poincare}, 
for every $\bm{q} \in \cQ_+$, the dual problem
\begin{equation}
\max_{\vz \in \langle \vone \rangle^{\bot}} \Phi(\vz,\bm{q})
\label{eq: dual}
\end{equation}
admits the unique optimal solution $\vz^*(\bm{q})$, where $\langle \vone \rangle^{\bot} = \{ \vz \in  \R^N : \langle \vz,\vone \rangle = 0 \}$. Furthermore, the mapping $\cQ_+ \ni \bm{q} \mapsto \vz^*(\bm{q})$ is Hadamard differentiable at $\vp$ tangentially to $\langle \vone \rangle^{\bot}$, i.e., there exists an $N \times N$ matrix $\bm{B}$ such that for every $t_n \downarrow 0$ and $\langle \vone \rangle^\bot \ni \vh_n \to \vh$, one has $\vz^* (\vp + t_n \vh_n) = \vz^*(\vp) + t_n\bm{B}\vh + o(t_n)$.}
\end{proposition}

\begin{remark}
\label{rem: rank}

(i) Since $\Phi (\vz+\mathsf{c} \vone,\bm{q}) = \Phi(\vz,\bm{q})$ for any $\mathsf{c} \in \R$, $\vz^*(\bm{q})$ maximizes $\Phi(\cdot,\bm{q})$ over the entire $\R^{N}$ space. \textcolor{black}{ (ii) The uniqueness of the dual potential vector is known to hold under a (somewhat) weaker assumption. Indeed, the uniqueness holds provided that the reference measure $R$ (has a finite second moment and) is absolutely continuous and its support agrees with the closure of a connected open set; see Theorem 7.18 in \cite{santambrogio15} and its proof. See also Remark 3.1 in \cite{bansil2022quantitative}.
(iii) 
The matrix $\bm{B}$ is not unique (one may add any matrix orthogonal to $\vone$), but its restriction to $\langle \vone \rangle^\bot$ is unique. 
Let $\bm{L} = (\ell_{ij})_{1 \le i,j \le N}$ be the symmetric matrix with $\ell_{ij} = - \|\vx_i-\vx_j\|^{-1} R^+(D_{ij})$ for $i \ne j$ and $\ell_{ii} = \sum_{k \ne i} \|\vx_i-\vx_k\|^{-1} R^+(D_{ik})$. The proof of Proposition \ref{lem: kitagawa} (or a simple adaptation of the results from \cite{kitagawa2019convergence,bansil2022quantitative}) shows that $\bm{L}$ is positive semidefinite with eigenvalues $0=\lambda_1 < \lambda_2 \le \cdots \le \lambda_{N}$ and eigenvector $\vone$ corresponding to eigenvalue $\lambda_1  = 0$. Denoting its spectral expansion by $\bm{L} = \sum_{i=2}^N \lambda_i \bm{P}_i$, $\bm{L}$ has formal inverse $\bm{L}^{-1} = \sum_{i=2}^N \lambda_i^{-1} \bm{P}_i$. Then, the proof of Proposition \ref{lem: kitagawa} yields that the restriction of $\bm{B}$ to $\langle \vone \rangle^\bot$ agrees with $\bm{L}^{-1}$ and is isomorphic onto $\langle \vone \rangle^\bot$. 
}
\end{remark}

Having  Proposition \ref{lem: kitagawa}, we now derive the CLT for $\hat{\vz}_n$  via the delta method. In what follows, we normalize $\vz^*$ and $\hat{\vz}_n$ in such a way that they are orthogonal to $\vone$.  Observe that $\vz^* = \vz^*(\vp)$ and $\hat{\vz}_n = \vz^*(\hat{\vp}_n)$ when $\hat{\vp}_n \in \cQ_+$, which holds with probability approaching one (indeed, $\hat{\vp}_n \in \cQ_+$ holds with probability at least $1-e^{-\mathsf{c}n}$ for some constant $\mathsf{c} > 0$). 
Since $\sqrt{n}(\hat{\vp}_n-\vp) \stackrel{d}{\to} \cN(\vzero,\bm{A})$ with $\bm{A} = \mathrm{diag} \{ p_1,\dots,p_N \} - \vp \vp^\intercal$
by the multivariate CLT, an application of the delta method yields 
\begin{equation}
\sqrt{n}(\hat{\vz}_n-\vz^*) \stackrel{d}{\to} \cN(\vzero,\bm{B}\bm{A}\bm{B}^\intercal). \label{eq: CLT}
\end{equation}
Since $\langle \hat{\vz}_n-\vz^*, \vone \rangle = 0$, the Gaussian distribution $\cN(\vzero,\bm{B}\bm{A}\bm{B}^\intercal)$ is singular. Note that, the matrix $\bm{A}$ (the covariance matrix of a multinomial vector) has rank $N-1$ (cf. \cite{tanabe1992exact}) and its null space agrees with $\langle \vone \rangle^\bot$, so in view of Remark \ref{rem: rank}, the matrix $\bm{B}\bm{A}\bm{B}^{\intercal}$ is isomorphic from $\langle \vone \rangle^{\bot}$ onto $\langle \vone \rangle^{\bot}$.

Now, combining \eqref{eq: CLT} and the stability results in Theorem \ref{thm: Hadamard_diff}, together with the extended delta method (cf. Lemma \ref{lem: functional delta method}), we obtain the following theorem that provides distributional limits and moment convergence  of $\hat{\vT}_n = \vT_{\hat{\vz}_n}$ to $\vT^* = \vT_{\vz^*}$, under the functionals of interest.

\begin{theorem}[Limit distributions]
\label{thm: limit}
Suppose Assumptions \ref{asp: cont} and \ref{asp: poincare} hold. Let $\vW = (W_1,\dots,W_N)^{\intercal} \sim \cN(\vzero,\bm{B}\bm{A}\bm{B}^\intercal)$. Then the following hold.
\begin{enumerate}
    \item[(i)] For arbitrary $s \in [1,\infty)$, we have 
    \begin{equation}
    \sqrt{n} \| \hat{\vT}_n-\vT^* \|_{L^s(R)}^s \stackrel{d}{\to}  \sum_{1 \leq i < j \leq N} \|\vx_i - \vx_j\|^{s-1} \, R^{+}(D_{ij})\, |W_i-W_j|.
    \label{eq: limit law}
    \end{equation}
    The limit law is absolutely continuous and its density is positive almost everywhere on $[0,\infty)$. Finally, for every continuous function $\Upsilon$ on $\R_{\ge 0}$ with polynomial growth, we have, for $V$ denoting the limit variable in (\ref{eq: limit law}), 
    \[
    \lim_{n \to \infty} \E\left [ \Upsilon\left ( \sqrt{n} \| \hat{\vT}_n-\vT^* \|_{L^s(R)}^s \right) \right] = \E[\Upsilon (V)]. 
    \]
    \item[(ii)] Suppose in addition that Assumption \ref{asp: text function} holds. Then
    \[
    \sqrt{n} \langle \bm{\varphi},\hat{\vT}_n - \vT^* \rangle_{L^2(R)} \stackrel{d}{\to} \cN(0,\sigma_{\bm{\varphi}}^2),
    \]
    where $\sigma_{\bm{\varphi}}^2$ is the variance of the following random variable
    \[
    \sum_{1\leq i < j \leq N} \frac{W_i - W_j}{\|\vx_i - \vx_j\|}\int_{D_{ij}}  \langle \vx_i - \vx_j,  \bm{\varphi}(\vy) \rangle  \rho (\vy) \, d\cH^{d-1}(\vy).
    \]
    Furthermore, for every continuous function $\Upsilon$ on $\R$ with polynomial growth, we have
    \[
\lim_{n \to \infty} \E\left [ \Upsilon \left (  \sqrt{n} \langle \bm{\varphi},\hat{\vT}_n - \vT^* \rangle_{L^2(R)}\right ) \right] = \E\left [ \Upsilon \left (\cN(0,\sigma_{\bm{\varphi}}^2) \right) \right ].
    \]
\end{enumerate}
\end{theorem}

\begin{remark}[Asymptotic efficiency of $\langle \bm{\varphi},\hat{\vT}_n \rangle_{L^2(R)}$]
Consider Part (ii) and assume that $\sigma_{\bm{\varphi}}^2 > 0$. 
Observe that $\langle \bm{\varphi},\vT^* \rangle_{L^2(R)} = \gamma_{\bm{\varphi}} (\vz^*(\vp))$ and the function $\bm{q} \mapsto \gamma_{\bm{\varphi}} (\vz^*(\bm{q}))$ is Hadamard differentiable at $\vp$.
Since $\hat{\vp}_n$ is the maximum likelihood estimator for $\vp$, the plug-in estimator $\langle \bm{\varphi},\hat{\vT}_n \rangle_{L^2(R)} = \gamma_{\bm{\varphi}} (\vz^*(\hat{\vp}_n))$ is asymptotically efficient in the Haj\'{e}k-Le Cam sense. See Chapter 8 in \cite{vanderVaart1998asymptotic} for details. \textcolor{black}{We will show in Section \ref{sec: dual holder} that, when viewed as elements of the dual of a H\"{o}lder space, the empirical OT map is asymptotically efficient. }
\end{remark}

\begin{remark}[Simple example]
\label{lem: simple case}
The following setting ties the linear functional case to the 2-Wasserstein distance. Assume that $\calY$ is bounded and consider $\bm{\varphi}(\vy) = \vy\mathbbm{1}_{\cY}(\vy)$, for which we have $\langle \bm{\varphi},\vT^* \rangle_{L^2(R)} = \int \langle \vy,\vx \rangle \, d\pi^*(\vy,\vx)$ for the optimal solution $\pi^*$ to the Kantorovich problem (\ref{eq: Kantorovich}). Denoting by $\mathsf{W}_2^2(R,P)$ the squared $2$-Wasserstein distance, i.e., twice the optimal value in (\ref{eq: Kantorovich}), we have $\mathsf{W}^2_2(R,P) = \int \|\vy\|^2 \, dR(\vy) + \int \|\vx\|^2 \, dP(\vx) - 2\langle \bm{\varphi},\vT^* \rangle_{L^2(R)}$.
Then, under Assumptions \ref{asp: cont} and \ref{asp: poincare}, we have $\sigma_{\bm{\varphi}}^2 = \Var_P\big(\|\cdot\|^2/2-\psi\big)$, where $\psi: \cX \to \R$ is a function defined by $\psi(\vx_i) = z_i^*$ for $i \in \{1,\dots,N\}$; see Section \ref{sec: proof limit} for a proof. 
Hence, in this case, we have $\sqrt{n}\langle \bm{\varphi},\hat{\vT}_n - \vT^* \rangle_{L^2(R)} \stackrel{d}{\to} N\big(0,\Var_P\big(\|\cdot\|^2/2-\psi\big)\big)$. This is consistent with (the implication of) Theorem 4.3 in \cite{del2019clt}.
\end{remark}

The second claim of Theorem \ref{thm: limit} (i) shows that the limit law in (\ref{eq: limit law}) is nondegenerate. This follows from the next proposition, which also derives an explicit form of density. The proof relies on Theorem 11.1 in \cite{Davydov1998} and the coarea formula.  

\begin{proposition}[Density formula for non-Gaussian limit law]
\label{prop: density}
Let $B = (\beta_{ij})_{1 \le i,j \le N}$ be a nonnegative symmetric matrix with zero diagonal entries such that $\sum_{j=1}^N \beta_{ij} 
> 0$ for every $i \in \{ 1,\dots, N \}$, and $\vW_{-N} = (W_1,\dots,W_{N-1})^{\intercal} \sim \mathcal{N}(\vzero,\bm{\Sigma})$ with $\bm{\Sigma}$ being nonsingular. Denote by $\phi_{\bm{\Sigma}}$ the density of $\vW_{-N}$ and 
set $W_N = -\sum_{i=1}^{N-1}W_i$. Consider the random variable 
$
V= g(\vW_{-N}) = \sum_{1 \le i < j \le N}\beta_{ij} |W_i-W_j|. 
$
Define 
\[
\mathsf{C}(B) =  \left \| \Big ( -\textstyle \sum_{j=1}^{i-1} \beta_{ij} + \sum_{j=i+1}^N  \beta_{ij} + \sum_{j=1}^{N-1}\beta_{Nj} \Big)_{i=1}^N \right \|.
\]
For every permutation $\sigma$ of $\{ 1,\dots, N \}$, set $B_{\sigma} = (\beta_{\sigma(i),\sigma(j)})_{1 \le i,j \le N}$ and $E_\sigma = \{ \vw_{-N} : w_{\sigma(1)} > \cdots > w_{\sigma(N)} \}$ with $w_N = -\sum_{i=1}^{N-1}w_i$.  Then, the following hold.
\begin{enumerate}
    \item[(i)] The law of $V$ is absolutely continuous and its density is positive almost everywhere on $[0,\infty)$.
    \item[(ii)] A version of the density of $V$ is given by
\[
f_{V}(v) = \sum_{\sigma} \frac{1}{\mathsf{C}(B_\sigma)} \int_{\{ g(\vw_{-N})=v \} \cap E_\sigma} \phi_{\bm{\Sigma}}(\vw_{-N}) \, d\calH^{N-2}(\vw_{-N}), \ v \in [0,\infty),
\]
where $\sum_{\sigma}$ is taken over all permutations $\sigma$ of $\{ 1,\dots, N\}$. 
\end{enumerate} 
\end{proposition}

The third claim of Theorem \ref{thm: limit} (i) in particular implies that the squared $L^2(R)$-risk of $\hat{\vT}_n$ can be  expanded as 
\[
\E\big [\| \hat{\vT}_n - \vT^* \|_{L^2(R)}^2 \big] = n^{-1/2}\sum_{i<j} \|\vx_i-\vx_j\| R^+(D_{ij}) \E[|W_i-W_j|] + o(n^{-1/2}),
\]
which is also a  new result in the literature. 

For the linear functional case, 
the next lemma gives a necessary and sufficient condition for the asymptotic variance $\sigma_{\bm{\varphi}}^2$ in Theorem \ref{thm: limit} (ii) to be strictly positive. Set 
\[
a_{ij}^{\bm{\varphi}} = \frac{1}{\|\vx_i - \vx_j\|}\int_{D_{ij}}  \langle \vx_i - \vx_j,  \bm{\varphi}(\vy) \rangle  \rho (\vy) \, d\cH^{d-1}(\vy), \quad i \ne j
\]
and $\tilde{a}_{i}^{\bm{\varphi}} = \sum_{j \ne i, j=1}^N a_{ij}^{\bm{\varphi}}$ for $i=1,\dots,N$

\begin{lemma}[Nondegeneracy of $\sigma_{\bm{\varphi}}^2$]
\label{lem: nondegeneracy}
Suppose that Assumptions \ref{asp: cont}--\ref{asp: text function} hold. Then, the asymptotic variance $\sigma_{\bm{\varphi}}^2$ in Theorem \ref{thm: limit} (ii) is zero if and only if  $\tilde{a}_{1}^{\bm{\varphi}} = \cdots =\tilde{a}_{N}^{\bm{\varphi}}$.
\end{lemma}

The limit distributions in Theorem \ref{thm: limit} depend on the population distribution $P$ in a complicated way, and the analytical estimation is nontrivial. The bootstrap offers an appealing alternative route for statistical inference. Our next result establishes consistency of 
the nonparametric bootstrap  for estimating the distributional limits.

Let $\vX_1^B,\dots,\vX_n^B$ be an i.i.d. sample from $\hat{P}_n$ conditional on $\vX_1,\dots,\vX_n$ and  $\hat{P}_n^B=n^{-1}\sum_{i=1}^n \delta_{\vX_i^B}$ denote the bootstrap empirical distribution. Let \[
\hat{\vp}_n^B = (\hat{p}_{n,1}^B,\dots,\hat{p}_{n,N}^B)^{\intercal} = \big(\hat{P}_n^B(\{\vx_1\}),\dots,\hat{P}_n^B(\{\vx_N \})\big)^{\intercal}
\]
denote the corresponding frequency vector. Let $\hat{\vz}_n^B \in \langle \vone \rangle^\bot$ be an optimal solution to the dual problem (\ref{eq: semi_dual}) with $\vp$ replaced by $\hat{\vp}_n^B$, and set $\hat{\vT}_n^B = \vT_{\hat{\vz}_n^B}$. Note that $\hat{\vz}_n^B = \vz^*(\hat{\vp}_n^B)$ when $\hat{\vp}_n^B \in \cQ_+$, which holds with probability approaching one. 

For a sequence of (univariate) bootstrap statistics $S_n^B$ (i.e., functions of $\vX_1,\dots,\vX_n$ and $\vX_1^B,\dots,\vX_n^B$) and a (nonrandom) distribution $\nu$ on $\R$, we say that \textit{the conditional law of $S_n^B$ given the sample converges weakly to $\nu$  in probability}  if  
\[
\sup_{g \in \mathsf{BL}_1(\R)} \Big |\E\left [ g(S_n^B) \big| \vX_1,\dots, \vX_n \right] - \E_{S \sim \nu}[g(S)] \Big | \to 0
\]
in probability, where $\mathsf{BL}_1(\R)$ is the set of $1$-Lipschitz functions $g: \R \to [-1,1]$; cf. Chapter 3.6 in \cite{van1996weak} and Chapter  23 in \cite{vanderVaart1998asymptotic}.

We are now ready to state the bootstrap consistency results. 

\begin{theorem}[Bootstrap consistency]
\label{prop: bootstrap consistency}
Suppose Assumptions \ref{asp: cont} and \ref{asp: poincare} hold. Then the following hold.
\begin{enumerate}
    \item[(i)] For arbitrary $s \in [1,\infty)$, the conditional law of $\sqrt{n}\| \hat{\vT}_n^B - \hat{\vT}_n\|_{L^s(R)}^s$ given the sample converges weakly to the limit law in (\ref{eq: limit law}) in probability. 
    \item[(ii)] Suppose in addition that Assumption \ref{asp: text function} holds. Then the conditional law of $\sqrt{n} \langle \bm{\varphi},\hat{\vT}_n^B - \hat{\vT}_n \rangle_{L^2(R)}$ given the sample converges weakly to $\cN(0,\sigma_{\bm{\varphi}}^2)$, where $\sigma_{\bm{\varphi}}^2$ is given in Theorem \ref{thm: limit} (ii). Furthermore, for $\hat{\sigma}_{n}^2 = n \E\big[\langle \bm{\varphi},\hat{\vT}_n^B - \hat{\vT}_n \rangle_{L^2(R)}^2 \mid \vX_1,\dots,\vX_n \big]$, we have $\hat{\sigma}_{n}^2 \to \sigma_{\bm{\varphi}}^2$ in probability. 
\end{enumerate} 
\end{theorem}

The proof first establishes a conditional CLT for $\sqrt{n}(\hat{\vz}_n^B-\hat{\vz}_n)$, which follows from the delta method for the bootstrap. Given this, the first claim of Part (ii) follows from another application of the delta method for the bootstrap, since the mapping $\vz \mapsto \gamma_{\bm{\varphi}}(\vz)$ is (Fr\'{e}chet) differentiable at $\vz^*$. The second claim of Part (ii), which establishes consistency of the bootstrap variance estimator, follows by verifying ``conditional'' uniform integrability of $n\langle \bm{\varphi},\hat{\vT}_n^B - \hat{\vT}_n \rangle_{L^2(R)}^2$; cf. Lemma 2.1 in \cite{kato2011note}.

Part (i) might seem surprising as the corresponding mapping $(\vz_1,\vz_2) \mapsto \delta_s(\vz_1,\vz_2)$ is only directionally differentiable with a nonlinear derivative. 
In fact, \cite{dumbgen1993nondifferentiable} and \cite{fang2019} show that the bootstrap fails to be consistent for functionals with nonlinear Hadamard derivatives. However, their results do not collide with Part (i). The results of \cite{dumbgen1993nondifferentiable} and \cite{fang2019} applied to our setting show that the conditional law of $\sqrt{n}\big(\| \hat{\vT}_n^B - \vT^* \|_{L^s(R)}^s - \| \hat{\vT}_n - \vT^* \|_{L^s(R)}^s\big) = \sqrt{n}\big(\delta_s(\hat{\vz}_n^B,\vz^*) - \delta_s(\hat{\vz}_n,\vz^*)\big)$ fails to be consistent for estimating the limit law in (\ref{eq: limit law}), which is indeed the case, but our application of the bootstrap is different and uses $\sqrt{n}\| \hat{\vT}_n^B - \hat{\vT}_n \|_{L^s(R)}^s = \sqrt{n}\delta_s(\hat{\vz}_n^B,\hat{\vz}_n)$ instead. See Proposition 3.8 in \cite{goldfeld2022limit} for a related discussion. 

The proof of Part (i) goes as follows.
By the stability result from Theorem~\ref{thm: Hadamard_diff} (i), we can approximate $\sqrt{n}\delta_s(\hat{\vz}_n^B,\hat{\vz}_n)$ by
\[
\sqrt{n}[\delta_s]_{(\vz^*,\vz^*)}'(\hat{\vz}_n^B-\vz^*,\hat{\vz}_n-\vz^*),
\]
which, from the explicit expression of the derivative in (\ref{eq: h_deriv}), agrees with 
\[
[\delta_s]_{(\vz^*,\vz^*)}'\big(\sqrt{n}(\hat{\vz}_n^B-\hat{\vz}_n),\vzero\big). 
\]
The conditional law of the above converges weakly to the law of $[\delta_{s}]_{(\vz^*,\vz^*)}'(\vW,\vzero)$ with $\vW \sim \cN(\vzero,\bm{B}\bm{A}\bm{B}^\intercal)$, which agrees with the limit law in (\ref{eq: limit law}).

\begin{remark}[Dependent data]
Theorem \ref{thm: limit} follows from the Hadamard directional derivatives of the composite mappings $\bm{q} \mapsto \| \vT_{\vz^*(\bm{q})}-\vT_{\vz^*} \|_{L^2(R)}^s$ and $\bm{q} \mapsto \gamma_{\bm{\varphi}}(\vz^*(\bm{q}))$ combined with a CLT for $\hat{\vp}_n$, which is simply a sum of bounded random vectors, $\hat{p}_{i} = n^{-1}\sum_{j=1}^n \mathbbm{1}(\vX_j = \vx_i)$ for $i \in \{1,\dots,N\}$.  Hence, the conclusion of the theorem extends readily to dependent data. For example, suppose that $\vX_1,\vX_2,\dots \sim P$ are a stationary $\alpha$-mixing sequence,
\[
\alpha(k) \coloneqq \sup_{\ell \ge 1}\sup_{A \in \calF_{1}^\ell,B \in \calF_{k+\ell}^\infty} |\Prob (A)\Prob(B) - \Prob (A \cap B)| \to 0, \ k \to \infty,
\]
where $\calF_{i}^j$ is the $\sigma$-field generated by $\{ \vX_t : i \le t \le j \}$. Since the $\alpha$-mixing property is preserved under (measurable) transformations, the process $\{ \vX_t : t=1,2,\dots \}$ could be generated as a (discrete) transformation of another $\alpha$-mixing sequence. We refer to Chapter 2 in \cite{fan2003nonlinear} for details on mixing processes. Now, by Theorem 2.21 in \cite{fan2003nonlinear} and the Cr\'{a}mer-Wold device, as long as $\sum_{k=1}^\infty \alpha(k) < \infty$, one obtains 
\[
\sqrt{n}\big(\hat{\vp}-\vp\big) \stackrel{d}{\to} \mathcal{N} \big (\vzero,\bar{\bm{A}}\big),
\]
where the $(i,j)$-component of the $N \times N$ matrix $\bar{\bm{A}}$ is given by
\[
\begin{cases}
p_i(1-p_i) + 2\sum_{k=1}^\infty \Cov \big(\mathbbm{1}(\vX_1=\vx_i),\mathbbm{1}(\vX_{1+k}=\vx_i) \big) & \text{if $i=j$}, \\
-p_i p_j + 2\sum_{k=1}^\infty \Cov \big(\mathbbm{1}(\vX_1=\vx_i),\mathbbm{1}(\vX_{1+k}=\vx_j) \big) & \text{if $i=j$}.
\end{cases}
\]
Hence, the conclusion of Theorem \ref{thm: limit} continues to hold with $\bm{W}$ replaced by $\bar{\bm{W}} \sim \mathcal{N}(\bm{0},\bm{B}\bar{\bm{A}}\bm{B}^{\intercal} \big)$. 

On the other hand, Theorem \ref{prop: bootstrap consistency} does not directly extend to the dependent data scenario, because the nonparametric bootstrap fails to take into account the dependence of the data. Instead, one may use the moving block bootstrap \cite{kunsch1989}, which can be shown to be consistent for both functionals under mild regularity conditions. See \cite{lahiri2013resampling} for bootstrap methods for dependent data. 
\end{remark}

Finally, we state a super consistency result for the empirical OT map mentioned in Section \ref{sec: preliminaries}. We present a nonasymptotic version. 

\begin{proposition}[Super consistency]
\label{prop: pointwise}
Set $\mathsf{C}_{\vx} = \max_{1 \le i < j \le N} \|\vx_i-\vx_j\|$ and $\delta_0 = \min_{1 \le i \le N} p_i/2$.
Under Assumptions \ref{asp: cont} and \ref{asp: poincare}, for every $i \in \{ 1,\dots, N \}$ and compact set $K \subset \inte (C_i(\vz^*))$, we have 
\begin{equation}
\hat{\vT}_n(\vy) = \vT^*(\vy) \quad \text{for all} \ \vy \in K
\label{eq: super concentration}
\end{equation}
with probability at least
\[
1 - N e^{-2n \delta_0^2}  - (2^N-2) e^{-\frac{8n \varepsilon_0^2 \delta_0^2}{N^8 \mathsf{C}_{\vx}^2 \mathsf{C}_{\mathrm{P}}^2}},
\]
where $\varepsilon_0 = \min_{j \in \{ 1,\dots, N \} \setminus \{ i \}} \inf_{\vy \in K} \big \{ \langle \vx_i - \vx_j, \vy \rangle - b_{ij} \big\} > 0$.
Therefore, with probability one, \eqref{eq: super concentration} holds for all large enough $n$. 
\end{proposition}

The proof first verifies that \eqref{eq: super concentration} holds if $\| \hat{\vz}_n - \vz^* \| < \varepsilon_0/2$. The rest is to derive a concentration inequality for $\| \hat{\vz}_n - \vz^* \|$, which is done by bounding $\| \hat{\vz}_n-\vz^* \|$ by $\| \hat{\vp}_n-\vp \|$ up to a constant and invoking a concentration inequality for the latter in \cite{weissman2003inequalities}. The final claim follows from the Borel-Cantelli lemma.

The parameter $\varepsilon_0$ quantifies how close $K$ is to the boundary of $C_i(\vz^*)$. So the closer $K$ is to the boundary of $C_i(\vz^*)$, the less likely the event \eqref{eq: super concentration} happens. 

\textcolor{black}{
\begin{remark}[General integral error functionals]
   The results of Theorems~\ref{thm: Hadamard_diff} (i), \ref{thm: limit} (i), and \ref{prop: bootstrap consistency} (i) extend to more general error functionals than the $L^s$-error. Indeed, for any centrally symmetric function $\mathfrak{g}: \cX' \to \R$ with $\mathfrak{g}(\vzero)=0$ (here $\calX' = \{ \vx_i-\vx_j : 1 \le i,j \le N \}$), consider the following integral error functional:
    \[
    \mathsf{err}_{\mathfrak{g}} (\hat{\vT}_n - \vT^*)\coloneqq\int \mathfrak{g}(\hat{\vT}_n-\vT^*) \, dR
    \]
    Adapting the proof of Theorem ~\ref{thm: Hadamard_diff} (i), one can readily show that, under Assumption \ref{asp: cont}, the map $(\vz_1,\vz_2) \mapsto \mathsf{err}_{\mathfrak{g}} (\vT_{\vz_1}-\vT_{\vz_2})$ is Hadamard directionally differentiable at $(\vz^*,\vz^*)$ with derivative\footnote{As $\calX'$ is finite, $\mathfrak{g}$ is Lipschitz on $\calX'$, from which local Lipschiz continuity of $(\vz_1,\vz_2) \mapsto \mathsf{err}_{\mathfrak{g}} (\vT_{\vz_1}-\vT_{\vz_2})$ follows. The directional Gateaux derivative follows analogously to the $L^s$-error case. }
    \[
    (\vh_1,\vh_2) \mapsto \sum_{1 \le i < j \le N}\frac{\mathfrak{g}(\vx_i-\vx_j)}{\|\vx_i-\vx_j\|}R^+(D_{ij})|h_{2,j}-h_{2,i}-h_{i,j}+h_{1,i}|.
    \]
    The distributional limit and boostrap consistency for $\mathsf{err}_{\mathfrak{g}} (\hat{\vT}_n - \vT^*)$ follow analogously under Assumptions \ref{asp: cont} and \ref{asp: poincare}.
    This general setting allows to cover, e.g., the $L^s$-error for each coordinate, i.e.,  $\mathsf{err}_{\mathfrak{g}}(\hat{\vT}_n-\vT^*) = \| \hat{T}_{n,k} - T^*_k \|_{L^s(R)}^s$ when $\mathfrak{g}(\vx) = |x_k|^s$ for $\vx = (x_1,\dots,x_d)^\intercal$. 
\end{remark}
}

\section{Applications and numerical results}
\label{sec: applications}
\subsection{Applications}

The results of Theorems \ref{thm: limit} and \ref{prop: bootstrap consistency} enable us to construct $L^s$-confidence sets for $\vT^*$ and confidence intervals for $\langle \bm{\varphi},\vT^* \rangle_{L^2(R)}$. As a particular example of a linear functional, we consider a maximal tail correlation \cite{beirlant2020center}. 

\medskip 
4.1.1.~\textit{ $L^1$-confidence set and confidence band}. \ 
Consider constructing an $L^1$-confidence set for $\vT^*$. Given $\alpha \in (0,1)$, set $\hat{\tau}_{n,1-\alpha}$ as the conditional $(1-\alpha)$-quantile of $\sqrt{n}\| \hat{\vT}_n^B - \hat{\vT}_n\|_{L^1(R)}$, which can be computed via simulations.  The next corollary verifies the validity of the resulting $L^1$-confidence set. 

\begin{corollary}[Validity of $L^1$-confidence set]
\label{prop: L1 confidence set}
Under Assumption \ref{asp: cont} and \ref{asp: poincare}, the set
\begin{equation}
\big \{ \vT : \sqrt{n}\| \hat{\vT}_n-\vT \|_{L^1(R)} \le \hat{\tau}_{n,1-\alpha} \big \}
\label{eq: L1 confidence set}
\end{equation}
contains $\vT^*$ with probability approaching $1-\alpha$.
\end{corollary}

One drawback of $L^1$-confidence sets is that they are difficult to visualize compared to $L^\infty$-confidence bands. Section 5.8 in \cite{wasserman2006all} discusses a method to construct a confidence band from an $L^2$-confidence set, which builds on an idea in \cite{juditsky2003nonparametric}. Such a confidence band does not satisfy the uniform coverage guarantee but instead satisfies the \textit{average coverage}. We adapt the method in \cite[Section 5.8]{wasserman2006all} to $L^1$-confidence sets. Consider the confidence band of the form
\[
\cC_{n,1-\alpha}(\vy) = \left \{  \vx : \| \hat{\vT}_n (\vy) - \vx \| \le\frac{\hat{\tau}_{n,1-\alpha/2}}{\sqrt{n}} \cdot \frac{2}{\alpha} \right \}, \quad \vy \in \cY.
\]
Then, the argument in \cite[p.~95]{wasserman2006all} yields the following corollary.

\begin{corollary}[Validity of confidence band]
\label{prop: confidence band}
Under Assumption \ref{asp: cont} and \ref{asp: poincare}, the band $\cC_{n,1-\alpha}$ has average coverage at least $1-\alpha + o(1)$, i.e., 
\begin{equation}
\int \Prob \big (\vT^*(\vy) \in \cC_{n,1-\alpha}(\vy) \big) \, dR(\vy) \ge 1-\alpha + o(1).
\label{eq: average coverage}
\end{equation}
\end{corollary}
For the reader's convenience, we include the proof of Corollary \ref{prop: confidence band} in Section \ref{sec: proof Sec4}. 
Since $\vT^*$ only take values in $\cX$, we may intersect $\cC_{n,1-\alpha}(\vy)$ with $\cX$ to construct a tighter confidence band (cf. \cite{chernozhukov2020generic}), 
\[
\tilde{\cC}_{n,1-\alpha}(\vy) = \cC_{n,1-\alpha}(\vy) \cap \cX.
\]
The average coverage property continues to hold for the latter.

\medskip

4.1.2.~\textit{Inference for maximum tail correlation risk measure}. \ 
For univariate data, the quantile function (or Value at Risk) and its functionals such as the expected shortfall are commonly used as risk measures in financial risk management; see \cite{mcneil2015quantitative}. For multivariate $\vX \sim P$, \cite{ruschendorf2006law,ekeland2012comonotonic} propose a class of risk measures defined by the maximal correlation, 
\begin{equation}
\max \big \{ \E[ \langle \tilde{\vX}, \vY \rangle] : \tilde{\vX} \sim P, \vY \sim R \big \},
\label{eq: maximal correlation}
\end{equation}
where $R$ is a reference measure. In particular, \cite{ekeland2012comonotonic} give axiomatic characterizations of the maximum correlation, extending Kusuoka's characterizations of law invariant coherent risk measures to the multivariate case \cite{kusuoka2001law}. By definition, the maximal correlation (\ref{eq: maximal correlation}) agrees with $\E[ \langle \vY, \vT^*(\vY) \rangle]$ for $\vY \sim R$. 
Building on this observation, \cite{beirlant2020center} propose a modified risk measure defined by the maximal tail correlation $\E[ \langle \vY, \vT^*(\vY) \rangle \mid \| \vY \| \ge 1-\alpha ]$ when $R$ is the spherical uniform distribution. \cite{beirlant2020center} establish consistency of the empirical estimator for the maximal tail correlation, but do not develop methods of statistical inference for it. We consider a version of the maximal tail correlation for a different choice of the reference measure $R$ and construct confidence intervals for the risk measure. It should be noted that \cite{beirlant2020center} allow for a general target distribution $P$, while we focus here on finitely discrete $P$.

Fix any $\alpha,\beta \in (0,1)$.
Let $R$ be the uniform distribution over the unit ball $B_1$, which satisfies Assumptions \ref{asp: cont} (with $\cY = B_1$) and \ref{asp: poincare}. 
For $\vY \sim R$, let $r_{1-\alpha}$ denote the $(1-\alpha)$-quantile of $\| \vY \|$, i.e., $r_{1-\alpha} = (1-\alpha)^{1/d}$. Define a maximal tail correlation for $P$ by 
\be
\label{eq: kappa_def}
\kappa_{\alpha} = \E\big[\langle \vY, \vT^*(\vY) \rangle\mid \| \vY \| > r_{1-\alpha} \big] =\int \langle \bm{\varphi}_{\alpha} , \vT^* \rangle \, dR,
\ee
where $\bm{\varphi}_{\alpha} (\vy) = \alpha^{-1}\vy \mathbbm{1}_{B^c_{r_{1-\alpha}}} (\vy)$. Similarly, a maximal trimmed correlation can be defined by $\E\big[\langle \vY, \vT^*(\vY) \rangle\mid \| \vY \| \le r_{1-\alpha} \big]$, which can be dealt with analogously to the tail case.
The empirical estimator for $\kappa_{\alpha}$ is given by
\[
\hat{\kappa}_{n,\alpha} = \int \langle \bm{\varphi}_{\alpha}, \hat{\vT}_n \rangle \, dR.
\]
Since the test function $\bm{\varphi}_{\alpha}$ satisfies Assumption \ref{asp: text function}, we have
$
\sqrt{n}(\hat{\kappa}_{n,\alpha} - \kappa_{\alpha}) \stackrel{d}{\to} \cN(0,\sigma_{\bm{\varphi}_{\alpha}}^2)$
by Theorem \ref{thm: limit}. Assume now $\sigma_{\bm{\varphi}_{\alpha}}^2 > 0$ (cf. Lemma \ref{lem: nondegeneracy}).
Confidence intervals for $\kappa_{\alpha}$ can be constructed by using the bootstrap. Set 
\[
\hat{\kappa}_{n,\alpha}^B =\int \big\langle \bm{\varphi}_{\alpha}, \hat{\vT}_n^B \big\rangle \, dR(\vy)
\]
and $\hat{\tau}_{n,\beta}$ by the conditional $\beta$-quantile of $\hat{\kappa}_{n,\alpha}^B$. Then, Theorem \ref{prop: bootstrap consistency} above and Lemma 23.3 in \cite{vanderVaart1998asymptotic} immediately lead to the following corollary.
\begin{corollary}[Validity of bootstrap confidence interval]
\label{cor: bootstrap maximum correlation}
Under the above setting, the data-dependent interval
$
\big[ 2\hat{\kappa}_{n,\alpha} -\hat{\tau}_{n,1-\beta/2},2\hat{\kappa}_{n,\alpha} -\hat{\tau}_{n,\beta/2} \big ]
$
contains $\kappa_{\alpha}$ with probability approaching $1-\beta$.
\end{corollary}

Alternatively, one may use the bootstrap variance estimator for $\sigma^2_{\bm{\varphi}}$ to construct a normal confidence interval. 

\subsection{Numerical experiments}
\label{sec: simulations}

We present small-scale simulations to  assess the finite sample properties of the empirical $L^1$-error and maximum tail correlation.  For each case, we draw the histograms of the sampling distribution and the coverage probabilities for the bootstrap critical value. 
We used the python package \texttt{pysdot} \cite{leclerc2019pysdot} for solving the dual problem in both cases. We consider the following settings.

\begin{itemize}
\item \textbf{$\bm{L^1}$-error:} The discrete distribution $P$ is taken as the uniform distribution on 5 points chosen at random from the unit square $[0,1]^2$; the reference measure $R$ is the uniform distribution on $[0,1]^2$. To compute the coverage probabilities $\PP(\sqrt{n} \| \hat{\vT}_n-\vT^* \|_{L^1(R)} \leq \hat \tau_{n,\alpha})$ for $\alpha \in (0,1)$, for each Monte Carlo iteration, we compute the rank of $\sqrt{n} \| \hat{\vT}_n-\vT^* \|_{L^1(R)}$ w.r.t. the bootstrap distribution (i.e., $\hat F_n^B( \sqrt{n} \| \hat{\vT}_n-\vT^* \|_{L^1(R)} )$, where $\hat F_n^B$ is the bootstrap distribution function of $\sqrt{n}\| \hat{\vT}_n^B - \hat{\vT}_n\|_{L^1(R)}$), and evaluate how many ranks are less than $\alpha$ in the Monte Carlo repetitions. 

\smallskip
\item \textbf{Maximum tail correlation:} For the maximum tail correlation $\kappa_\alpha$, we choose $\alpha = 0.1$ and $P$ supported on the 4 points $(0.5,0.5),\,(0.5,-0.5),\,(-0.5,0.5)$, and $(-0.5,-0.5)$ in $\R^2$ with weights $\vp = (0.2,0.2,0.3,0.3)$. The reference measure $R$ is the uniform distribution over the unit ball $B_1$. We approximate the linear functional form \eqref{eq: kappa_def} via Monte Carlo integration with $M$  random points from $[-1,1]^2$, with $M=10000$ for the histogram and $M=1000$ for the coverage probabilities. Like in the $L^1$-error case, coverage probabilities are evaluated by first computing the rank of $\sqrt{n} (\hat{\kappa}_{n,\alpha}-\kappa_{\alpha})$ w.r.t. the bootstrap distribution for each Monte Carlo repetition, and then counting how many ranks are below $\beta$.
\end{itemize}

\begin{figure}
     \centering
     \begin{subfigure}[b]{0.47\textwidth}
         \centering
         \includegraphics[width=\textwidth]{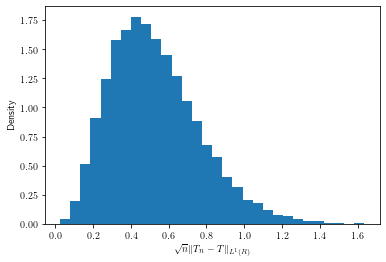}
         \caption{$L^1$-error: $n = 5000$}
         \label{fig: L1_error_hist_n5000}
     \end{subfigure}
     \hfill
     \begin{subfigure}[b]{0.47\textwidth}
         \centering
         \includegraphics[width=\textwidth]
         {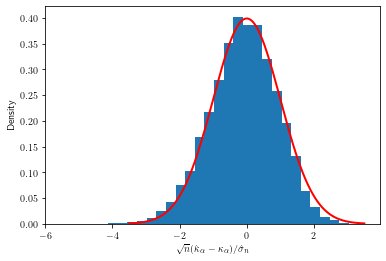}
         \caption{Maximal tail correlation: $n = 5000$}
         \label{fig: kappa_error_hist_n5000}
     \end{subfigure}
     \\
     \begin{subfigure}[b]{0.47\textwidth}
         \centering
         \includegraphics[width=\textwidth]{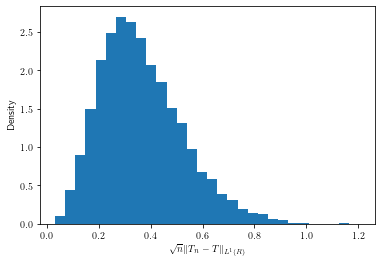}
         \caption{$L^1$-error: $n = 10000$}
         \label{fig: L1_error_hist}
     \end{subfigure}
     \hfill
     \begin{subfigure}[b]{0.47\textwidth}
         \centering
         \includegraphics[width=\textwidth]
         {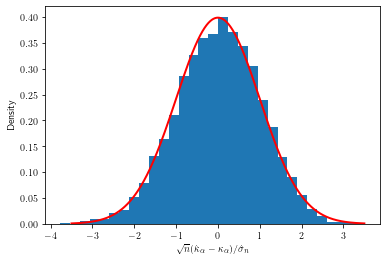}
         \caption{Maximal tail correlation: $n = 10000$}
         \label{fig: kappa_error_hist}
     \end{subfigure}
        \caption{\small Histograms for $\sqrt{n}\|\hat{\bm{T}}_n-\bm{T}^*\|_{L^1(R)}$ ((A) and (C)) and $\sqrt{n}(\hat{\kappa}_{n,\alpha}-\kappa_{\alpha})/\sigma_n$ (with $\alpha=0.1$) ((B) and (D)) based on $10000$ Monte Carlo repetitions. The sample size is $n=5000$ for (A) and (B) and $n=10000$ for (C) and (D); $\sigma_n$ is the Monte Carlo standard deviation of $\sqrt{n}(\hat{\kappa}_{n,\alpha}-\kappa_{\alpha})$. The red curves in (B) and (D) represent the $N(0,1)$ density. }
        \label{fig: histogram}
\end{figure}

Figure \ref{fig: histogram} shows the histograms of the $L^1$-error $\sqrt{n}\|\hat{\bm{T}}_n - \bm{T}^* \|_{L^1(R)}$ and the (standardized) maximum tail correlation $\sqrt{n}(\hat{\kappa}_{n,\alpha}-\kappa_{\alpha})/\sigma_n$, both with $n \in \{ 5000, 10000 \}$, based on $10000$ Monte Carlo repetitions. Here $\sigma_n$ is computed as the Monte Carlo standard deviation of $\sqrt{n}(\hat{\kappa}_{n,\alpha}-\kappa_{\alpha})$. The figure shows that: (i) the sampling distribution of $\sqrt{n}\|\hat{\bm{T}}_n - \bm{T}^* \|_{L^1(R)}$ is reasonably stable from $n=5000$ to $n=10000$, and the histograms are right-skewed, as can be expected from the form of the limit distribution from Theorem~\ref{thm: limit}; (ii) the histograms of $\sqrt{n}(\hat{\kappa}_{n,\alpha}-\kappa_{\alpha})/\sigma_n$ are close to the $N(0,1)$ density, which suggests that the normal approximation works well for the maximum tail correlation in the finite sample. 

Figure \ref{fig: PP plot} presents results concerning the coverage probabilities $\PP(\sqrt{n} \| \hat{\vT}_n-\vT^* \|_{L^1(R)} \leq \hat \tau_{n,\alpha})$ for $\alpha \in (0,1)$, and $\Prob(\sqrt{n}(\hat \kappa_{n,\alpha} - \kappa_\alpha) \leq \hat\tau_{n,\beta})$ for $\beta \in (0,1)$, both based on $250$ Monte Carlo repetitions. For each case, the sample size is $n \in \{ 5000,10000 \}$ and the number of bootstrap iterations is $500$. Here $\hat{\tau}_{n,\alpha}$ in (A) and (C) is the $\alpha$-quantile of the bootstrap distribution of $\sqrt{n}\|\hat{\bm{T}}_n-\hat{\bm{T}}_n\|_{L^1(R)}$, while $\hat{\tau}_{n,\beta}$ in (B) and (D) is the $\beta$-quantile of the bootstrap distribution of $\sqrt{n}(\hat{\kappa}_{n,\alpha}^B-\hat{\kappa}_{n,\alpha})$. The figure shows that in both cases, the coverage probabilities are close to the $45^\circ$ degree line, especially in the upper tails, which are relevant in applications, showing that the bootstrap works reasonably well to approximate the sampling distributions.

\begin{figure}
     \centering
     \begin{subfigure}[b]{0.47\textwidth}
         \centering
         \includegraphics[width=\textwidth]
         {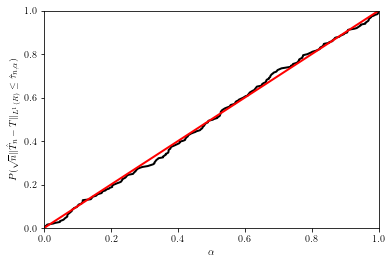}
         \caption{$L^1$-error: $n=5000$}
         \label{fig: L1_bootstrap_ppplot_n5000}
     \end{subfigure}
     \hfill
     \begin{subfigure}[b]{0.47\textwidth}
         \centering
         \includegraphics[width=\textwidth]
         {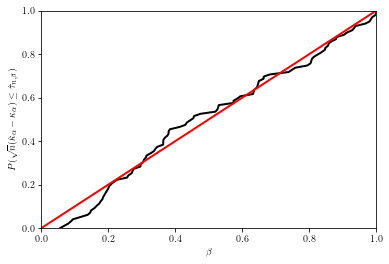}
         \caption{Maximum tail correlation: $n=5000$}
         \label{fig: kappa_bootstrap_ppplot_n5000}
     \end{subfigure}
     \\
     \begin{subfigure}[b]{0.47\textwidth}
         \centering
         \includegraphics[width=\textwidth]
         {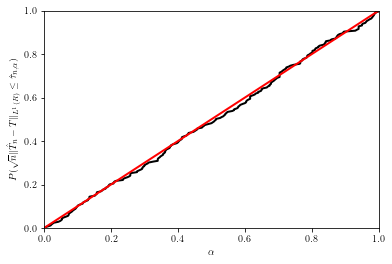}
         \caption{$L^1$-error: $n=10000$}
         \label{fig: L1_bootstrap_ppplot}
     \end{subfigure}
     \hfill
     \begin{subfigure}[b]{0.47\textwidth}
         \centering
         \includegraphics[width=\textwidth]
         {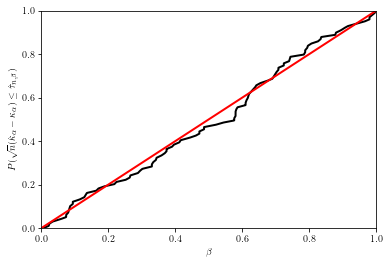}
         \caption{Maximum tail correlation: $n=10000$}
         \label{fig: kappa_bootstrap_ppplot}
     \end{subfigure}
     
        \caption{\small Coverage probabilities with bootstrap critical values based on 250 Monte Carlo repetitions: $\Prob (\sqrt{n}\|\hat{\bm{T}}_n-\bm{T}^*\|_{L^1(R)} \le \hat{\tau}_{n,\alpha})$ for $\alpha \in (0,1)$ ((A) and (C)) and $\Prob (\sqrt{n}(\hat{\kappa}_{n,\alpha}-\kappa_{\alpha})\le \hat{\tau}_{n,\beta})$ for $\beta \in (0,1)$ with $\alpha=0.1$ ((B) and (D)), where $\hat{\tau}_{n,\alpha}$ and $\hat{\tau}_{n,\beta}$ are the $\alpha$ and $\beta$-quantiles of the corresponding bootstrap distributions. The sample size is $n=5000$ for (A) and (B), and $n = 10000$ for (C) and (D); the number of bootstrap iterations is 500. The red line represents the $45^\circ$ degree line.}
        \label{fig: PP plot}
\end{figure}

\section{Weak limits and asymptotic efficiency in dual H\"older space}

\label{sec: dual holder}
\textcolor{black}{
So far, we have established limit theorems for the $L^s$-error and linear functionals of the empirical OT map. A question remains as to whether the empirical OT map has a weak limit in a function space. We first point out in Proposition \ref{prop: impossibility} below that the $L^2(R)$ norm is too strong for that purpose, meaning that the empirical OT map does not possess nontrivial weak limits in $L^2(R)$.\footnote{A similar observation was made in the recent preprint \cite{manole2023central} in the context of smooth OT map estimation for absolutely continuous distributions supported on the flat torus.}   Instead of $L^2(R)$, we shall view $\hat{\bm{T}}_n$ as elements of the topological dual of the $\alpha$-H\"{o}lder space with any $\alpha \in (0,1]$ and show that $\hat{\bm{T}}_n$ satisfies a CLT in the said Banach space. This enables us to address the question of asymptotic efficiency for estimating the OT map, for the first time, in an infinite-dimensional setting.
}

\subsection{Impossibility of nontrivial weak limits in $L^2$ space}
\textcolor{black}{
In view of Theorem \ref{thm: limit}, it would be natural to inquire whether $r_n\big(\hat{\vT}_n - \vT^*\big)$ has a distributional limit in $L^2(R;\R^d)$ for some norming sequence $r_n \to \infty$, where $L^2(R;\R^d)$ is the Hilbert space of Borel vector fields $\bm{\varphi}: \R^d \to \R^d$ whose coordinate functions are square integrable w.r.t. $R$, equipped with the inner product $\langle \cdot,\cdot \rangle_{L^2(R)}$. Indeed, Theorem \ref{thm: limit} implies that, for a given norming sequence $r_n \to \infty$, the only possible weak limit of $r_n\big(\hat{\vT}_n - \vT^*\big)$ is the point mass at $0$, which happens if and only if $r_n = o(n^{1/4})$. In other cases, $r_n\big(\hat{\vT}_n - \vT^*\big)$ does not possess a distributional limit. }

\begin{proposition}[Impossibility of nontrivial weak limits in $L^2$ space]
\label{prop: impossibility}
\textcolor{black}{Suppose Assumptions \ref{asp: cont} and \ref{asp: poincare} hold and that $L^2(R; \R^d)$ admits a complete orthonormal system $\{ \bm{\varphi}_j \}_{j=1}^\infty$ consisting of functions satisfying Assumption \ref{asp: text function}. Then, for a given norming sequence $r_n \to \infty$, if $r_n = o(n^{1/4})$, then $r_n\big(\hat{\vT}_n - \vT^*\big)$ converges to $0$ in probability in $L^2(R;\R^d)$, while if $n^{1/4} = O(r_n)$, then $r_n\big(\hat{\vT}_n - \vT^*\big)$ does not converge in distribution in $L^2(R;\R^d)$.}
\end{proposition}

\textcolor{black}{
For most common reference measures, a complete orthonormal system in $L^2(R;\R^d)$ can be constructed by applying Gram-Schmidt orthogonalization to polynomials, for which Assumption \ref{asp: text function} holds under tail conditions on the reference density. For example, the above proposition allows for the uniform distributions over the unit cube and ball, as well as the non-degenerate Gaussian distributions as reference measures. }

\subsection{Stability and CLT in dual H\"older space}
\textcolor{black}{First, we need some new notation. 
For $\alpha \in (0,1]$, recall that the $\alpha$-H\"older space $\cC^\alpha (\cY)$ is defined by the set of functions $f: \cY \to \R$ with $\| f \|_{\cC^\alpha} < \infty$, where
\[
\| f \|_{\cC^\alpha} \coloneqq \| f \|_{\infty} + \sup_{\substack{\vy \ne \vy'\\ \vy,\vy' \in \cY}} \frac{|f(\vy) - f(\vy')|}{\|\vy-\vy'\|^\alpha}. 
\]
Fix $\alpha \in (0,1]$. Consider the Banach space $\BB = \prod_{i=1}^N \cC^\alpha(\cY)  = \{ \bm \varphi = (\varphi_1,\dots,\varphi_N)^{\intercal} : \varphi_i \in \cC^{\alpha}(\cY), \forall i \in \{ 1,\dots, N \} \}$ with norm $\| \bm \varphi \|_{\BB} = \max_{1 \le i \le N} \| \varphi_i \|_{\cC^\alpha}$. Denote by $\BB_1 = \{ \bm \varphi : \| \bm \varphi \|_{\BB} \le 1\}$ the unit ball in $\BB$ and $\BB^*$ the (topological) dual of $\BB$, i.e., $\BB^*$ is the Banach space of bounded linear functionals on $\BB$ equipped with  norm $\| b^* \|_{\BB^*} = \sup_{\bm \varphi \in \BB_1} b^* (\bm\varphi)$.}

\textcolor{black}{
Observe that the OT map $\vT_{\vz}$ induces the linear functional on $\BB$ as
\[
\bm \varphi \mapsto \langle \bm \varphi, \vT_{\vz} \rangle_{L^2(R)},
\]
which is bounded on $\BB_1$. Hence, $\vT_{\vz}$ can be identified with an element of $\BB^*$. As it turns out, $\sqrt{n}(\hat \vT_n - \vT^*)$ does have a nontrivial weak limit as elements of $\BB^*$. To derive this result, we formally identify the mapping $\vz \mapsto \vT_{\vz}$ with $\Gamma: \R^N \to \BB^*$ defined as
\be\label{eq:dual_holder_functional}
\Gamma(\vz)(\bm\varphi) \coloneqq \langle \bm{\varphi},\vT_{\vz} \rangle_{L^2(R)}, \ \bm \varphi \in \BB, 
\ee
and show that the latter map is Fr\'echet differentiable. Combining the CLT for $\hat{\vz}_n$ and the extended delta method yields a CLT for $\sqrt{n} (\hat \vT_n-\vT^*) = \sqrt{n}\big(\Gamma (\hat{\vz}_n) - \Gamma (\vz^*)\big)$ in $\BB^*$.
To state the result, define $b_{i}^* \in \BB^\ast$ for $i \in \{ 1,\dots,N \}$ by 
\[
b^*_{i} : \bm\varphi \mapsto \sum_{j \ne i} \int_{D_{ij}} \frac{\langle \vx_i - \vx_j, \bm\varphi(\vy) \rangle}{\|\vx_i - \vx_j\|} \rho(\vy)\,d\cH^{d-1}(\vy).
\]
The map $\vz \mapsto \Gamma(\vz), \R^N \to \BB^*$ is continuous (indeed, locally Lipschitz; see the proof of Theorem \ref{thm: dual_holder_limit}~(i) below), so Borel measurable, which guarantees that $\hat{\vT}_n = \Gamma (\hat{\vz}_n)$ is Borel measurable as a mapping into $\BB^*$.}

\begin{theorem}[Stability and CLT in dual H\"older space]
\label{thm: dual_holder_limit}
\textcolor{black}{The following hold. 
\begin{enumerate}
    \item[(i)] Under Assumption~\ref{asp: cont}, $\Gamma: \R^N \to \BB^\ast$ is Fr\'echet differentiable at $\vz^\ast$ with derivative
    $\Gamma'_{\vz^*}(\vh) = \sum_{i=1}^N h_ib_i^*$, for $\vh = (h_1,\dots,h_N)^\intercal \in \R^N$.
    \item[(ii)] Under Assumptions~\ref{asp: cont} and \ref{asp: poincare}, we have 
$\sqrt{n} ( \hat{\vT}_n-\vT^* ) \stackrel{d}{\to} \mathbb{G}$ in $\BB^*$,
    where $\mathbb{G}$ is a centered Gaussian random variable in $\BB^\ast$ with $\mathbb{G} \stackrel{d}{=} \sum_{i = 1}^n W_i b_i^*$, for $\vW = (W_1,\dots,W_N)^{\intercal} \sim \cN(\vzero,\bm{B} \bm{A} \bm{B}^\intercal)$ given in Theorem~\ref{thm: limit}. 
\end{enumerate}
}
\end{theorem}

\textcolor{black}{
Part (i) of the theorem can be seen as a \textit{uniform-in-$\bm{\varphi}$} version of the Hadamard differentiability result for the linear functional in Theorem \ref{thm: Hadamard_diff}~(ii). In order that the derivative formula holds uniformly in $\bm{\varphi}$, we require the test functions to be (uniformly bounded and) uniformly equicontinuous. Part (ii) follows directly from the CLT for $\hat{\vz}_n$ and the extended delta method.}

\subsection{Asymptotic efficiency}
\textcolor{black}{
Next, we establish asymptotic efficiency (in the H\'ajek-Le Cam sense) of $\hat \vT_n$ in estimating $\vT^*$ when viewed as elements of $\BB^*$.
We follow the framework in Chapter 3.11 of \cite{van1996weak}. To this end, we need to specify statistical experiments $(\mathsf{X}_n,\cA_n, P_{n,\vh}: \vh \in H)$ indexed by a subspace $H$ of a Hilbert space and local parameters. Choose $H$ to be $H = \{ \vh \in \R^N : \vh^\intercal \vp = 0 \}$ equipped with the inner product $\langle \bm{a},\, \bm{b} \rangle_{\vp} \coloneqq \sum_{i=1}^N a_i b_i p_i$ for $\bm a =(a_1,\dots,a_N)^\intercal$ and $\bm b = (b_1,\dots,b_N)^\intercal$. Set $\mathsf X_n = \cX^n, \cA_n = 2^{\mathsf X_n}$, and $P_{n,\vh} = \big(\sum_{i=1}^N p_{n,\vh,i}\delta_{\vx_i}\big)^{\otimes n}$ with $p_{n,\vh,i} = (1+h_i/\sqrt{n})p_i$ for $\vh = (h_1,\dots,h_N)^\intercal \in H$. The observations $\vX_1,\dots,\vX_n$ are the coordinate projections of $\mathsf X_n$. 
By direct calculations, the log-likelihood ratio can be expanded as
\[
\log \frac{dP_{n,\vh}}{dP_{n,\vzero}} =   \underbrace{\langle \sqrt{n}(\hat \vp_n - \vp),\vh \rangle}_{=\Delta_{n,\vh}}- \frac{1}{2} \|\vh\|_{\vp}^2 + o_{P_{n,\vzero}}(1), 
\]
where $(\Delta_{n,\vh})_{\vh \in H} \stackrel{d}{\to} (\Delta_{\vh})_{\vh \in H}$ (finite dimensional convergence) under $P_{n,\vzero}$, and $(\Delta_{\vh})_{\vh \in H}$ is a centered Gaussian process with covariance function $\E[\Delta_{\vh_1}\Delta_{\vh_2}] = \vh_1^{\intercal}\bm{A}\vh_2 = \langle \vh_1,\vh_2 \rangle_{\vp}$ for $\vh_1,\vh_2 \in H$. Hence, the sequence of experiments $(\mathsf{X}_n,\cA_n, P_{n,\vh}: \vh \in H)$ is asymptotically normal in the sense of \cite[p.~412]{van1996weak}.}

\textcolor{black}{
We consider the local parameter sequence $\vT_n (\vh) = \Gamma(\vz^\ast(\vp_{n,\vh})) \in \BB^\ast$, and view $\hat \vT_n$ as an estimator for $\vT_n(\vh)$ with values in $\BB^*$, i.e., $\hat{\vT}_n = \Gamma (\hat{\vz}_n)$.
We say that the local parameter sequence $\vT_n (\vh)$ is \textit{regular} if $\sqrt{n}\big(\vT_n(\vh) - \vT_n(\vzero)\big) \to \dot \vT (\vh)$ for every $\vh \in H$ for some continuous linear map $\dot \vT: H \to \BB^*$. Furthermore, we say that a sequence of estimators $\hat{\cT}_n$ for $\vT_n(\vh)$ is \textit{regular} if the limit law of $\sqrt{n}\big(\hat{\cT}_n - \vT_n (\vh)\big)$ under $P_{n,\vh}$ exists for every $\vh \in H$ and is independent of $\vh$.}

\begin{proposition}[Asymptotic efficiency of empirical OT map]
\label{prop: asymptotic efficiency}
\textcolor{black}{
Consider the above setting, and let Assumptions~\ref{asp: cont} and \ref{asp: poincare} hold. Let $\mathbb{G}$ be a centered Gaussian random variable in $\BB^*$ given in Theorem \ref{thm: dual_holder_limit}. Then the following hold.
\begin{enumerate}
\item[(i)] (Convolution) The sequences of parameters $\vT_n(\vh)$ and estimators $\hat \vT_n$ are regular. For every regular sequence of  estimators $\hat\cT_n$, the limit law of $\sqrt{n}\big(\hat\cT_n-\vT_n(\bm 0)\big)$ under $P_{n,\bm 0}$ equals the distribution of the sum  $\mathbb{G}+\mathbb{W}$ for some $\mathbb{B}^\ast$-valued random variable $\mathbb{W}$ independent of $\mathbb{G}$.
\item[(ii)] (Local asymptotic minimaxity) For every sequence of (Borel measurable) estimators $\hat\cT_n$ and every $k \in \NN$,
\[
\sup_{I \subset H: \text{finite}} \liminf_{n \to \infty}\sup_{\vh \in I}\E_{\vh}\Big [ \Big \| \sqrt{n}(\hat\cT_n-\vT_n(\vh)) \Big\|_{\BB^*}^k \Big] \ge \E\big[ \| \mathbb{G} \|_{\BB^*}^k \big],
\]
where $\E_{\vh}$ denotes the expectation under $P_{n,\vh}$. Furthermore, for every $I \subset H$ finite, we have
\[
 \lim_{n \to \infty} \sup_{\vh \in I} \EE_{\vh} \left [ \Big \| \sqrt{n}(\hat\vT_n-\vT_n(\vh)) \Big\|_{\BB^*}^k\right ] = \E\big[ \| \mathbb{G} \|_{\BB^*}^k \big].
\]
\end{enumerate}
}
\end{proposition}

\textcolor{black}{
Part (i) of the theorem shows that the limit law of $\hat{\bm{T}}_n$ is the most concentrated around the origin among the regular estimators for $\bm{T}_n(\vh)$.
Part (ii) shows that the local asymptotic minimax lower bound for the loss $\| \cdot \|_{\BB^*}^k$ agrees with $\E[\| \GG \|_{\BB^*}^k]$, and the empirical OT map attains this lower bound, establishing its local asymptotic minimaxity against the said loss. The second claim of Part (ii) requires to establish moment convergence for the empirical OT map under local alternatives $P_{n,\bm{h}}$, which requires additional work.
}

\section{Proofs}
\label{sec: proofs}

\subsection{Proof of Theorem \ref{thm: Hadamard_diff}}
We separately prove Parts (i) and (ii). 

\medskip 

6.1.1.~\textit{Proof of Theorem \ref{thm: Hadamard_diff} (i)}. \
We first note that
\[
\delta_s(\vz_1, \vz_2) = \sum_{1\le i \neq j \le N} \|\vx_i - \vx_j\|^s R\big(C_i(\vz_1) \cap C_j(\vz_2)\big). 
\]
For $\vh_1,\vh_2 \in \R^N$ and $t > 0$, we have 
\[
C_j(\vz^* + t\vh_2) = \bigcap_{k\neq j} \big\{\vy: \langle \vx_j-\vx_k,\vy \rangle - b_{jk} \ge t (h_{2,k} -  h_{2,j})\big\},
\]
so  $C_i(\vz^* + t\vh_1) \cap C_j(\vz^* + t\vh_2) \subset D_{ij}(\vh_1,\vh_2,t) $, where
\[
\begin{split}
D_{ij}(\vh_1,\vh_2,t) &\coloneqq C_i(\vz^*+t\vh_1) \cap \big\{\vy: \langle \vx_i-\vx_j,\vy \rangle - b_{ij} \le t (h_{2,j} - t h_{2,i})\big\} \\
&\mspace{3mu}=\Bigg (\bigcap_{k \ne i,j} \big\{ \vy : \langle \vx_i-\vx_k,\vy \rangle - b_{ik} \ge t(h_{1,k} - h_{1,i}) \big\} \Bigg) \\
&\qquad \cap \big \{\vy : t(h_{1,j} - h_{1,i}) \leq  \langle \vx_i-\vx_j,\vy \rangle - b_{ij} \leq t(h_{2,j} - h_{2,i}) \big\}.
\end{split}
\]
Here we have used $b_{ji} = -b_{ij}$. We divide the rest of the proof into three steps. 

\medskip

\underline{Step 1}. Observe that $D_{ij}(\vh_1,\vh_2,t) \setminus \big(C_i(\vz^* + t\vh_1) \cap C_j(\vz^* + t\vh_2)\big) 
= D_{ij}(\vh_1,\vh_2,t) \setminus C_j(\vz^* + t\vh_2)\big)$. First, we shall show that for every $i,j \in \{ 1,\dots, N \}
$ with $i \ne j$  and every $\vh_1,\vh_2 \in \R^N$, 
\[
R\big(D_{ij}(\vh_1,\vh_2,t) \setminus   C_j(\vz^* + t\vh_2)  \big) =O(t^2), \ t \downarrow 0. 
\]
For every $\vy \in  D_{ij}(\vh_1,\vh_2,t) \setminus   C_j(\vz^* + t\vh_2)$, there exists $k \in \{ 1,\dots, N \} \setminus \{ i,j \}$ such that $\langle \vx_j - \vx_k, \vy \rangle - b_{jk} < t(h_{2,k} - h_{2,j})$. For such $k$, we have
\[
\begin{split}
\langle \vx_i - \vx_k,\vy \rangle -b_{ik} &= \langle \vx_i - \vx_j,\vy \rangle - b_{ij} + \langle \vx_j - \vx_k,\vy \rangle - b_{jk} \\
&< t(h_{2,j}-h_{2,i} + h_{2,k} - h_{2,j}) = t(h_{2,k}-h_{2,i}),
\end{split}
\]
where we have used the fact that $b_{ik}=b_{ij} + b_{jk}$. 
Since $\vy \in D_{ij}(\vh_1,\vh_2,t)$, we have $\langle \vx_i - \vx_k,\vy \rangle -b_{ik} \ge t(h_{1,k}-h_{1,i})$, and thus
\[
t(h_{1,k}-h_{1,i}) \le \langle \vx_i - \vx_k,\vy \rangle -b_{ik} <t(h_{2,k}-h_{2,i}). 
\]
Conclude that
\[
    D_{ij}(\vh_1,\vh_2,t) \setminus   C_j(\vz^* + t\vh_2) 
    \subset \bigcup_{k \neq i,j} A_{ijk}(t), 
\]
where
\[
\begin{split}
    A_{ijk}(t) \coloneqq &\big\{\vy: t(h_{1,j} - h_{1,i}) \leq  \langle \vx_i-\vx_j,\vy \rangle - b_{ij} \leq t(h_{2,j} - h_{2,i}) \big \}  \\&\quad \cap \big\{\vy: t(h_{1,k} - h_{1,i}) \leq  \langle \vx_i-\vx_k,\vy \rangle - b_{ik} \leq t(h_{2,k} - h_{2,i}) \big \}.
\end{split}
\]

Pick any $k \ne i,j$. 
If $\vx_i - \vx_j$ and $\vx_i - \vx_k$ are linearly independent, then we have
$R(A_{ijk}(t)) = O(t^2)$, which follows from the next lemma. 

\begin{lemma}
\label{lem: higher-order term}
Let $\bm{\alpha}_1,\bm{\alpha}_2 \in \R^d$ be linearly independent unit vectors with $\Delta\coloneqq \langle \bm{\alpha}_1,\bm{\alpha}_2 \rangle$, and let $b_i \in \R$ and $\underline{t}_i < \overline{t}_i$ for $i=1,2$.  Consider the affine subspace $H = \bigcap_{i=1}^2 \{ \vy : \langle \bm{\alpha}_i,\vy \rangle = b_i \}$ of dimension $d-2$. Then, we have 
\[
\bigcap_{i=1}^2 \big \{ \vy : \underline{t}_i \le \langle \bm{\alpha}_i,\vy \rangle - b_i \le \overline{t}_i \big \} \subset H^\delta \coloneqq \{ \vy : \dist (\vy,H) \le \delta \},
\]
where $\delta =2(1-\Delta^2)^{-1}\sum_{i=1}^2 \big ( |\overline{t}_i| + | \underline{t}_i| \big)$. Furthermore, under Assumption \ref{asp: cont}, we have $R\big(H^{\delta}\big) = O(\delta^2)$ as $\delta \downarrow 0$. 
\end{lemma}

The proof of this lemma will appear after the proof of Theorem \ref{thm: Hadamard_diff} (i).

Conversely, suppose that $\vx_i - \vx_j$ and $\vx_i - \vx_k$ are linearly dependent, i.e., $\vx_i - \vx_k = c (\vx_i-\vx_j)$ for some $c\neq 0$. Set $L_1 = \{\vy : \langle \vx_i - \vx_j, \vy \rangle = b_{ij} \}$ and $L_2 = \{\vy: \langle \vx_i - \vx_k, \vy \rangle =b_{ik}\} = \{ \vy : \langle \vx_i-\vx_j,\vy \rangle = c^{-1}b_{ik}  \}$. For every $\vy \in A_{ijk}(t)$,
\[
\dist(\vy,L_1) \vee \dist(\vy,L_2) \leq \frac{t (| h_{1,j} - h_{1,i}| \vee |h_{2,j} - h_{2,i} |\vee| h_{1,k} - h_{1,i}| \vee |h_{2,k} - h_{2,i} |)}{\|\vx_i - \vx_j\|}.
\]
Furthermore,  since $L_1$ and $L_2$ are parallel, we have
\[
\dist(L_1, L_2) \coloneqq \max \Big \{ \sup_{\vy' \in L_2} \dist (\vy',L_1), \sup_{\vy' \in L_1} \dist (\vy',L_2) \Big \} = \frac{|b_{ij} - c^{-1}b_{ik}|}{\|\vx_i - \vx_j\|}.
\]
Pick any $t_\ell \downarrow 0$ as $\ell \to \infty$. Suppose that $A_{ijk}(t_\ell) \ne \varnothing$ for infinitely many $\ell$. Then, we  have $\dist(L_1, L_2) = 0$, i.e., $b_{ik} = cb_{ij}$. In what follows, we separately consider the cases when $c < 0$, $c \in (0,1)$, or $c > 1$. 

\begin{itemize}
    \item When $c < 0$: Observe that 
\[
\begin{split}
C_i (\vz^*) &\subset \big\{\vy: \langle \vx_i - \vx_j, \vy \rangle \ge b_{ij} \big\} \cap \big\{ \vy: \langle \vx_i - \vx_k, \vy \rangle \ge b_{ik} \big\} \\
&=\big\{\vy: \langle \vx_i - \vx_j, \vy \rangle \ge b_{ij} \big\} \cap \big\{ \vy: c \langle \vx_i - \vx_j, \vy \rangle \ge c b_{ij} \big\} \\
&=\big\{\vy: \langle \vx_i - \vx_j, \vy \rangle \ge b_{ij} \big\} \cap \big\{ \vy: \langle \vx_i - \vx_j, \vy \rangle \le b_{ij} \big\} \quad (\text{because $c<0$}) \\
&=\big \{\vy: \langle \vx_i - \vx_j, \vy \rangle = b_{ij}\big\},
\end{split}
\]
which entails $R\big(C_i(\vz^*)\big) = 0$. But this contradicts the fact that $R\big(C_i(\vz^*)\big) = p_i > 0$.
  \item When $c \in (0,1)$: Since $b_{ki} = -cb_{ij}$, $b_{kj} = b_{ki} + b_{ij}  = (1-c) b_{ij}$, and $\vx_k-\vx_j = (1-c)(\vx_i-\vx_j)$, we have 
\[
\begin{split}
C_k (\vz^*) &\subset \big\{\vy: \langle \vx_k - \vx_i, \vy \rangle \ge b_{ki} \big\} \cap \big\{ \vy: \langle \vx_k - \vx_j, \vy \rangle \ge b_{kj} \big\} \\
&=\big\{\vy: -c\langle \vx_i - \vx_j, \vy \rangle \ge -cb_{ij} \big\} \cap \big\{ \vy: (1-c) \langle \vx_i - \vx_j, \vy \rangle \ge (1-c) b_{ij} \big\} \\
&=\big\{\vy: \langle \vx_i - \vx_j, \vy \rangle \le b_{ij} \big\} \cap \big\{ \vy: \langle \vx_i - \vx_j, \vy \rangle \ge b_{ij} \big\} \quad (\text{because $c \in (0,1)$}) \\
&=\big \{\vy: \langle \vx_i - \vx_j, \vy \rangle = b_{ij}\big\},
\end{split}
\]
which entails $R\big(C_k(\vz^*)\big) = 0$, a contradiction.
\item When $c > 1$: Since $b_{jk} = (c-1) b_{ij}$ and $\vx_j-\vx_k = (c-1)(\vx_i-\vx_j)$, we have
\[
\begin{split}
C_j (\vz^*) &\subset \big\{\vy: \langle \vx_j - \vx_i, \vy \rangle \ge b_{ji} \big\} \cap \big\{ \vy: \langle \vx_j - \vx_k, \vy \rangle \ge b_{jk} \big\} \\
&=\big\{\vy: \langle \vx_i - \vx_j, \vy \rangle \le b_{ij} \big\} \cap \big\{ \vy: (c-1) \langle \vx_i - \vx_j, \vy \rangle \ge (c-1) b_{ij} \big\} \\
&=\big\{\vy: \langle \vx_i - \vx_j, \vy \rangle \le b_{ij} \big\} \cap \big\{ \vy: \langle \vx_i - \vx_j, \vy \rangle \ge b_{ij} \big\} \quad (\text{because $c > 1$}) \\
&=\big \{\vy: \langle \vx_i - \vx_j, \vy \rangle = b_{ij}\big\},
\end{split}
\]
which entails $R\big(C_j(\vz^*)\big) = 0$, once more, a contradiction.
\end{itemize}

\medskip
Therefore, in all cases, we have $A_{ijk}(t) = \varnothing$ for all sufficiently small $t > 0$. Finally, summing $R\big(A_{ijk}(t)\big)$ over $k \neq i, j$, we obtain
\[
R\big(D_{ij}(\vh_1,\vh_2,t) \setminus   C_j(\vz^* + t\vh_2)  \big) =O(t^2),
\]
as desired.

\medskip

\underline{Step 2}. 
Next, we shall evaluate the probability $R(D_{ij}(\vh_1,\vh_2,t))$ as $t \downarrow 0$, which is given in the following lemma. 

\begin{lemma}
\label{lem: thin mass}
Under Assumption \ref{asp: cont}, for every $i,j \in \{ 1,\dots,N \}$ with $i \ne j$ and every $\vh_1,\vh_2 \in \R^N$, we have as $t \downarrow 0$,
\begin{equation}
R \big ( D_{ij}(\vh_1,\vh_2,t) \big) + R\big( D_{ji}(\vh_1,\vh_2,t) \big) = R^{+}(D_{ij})  \frac{t | h_{2,j}-h_{2,i}-h_{1,j}+h_{1,i} |}{\|\vx_i - \vx_j\|} + o(t).
\label{eq: surface measure}
\end{equation}
\end{lemma}
The proof of Lemma \ref{lem: thin mass} is lengthy and  deferred to after the proof of Theorem \ref{thm: Hadamard_diff} (i).

\medskip

\underline{Step 3}. By Steps 1 and 2, we have
\begin{align*}
    \delta_s(\vz^* + t\vh_1,\vz^* + t\vh_2) &= \sum_{1\leq i \neq j \leq N} \|\vx_i - \vx_j\|^s R \big ( C_i(\vz^* + t\vh_1) \cap C_j(\vz^* + t\vh_2) \big )\\
    &= \sum_{1\leq i \ne j \leq N} \|\vx_i - \vx_j\|^s R \big ( D_{ij}(\vh_1,\vh_2, t) \big) + o(t) \\
    &=t \underbrace{\sum_{1\leq i < j \leq N} \|\vx_i - \vx_j\|^{s-1} R^{+}(D_{ij}) | h_{2,j}-h_{2,i} - h_{1,j} + h_{1,i} |}_{=: [\delta_s]'_{(\vz^*,\vz^*)}(\vh_1,\vh_2)} + o(t).
\end{align*}
Thus, we have derived the directional  Gateaux differentiability for $\delta$, i.e., 
\[
\lim_{t \downarrow 0}\frac{\delta_s(\vz^* + t\vh_1,\vz^* + t\vh_2)}{t} = [\delta_s]'_{(\vz^*,\vz^*)}(\vh_1,\vh_2).
\]
To lift the Gateaux differentiability to the Hadamard one, it suffices to verify that $\delta_s$ is locally Lipschitz \cite{shapiro1990concepts}. To this end, observe that for $\vY \sim R$,
\begin{equation}
\begin{split}
    &|\delta_s(\vz_1,\vz_2) - \delta_s(\vz'_1,\vz'_2)| \\
    &= \left| \EE \left [ \|\vT_{\vz_1}(\vY) - \vT_{\vz_2}(\vY)\|^s - \|\vT_{\vz'_1}(\vY) - \vT_{\vz'_2}(\vY)\|^s \right] \right |\\
    &\leq s \max_{1 \leq i < j \leq N}\|\vx_i -\vx_j\|^{s-1} \EE\left [ \|\vT_{\vz_1}(\vY) - \vT_{\vz'_1}(\vY)\| + \|\vT_{\vz'_2}(\vY) - \vT_{\vz_2}(\vY)\| \right ],
    \end{split}
    \label{eq: delta Lip}
\end{equation}
where we used the inequality $|a^s-b^s| \le s(a \vee b)^{s-1}|a-b|$ for $a,b \ge 0$ combined with the fact that $\| \vT_{\vz} (\vy) - \vT_{\vz'}(\vy) \| \le \max_{1 \leq i < j \leq N}\|\vx_i -\vx_j\|$.
For $\vz,\vz' \in \R^N$, we have
\begin{equation}
\begin{split}
    \EE\big[\|\vT_{\vz}(\vY) - \vT_{\vz'}(\vY)\|\big] &= \sum_{1 \le i \ne j \le N} \|\vx_i - \vx_j\| R\big(C_i(\vz) \cap C_j(\vz')\big)\\
    &\leq \max_{1 \leq i < j \leq N} \|\vx_i - \vx_j\| \sum_{i=1}^N  R\big(C_i(\vz) \setminus C_i(\vz')\big),
    \label{eq: delta Lip2}
    \end{split}
\end{equation}
where the second inequality follows because $\{C_i(\vz')\}_{i=1}^N$ forms a partition of $\R^d$ up to Lebesgue negligible sets. 

Combining (\ref{eq: delta Lip}) and (\ref{eq: delta Lip2}), we have 
\begin{equation}
\begin{split}
 &|\delta_s(\vz_1,\vz_2) - \delta_s(\vz'_1,\vz'_2)| \\
 &\leq  s\max_{1 \leq i < j \leq N} \|\vx_i - \vx_j\|^s \sum_{i=1}^N \big \{ R\big(C_i({\vz_1}) \setminus C_i({\vz_1}')\big) + R\big(C_i({\vz_2}) \setminus C_i({\vz_2}')\big) \big \}.
 \end{split}
 \label{eq: delta_difference_bound}
\end{equation}
Hence, local Lipschitz continuity of $\delta_s$ follows from the next lemma, whose proof is postponed to after the proof of Theorem \ref{thm: Hadamard_diff} (i). 
\begin{lemma}
\label{lem: quant_stability}
Pick any $i \in \{ 1,\dots, N \}$. 
Under Assumption \ref{asp: cont}, for sufficiently small $U > 0$, there exists a constant $\ell_U$ such that $R\big(C_i(\vz) \setminus C_i(\vz')\big) \le \ell_U \| \vz-\vz' \|$ whenever $\|\vz-\vz^*\| \vee \| \vz'-\vz^*\| \le U$.
\end{lemma}

Now, for every $\vh_1^t \to \vh_1 \in \R^N$ and $\vh_2^t \to \vh_2 \in \R^N$ as $t \downarrow 0$, we have
\begin{align*}
&\lim_{t \downarrow 0} \frac{\delta_s(\vz^* + t\vh_1^t,\vz^* + t\vh_2^t) - \delta_s(\vz^*,\vz^*)}{t} - [\delta_s]'_{(\vz^*,\vz^*)}(\vh_1,\vh_2)  \\
&= \lim_{t \downarrow 0} \frac{\delta_s(\vz^* + t\vh_1,\vz^* + t\vh_2) - \delta_s(\vz^*,\vz^*)}{t} - [\delta_s]'_{(\vz^*,\vz^*)}(\vh_1,\vh_2) \\
&\qquad   + \lim_{t \downarrow 0}\frac{\delta_s(\vz^* + t\vh_1^t,\vz^* + t\vh_2^t) - \delta_s(\vz^*+t\vh_1,\vz^*+t\vh_2)}{t}\\
& = 0+\lim_{t \downarrow 0} O(\|\vh_1^t - \vh_1\| + \|\vh_2^t - \vh_2\|) = 0.
\end{align*}
This completes the proof.
\qed

\medskip

We are left to prove Lemmas \ref{lem: higher-order term}--\ref{lem: quant_stability}.

\begin{proof}[Proof of Lemma \ref{lem: higher-order term}]
Assume $\Delta \ge 0$. The $\Delta \le 0$ case is similar. For $\vy$ with $\underline{t}_i \le \langle \bm{\alpha}_i,\vy \rangle - b_i \le \overline{t}_i$ for $i=1,2$, one can find $\vy_0 \in H$ and $\tau_1,\tau_2 \in \R$ such that $\vy = \vy_0 + \sum_{i=1}^2 \tau_i \bm{\alpha}_i$. The scalars $\tau_1,\tau_2$ must satisfy $\underline{t}_1 \le \tau_1 + \tau_2 \Delta \le \overline{t}_1$ and $\underline{t}_2 \le \Delta \tau_1 + \tau_2  \le \overline{t}_2$. Solving these linear inequalities, we have
\[
\frac{\underline{t}_1 - \Delta \overline{t}_2}{1-\Delta^2} \le \tau_1 \le \frac{\overline{t}_1 - \Delta \underline{t}_2}{1-\Delta^2} \quad \text{and} \quad \frac{\underline{t}_2 - \Delta \overline{t}_1}{1-\Delta^2} \le \tau_2 \le \frac{\overline{t}_2 - \Delta \underline{t}_1}{1-\Delta^2},
\]
which implies that 
\[
|\tau_1| \vee | \tau_2| \le \frac{1}{1-\Delta^2} \sum_{i=1}^2 \big ( |\overline{t}_i| + | \underline{t}_i| \big)
\]
Note that $\Delta^2 < 1$ by linear independence between $\bm{\alpha}_1$ and $\bm{\alpha}_2$. Now,
$
\dist (\vy,H) \le \big \| \sum_{i=1}^2 \tau_i \bm{\alpha}_i \| \le \sum_{i=1}^2 | \tau_i | \le 2(1-\Delta^2)^{-1}\sum_{i=1}^2 \big ( |\overline{t}_i| + | \underline{t}_i| \big) = \delta$. 

The second claim follows from a change-of-variable argument. Fix arbitrary $\vy_0 \in H$. Let $\{ \bm{u}_1,\dots,\bm{u}_{d-2} \}$ be an orthonormal basis of the linear subspace $H-\vy_0=\{ \vy - \vy_0 : \vy \in H \}$, and let $\{\bm{v}_1,\bm{v}_2 \}$ be an orthonormal basis of the linear subspace spanned by $\{ \bm{\alpha}_1,\bm{\alpha}_2 \}$. Then, every $\vy \in \R^d$ can be parameterized as 
\[
\vy = \underbrace{\vy_0 + \sum_{i=1}^{d-2} w_i \vu_i}_{=: \bar{\vy}(w_1,\dots,w_{d-2})} + \sum_{j=1}^2 \tau_j \vv_j, \quad (w_1,\dots,w_{d-2},\tau_1,\tau_2)^{\intercal} \in \R^d.
\]
Since the Jacobian matrix w.r.t. the change of variables $(w_1,\dots,w_{d-2},\tau_1,\tau_2)^{\intercal} \to \vy$ is $(\vu_1,\dots,\vu_{d-2},\vv_1,\vv_2)$, which is orthogonal, $R\big(H^{\delta}\big)$ can be expressed as 
\[
\begin{split} 
&\int_{\sum_{i=1}^2 \tau_i^2 \le \delta^2} \left  \{\int_{\R^{d-2}} \rho \big(\bar{\vy}(w_1,\dots,w_{d-2}) + {\textstyle \sum}_{i=1}^2 \tau_i \vv_i\big) \, d(w_1,\dots,w_{d-2}) \right\} d(\tau_1,\tau_2) \\
&=\int_{\sum_{i=1}^2 \tau_i^2 \le \delta^2} \left  \{\int_{H} \rho \big(\vy+ {\textstyle \sum}_{i=1}^2 \tau_i \vv_i\big) \, d\calH^{d-2}(\vy) \right\} d(\tau_1,\tau_2),
\end{split}
\]
where we used the area formula (cf. Theorem 3.9 in \cite{evans1991measure}) to deduce the second line. Assumption \ref{asp: cont} then guarantees that, for small $\delta$, the inner integral is bounded as a function of $(\tau_1,\tau_2)$, from which we conclude that $R\big (H^\delta\big) = O(\delta^2)$. 
\end{proof}

\begin{proof}[Proof of Lemma \ref{lem: thin mass}]
The proof is partially inspired by that of  Lemma 2 in \cite{chernozhukov2017detailed}. We first note that if $h_{1,j} - h_{1,i} = h_{2,j} - h_{2,i}$, then
\[
R\big(D_{ij}(\vh_1,\vh_2,t)\big) = R\big(D_{ji}(\vh_1,\vh_2,t)\big) = 0,
\]
so there is nothing to prove. Consider the case where $h_{1,j} - h_{1,i} \ne h_{2,j} - h_{2,i}$. Assume without loss of generality that $h_{2,j}-h_{2,i} > h_{1,j} - h_{1,i}$, for which $D_{ji}(\vh_1,\vh_2,t) = \varnothing$.  We shall show that 
\[
R\big(D_{ij}(\vh_1,\vh_2,t)\big) = \frac{t (h_{2,j}-h_{2,i} - h_{1,j}+h_{i,1})}{\|\vx_i - \vx_j\|} R^{+}(D_{ij}) + o(t).
\]
We divide the remaining proof into two steps. 

\medskip

\underline{Step 1}. 
We first show that 
\[
R(D_{ij}\big(\vh_1,\vh_2,t)\big) = R\big(\tilde D_{ij}(\vh_1,\vh_2,t)\big) + O(t^2),
\]
where 
\[
\tilde D_{ij}(\vh_1, \vh_2, t) \coloneqq\left \{\vy_0 + \tau \vv : \vy_0 \in D_{ij},\, \frac{t(h_{1,j} - h_{1,i})}{\|\vx_i - \vx_j\|}\leq \tau \leq \frac{t(h_{2,j}- h_{2,i})}{\|\vx_i - \vx_j\|} \right\}
\]
and $\vv = (\vx_i - \vx_j)/\|\vx_i - \vx_j\|$, the unit normal vector to the hyperplane $H_{ij} \coloneqq \{ \vy : \langle \vx_i-\vx_j,\vy \rangle = b_{ij} \}$ containing $D_{ij}$; see Figure~\ref{fig:laguerre_2}.

\begin{figure}
    \centering
    \includegraphics[width =0.48\textwidth]{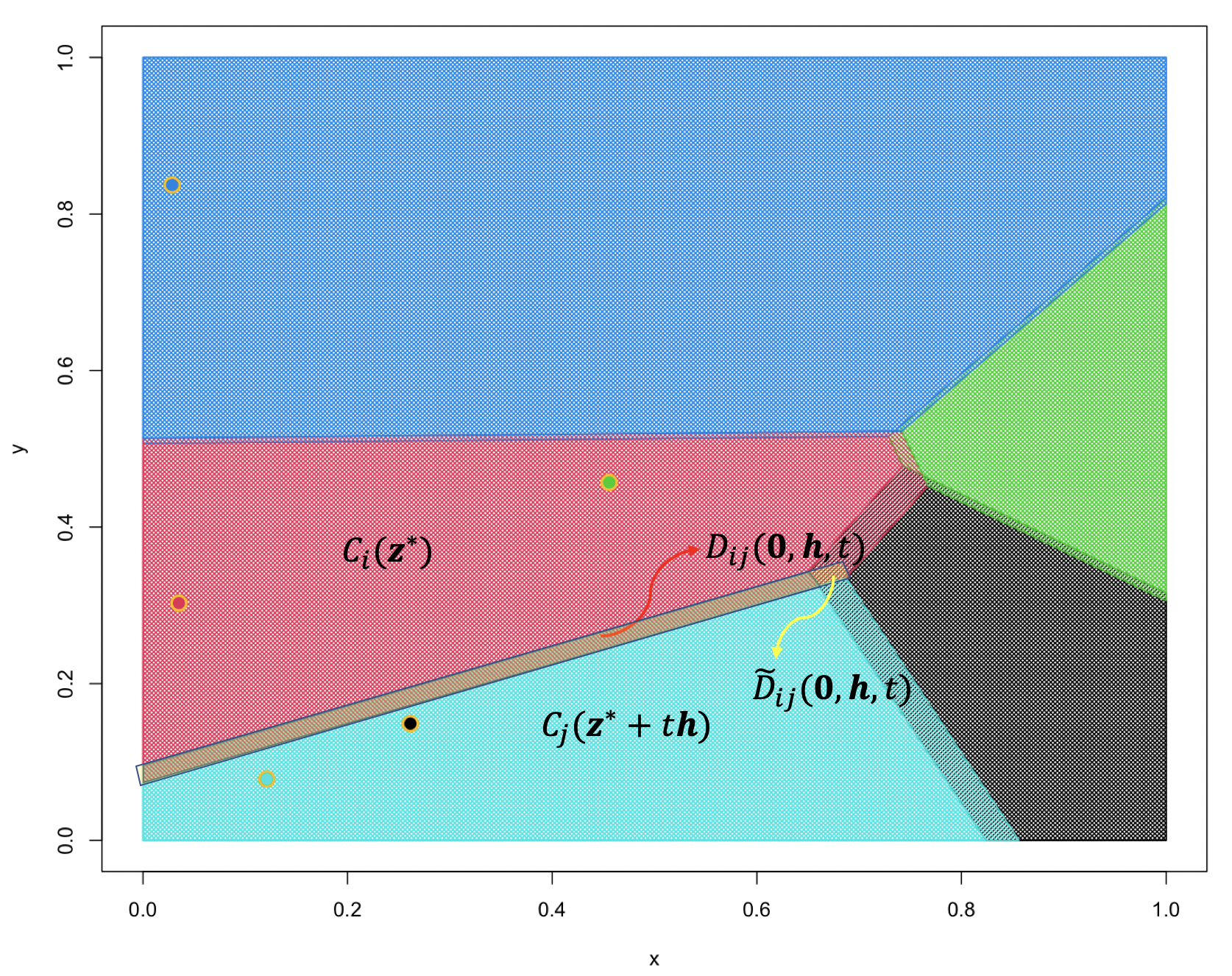}
    \caption{\small Perturbations of Laguerre cells. The figure shows the overlap of Laguerre tessellations corresponding to two dual potential vectors $\vz^*$ and $\vz^*+t\vh$. $\tilde{D}_{ij}(\vzero,\vh,t)$ is an inflation of the boundary $D_{ij}(\vz^*)$ towards the interior of $C_i(\vz^*)$, while $D_{ij}(\vzero,\vh,t)$ is a superset of the overlap between $C_j(\vz^{*}+th)$ and $C_i(\vz^*)$.}
    \label{fig:laguerre_2}
\end{figure}
Recall that 
\[
D_{ij} = \bigg (\bigcap_{k \ne i,j}\big \{ \vy : \langle \vx_i-\vx_k,\vy \rangle \ge b_{ik}  \big\} \bigg) \cap \big\{ \vy: \langle \vx_i-\vx_j,\vy \rangle = b_{ij}  \big\}.
\]

Consider $\vy = \vy_0 + \tau\vv \in  \tilde{D}_{ij}(\vh_1,\vh_2,t) \setminus D_{ij}(\vh_1,\vh_2,t)$. As $\vy \notin D_{ij}(\vh_1,\vh_2,t)$, we have either $\vy \notin C_i(\vz^*+t\vh_1)$ or $\langle \vx_i-\vx_j,\vy \rangle -b_{ij} > t(h_{2,j}-h_{2,i})$, but the latter cannot hold because $\vy \in \tilde D_{ij}(\vh_1, \vh_2, t)$. So we must have $\vy \notin C_i(\vz^*+t\vh_1)$. Furthermore, since
\[
\begin{split}
\langle \vx_i - \vx_j, \vy \rangle - b_{ij} &= \underbrace{\langle \vx_i - \vx_j, \vy_0 \rangle}_{=b_{ij}} +\tau \underbrace{\langle \vx_i - \vx_j, \vv \rangle}_{=\| \vx_i-\vx_j \|} - b_{ij} \\
&=  \tau \|\vx_i - \vx_j\|\\
& \geq t(h_{1,j} - h_{1,i}),
\end{split}
\]
there exists $k \neq i, j$ such that
\[
t(h_{1,k} - h_{1,i}) > \langle \vx_i - \vx_k, \vy \rangle - b_{ik} \ge -t\underbrace{|\langle \vx_i-\vx_k,\vx_i-\vx_j \rangle| \big (|h_{1,j}-h_{1,i}| \vee |h_{2,j}-h_{2,i}|\big)}_{=:\ell_{ijk}}.
\]
Conclude that 
\[
\tilde D_{ij}(\vh_1, \vh_2, t) \setminus D_{ij}(\vh_1,\vh_2,t) \subset \bigcup_{k \ne i,j} \tilde{A}_{ijk}(t),
\]
where 
\[
\begin{split}
\tilde{A}_{ijk}(t) &= \big\{\vy: t(h_{1,j} - h_{1,i}) \leq  \langle \vx_i-\vx_j,\vy \rangle - b_{ij} \leq t(h_{2,j} - h_{2,i}) \big \} \\
&\quad \cap \big \{ \vy : -t \ell_{ijk} \le \langle \vx_i - \vx_k, \vy \rangle - b_{ik} \le t(h_{1,k} - h_{1,i}) \big \}.
\end{split}
\]
Arguing as in Step 1 in the proof of Theorem \ref{thm: Hadamard_diff}, we see that $R\big(\tilde{A}_{ijk}(t)\big) = O(t^2)$, so  $R\big(\tilde D_{ij}(\vh_1, \vh_2, t) \setminus D_{ij}(\vh_1,\vh_2,t)\big) = O(t^2)$. 

Next, we consider $D_{ij}(\vh_1,\vh_2,t) \setminus  \tilde D_{ij}(\vh_1,\vh_2,t)$. Recall the hyperplane $H_{ij}=:\{ \vy : \langle \vx_i-\vx_j,\vy \rangle = b_{ij} \}$. Since 
\[
D_{ij}(\vh_1,\vh_2,t) \subset \left \{\vy_0 + \tau \vv : \vy_0 \in H_{ij},\, \frac{t(h_{1,j} - h_{1,i})}{\|\vx_i - \vx_j\|}\leq \tau \leq \frac{t(h_{2,j}- h_{2,i})}{\|\vx_i - \vx_j\|} \right\}, 
\]
the projection of every $\vy \in D_{ij}(\vh_1,\vh_2,t) \setminus  \tilde D_{ij}(\vh_1,\vh_2,t)$ onto $H_{ij}$ falls outside of $D_{ij}$, i.e., $\vy$ can be decomposed as $\vy = \vy_0 + \tau \vv$ for some $\vy_0 \in H_{ij} \setminus D_{ij}$ and $\frac{t(h_{1,j} - h_{1,i})}{\|\vx_i - \vx_j\|}\leq \tau \leq \frac{t(h_{2,j}- h_{2,i})}{\|\vx_i - \vx_j\|}$. By definition, there exists $k \neq i,j$ such that
$\langle \vx_i - \vx_k, \vy_0 \rangle < b_{ik}$, which implies that 
$
 t(h_{1,k}-h_{1,i}) \le \langle \vx_i - \vx_k, \vy \rangle - b_{ik} < t\ell_{ijk}.
$
Hence,
\[
D_{ij}(\vh_1,\vh_2,t) \setminus  \tilde D_{ij}(\vh_1,\vh_2,t) \subset \bigcup_{k \ne i,j} \tilde{\tilde{A}}_{ijk}(t),
\]
where 
\[
\begin{split}
\tilde{\tilde{A}}_{ijk}(t) &= \big\{\vy: t(h_{1,j} - h_{1,i}) \leq  \langle \vx_i-\vx_j,\vy \rangle - b_{ij} \leq t(h_{2,j} - h_{2,i}) \big \} \\
&\quad \cap \big \{ \vy :  t(h_{1,k}-h_{1,i}) \le \langle \vx_i - \vx_k, \vy \rangle - b_{ik} \le t\ell_{ijk} \big \}.
\end{split}
\]
Again, arguing as in Step 1 in the proof of Theorem \ref{thm: Hadamard_diff}, we see that $R\big(\tilde{\tilde{A}}_{ijk}(t)\big) = O(t^2)$, so  $R\big( D_{ij}(\vh_1, \vh_2, t) \setminus \tilde D_{ij}(\vh_1,\vh_2,t)\big) = O(t^2)$. 
Conclude that 
\[
R\big( D_{ij}(\vh_1,\vh_2,t) \Delta \tilde D_{ij}(\vh_1,\vh_2,t) \big ) =  O(t^2).
\]

\medskip

\underline{Step 2}. In view of Step 1, it suffices to prove the desired conclusion with $D_{ij}(\vh_1,\vh_2,t)$ replaced by $\tilde{D}_{ij}(\vh_1,\vh_2,t)$. Observe that
\begin{align*}
    R\big(\tilde D_{ij}(\vh_1,\vh_2,t)\big) &= \int_{\tilde D_{ij}(\vh_1,\vh_2,t)} \rho(\vy) \, d\vy \\&=  \int^{\frac{t(h_{2,j} - h_{2,i})}{\|\vx_i - \vx_j\|}}_{\frac{t(h_{1,j} - h_{1,i})}{\|\vx_i - \vx_j\|}} \left ( \int_{D_{ij}} \rho(\vy + \tau\vv) \,d\cH^{d-1}(\vy)\right )\,d\tau\\
    &= t \int^{\frac{h_{2,j} - h_{2,i}}{\|\vx_i - \vx_j\|}}_{\frac{h_{1,j} - h_{1,i}}{\|\vx_i - \vx_j\|}} \left ( \int_{D_{ij}} \rho(\vy+ t\tau\vv) \,d\cH^{d-1}(\vy)\right )\,d\tau,
\end{align*}
where the second equality follows from the coarea formula (cf. Theorem 3.11 in \cite{evans1991measure}) and translation invariance of the Hausdorff measure (cf. Theorem 2.2 in \cite{evans1991measure}). 
Recall that $R^{+}(D_{ij}) = \int_{D_{ij}} \rho\,d\cH^{d-1}$.

We first consider the case where $\cH^{d-1}(\partial \cY \cap D_{ij}) = 0$. Since $\rho$ is continuous $\cH^{d-1}$-a.e. on $\cY$ by Assumption \ref{asp: cont} and vanishes on $\cY^c$, for $\cH^{d-1}$-a.e. $\vy \in D_{ij}$ and every $\tau \in \R$, we have $\rho(\vy + t\tau\vv) \to \rho(\vy)$ as $t \downarrow 0$.  
Combined with Condition (\ref{eq: DCT}), we may apply the dominated convergence theorem to conclude that
\[
R\big(\tilde D_{ij}(\vh_1,\vh_2,t)\big) = R^{+}(D_{ij})  \frac{t (h_{2,j}-h_{2,i}-h_{1,j}+h_{1,i})}{\|\vx_i - \vx_j\|} + o(t).
\]

Next, suppose that $\cH^{d-1}(\partial \cY \cap D_{ij}) > 0$, which entails $\cH^{d-1}(\partial \cY \cap H_{ij}) > 0$. By Assumption \ref{asp: cont}, this happens only when $\cY$ is polyhedral, in which case $H_{ij}$ is a supporting hyperplane to $\cY$ and hence $\inte (\cY) \cap H_{ij} = \varnothing$. 
So, we have either $\langle \vx_i-\vx_j,\vy \rangle - b_{ij}> 0$ for all $\vy \in \interior(\cY)$, or $\langle \vx_i-\vx_j,\vy \rangle - b_{ij}< 0$ for all $\vy \in \interior(\cY)$. This implies $p_i = R\big(C_i(\vz^*)\big) = 0$ or $p_j = R\big(C_j(\vz^*)\big) = 0$, both of which are contradictions. This completes the proof of Lemma \ref{lem: thin mass}.
\end{proof}

\begin{proof}[Proof of Lemma \ref{lem: quant_stability}]
Set $\varepsilon_0 = \min_{1 \le i < j \le N} \| \vx_i-\vx_j \| > 0$. If $\vy \in C_i(\vz) \setminus C_i(\vz')$, then there exists $j \ne i$ such that
\[
z_j-z_j^*-(z_i-z_i^*)\le \langle \vx_i-\vx_j,\vy \rangle - b_{ij} < z_j'-z_j^*-(z_i'-z_i^*).
\]
Hence,
\[
C_i(\vz) \setminus C_i(\vz') \subset \bigcup_{j \ne i} \big \{ \vy : z_j-z_j^*-(z_i-z_i^*) \le \langle \vx_i-\vx_j,\vy \rangle - b_{ij} < z_j'-z_j^*-(z_i'-z_i^*) \big \}.
\]
Fix any $j \ne i$. Consider the hyperplane $H = \{ \vy : \langle \vx_i-\vx_j,\vy \rangle = b_{ij} \}$ and the associated unit normal vector $\vv = (\vx_i-\vx_j)/\|\vx_i-\vx_j\|$. Then,
\begin{equation}
\begin{split}
&R\Big ( \big \{ \vy : z_j-z_j^*-(z_i-z_i^*) \le \langle \vx_i-\vx_j,\vy \rangle - b_{ij}(\vzero) < z_j'-z_j^*-(z_i'-z_i^*) \big \} \Big ) \\
&= \int_{\frac{z_j-z_j^*-(z_i-z_i^*)}{\|\vx_i-\vx_j\|}}^{\frac{z_j'-z_j^*-(z_i'-z_i^*)}{\|\vx_i-\vx_j\|}} \left ( \int_{H} \rho(\vy + t\vv) \, d\cH^{d-1}(\vy) \right )\, dt.
\end{split}
\label{eq: stability}
\end{equation}
The inner integral is bounded by 
\[
\sup_{|t| \le t_1}\int_{H}  \rho (\vy + t\vv) \, d\cH^{d-1} (\vy) < \infty
\]
with 
\[
t_1 \coloneqq \frac{2U}{\varepsilon_0} \ge \frac{|z_j-z_j^*-(z_i-z_i^*)| \vee | z_j'-z_j^*-(z_i'-z_i^*) |}{\|\vx_i-\vx_j\|}.
\]
Hence, the right-hand side of \eqref{eq: stability} is bounded by
\[
\frac{|z_j'-z_j| + |z_i'-z_i|}{\varepsilon_0} \sup_{|t| \le t_1}\int_{H}  \rho (\vy + t\vv) \, d\cH^{d-1} (\vy),
\]
which leads to the desired result. 
\end{proof}

\medskip

6.1.2.~\textit{Proof of Theorem \ref{thm: Hadamard_diff} (ii)}. \ 
Fr\'{e}chet differentiability is equivalent to Hadamard differentiability for a function defined on an open subset of a Euclidean space, so we shall prove Hadamard differentibility of $\gamma_{\bm{\varphi}}$. We continue using the same notation as in Part (i). 
Observe that
\[
\gamma_{\bm{\varphi}}(\vz^* + t\vh) - \gamma_{\bm{\varphi}}(\vz^*) = \sum_{1 \le i \neq j \le N} \int_{C_i(\vz^*) \cap C_j(\vz^* + t\vh)} \langle \vx_j - \vx_i, \bm{\varphi}(\vy) \rangle \ dR(\vy).
\]

Let $\vh \in \R^N$ and $t > 0$. Fix any $i,j \in \{1,\dots,N \}$ with $i \ne j$. 
We have $C_i(\vz^*) \cap C_j(\vz^* + t\vh) \subset D_{ij}(\vzero,\vh,t)$ and $C_j(\vz^*) \cap C_i(\vz^* + t\vh) \subset D_{ji}(\vzero,\vh,t)$.
If $h_i = h_j$, then $R\big (D_{ij}(\vzero,\vh,t)\big) = R\big ( D_{ji}(\vzero,\vh,t)\big ) = 0$.

Consider the case where $h_j > h_i$, for which $D_{ji} (\vzero,\vh,t) = \varnothing$. Define a finite Borel measure $R_{\bm{\varphi}}$ on $\R^d$ by $dR_{\bm{\varphi}} = \| \bm{\varphi} \| \, dR = \| \bm{\varphi} \| \cdot \rho \, d\vy$. Arguing as in the proof of Part (i), we know that $R_{\bm{\varphi}}\big(D_{ij}(\vzero,\vh,t) \setminus  \big(C_i(\vz^*) \cap C_j(\vz^* + t\vh) \big)\big) = O(t^2)$, so that we have
\begin{align*}
    \int_{C_i(\vz^*) \cap C_j(\vz^* + t\vh)} \langle \vx_j - \vx_i, \bm{\varphi}(\vy) \rangle \, dR(\vy) =\int_{D_{ij}(\vzero,\vh,t)} \langle \vx_j - \vx_i, \bm{\varphi}(\vy) \rangle \, dR(\vy) + O(t^2).
\end{align*}
Next, set
\[
\tilde D_{ij}(\vh, t) \coloneqq\left \{\vy_0 + \tau \vv : \vy_0 \in D_{ij}, 0 \leq \tau \leq \frac{t(h_{j}- h_{i})}{\|\vx_i - \vx_j\|} \right\}
\]
with $\vv = (\vx_i - \vx_j)/\|\vx_i - \vx_j\|$. Arguing as in Step 1 in the proof of Lemma \ref{lem: thin mass} above, we see that $R_{\bm{\varphi}}(D_{ij}(\vzero,\vh,t) \Delta \tilde D_{ij}(\vh,t)) = O(t^2)$, so that 
\begin{align*}
&\int_{D_{ij}(\vzero,\vh,t)} \langle \vx_j - \vx_i, \bm{\varphi}(\vy) \rangle \, dR(\vy) = \int_{\tilde D_{ij}(\vh,t)} \langle \vx_j - \vx_i, \bm{\varphi}(\vy) \rangle \, dR(\vy) + O(t^2) \\&\qquad = t \int_0^{\frac{h_j - h_i}{\|\vx_i - \vx_j\|}} \left ( \int_{D_{ij}} \langle \vx_j - \vx_i,\bm{\varphi}(\vy + t\tau \vv) \rangle  \rho(\vy + t\tau \vv) \,d\cH^{d-1}(\vy)\right ) \,d\tau + O(t^2). 
\end{align*}
Since $\bm{\varphi}$ is continuous $\cH^{d-1}$-a.e. on $\cY \cap D_{ij}$ by Assumption \ref{asp: text function}, arguing as in Step 2 of the proof of Lemma \ref{lem: thin mass}, have
\[
\begin{split}
&\int_0^{\frac{h_j - h_i}{\|\vx_i - \vx_j\|}} \left ( \int_{D_{ij}} \langle \vx_j - \vx_i,\bm{\varphi}(\vy + t\tau \vv) \rangle  \rho(\vy + t\tau \vv) \,d\cH^{d-1}(\vy)\right ) \,d\tau \\
&= \frac{h_j - h_i}{\|\vx_i - \vx_j\|} \int_{D_{ij}}  \langle \vx_j - \vx_i, \bm{\varphi}(\vy) \rangle \rho(\vy) \, d\cH^{d-1}(\vy) + o(1).
\end{split}
\]
Conclude that 
\[
\begin{split}
&\int_{C_i(\vz^*) \cap C_j(\vz^* + t\vh)} \langle \vx_j - \vx_i, \bm{\varphi}(\vy) \rangle \, dR(\vy)\\ 
&\quad = \frac{t(h_i - h_j)}{\|\vx_i - \vx_j\|} \int_{D_{ij}}  \langle \vx_i - \vx_j, \bm{\varphi}(\vy) \rangle \rho(\vy) \, d\cH^{d-1}(\vy) + o(t), 
\end{split}
\]
when $h_j > h_i$. 

By symmetry, when $h_i > h_j$, we have
\[
\begin{split}
&\int_{C_j(\vz^*) \cap C_i(\vz^* + t\vh)} \langle \vx_i - \vx_j, \bm{\varphi}(\vy) \rangle \, dR(\vy)\\ 
&\quad = \frac{t(h_i - h_j)}{\|\vx_i - \vx_j\|} \int_{D_{ij}}  \langle \vx_i - \vx_j, \bm{\varphi}(\vy) \rangle \rho(\vy) \, d\cH^{d-1}(\vy) + o(t),
\end{split}
\]
where we used the fact that $D_{ji} = D_{ij}$. 

Therefore, we have shown the Gateaux differentiability for $\gamma_{\bm{\varphi}}$, 
\[
\lim_{t \downarrow 0} \frac{\gamma_{\bm{\varphi}}(\vz^* + t\vh) - \gamma_{\bm{\varphi}}(\vz^*)}{t} = \underbrace{\sum_{1 \le i < j \le N}\frac{h_i - h_j}{\|\vx_i - \vx_j\|} \int_{D_{ij}}  \langle \vx_i - \vx_j, \bm{\varphi}(\vy) \rangle \rho(\vy) \, d\cH^{d-1}(\vy)}_{=[\gamma_{\bm{\varphi}}]'_{\vz^*}(\vh)}.
\]
It remains to show that $\gamma_{\bm{\varphi}}$ is locally Lipschitz. Observe that
\begin{equation}
\begin{split}
    |\gamma_{\bm{\varphi}}(\vz) - \gamma_{\bm{\varphi}}(\vz')| &\leq \sum_{1 \le i \neq j \le N} \int_{C_i(\vz) \cap C_j(\vz')} |\langle \vx_j - \vx_i, \bm{\varphi}(\vy) \rangle|\ dR(\vy)\\
    &\leq \max_{1 \le i<j \le N} \|\vx_j - \vx_i\| \sum_{1 \le i \ne j \le N} R_{\bm{\varphi}}\big(C_i(\vz) \cap C_j(\vz')\big) \\
    &= \max_{1 \le i < j\le N} \| \vx_i-\vx_j \|\sum_{i=1}^N R_{\bm{\varphi}}\big(C_i(\vz) \setminus C_i(\vz')\big). 
    \end{split}
    \label{eq: lipschitz linear functional}
\end{equation}
Lemma \ref{lem: quant_stability} (with $\rho$ replaced by $\| \bm{\varphi} \| \cdot \rho$) now guarantees local Lipschitz continuity of $\gamma_{\bm{\varphi}}$, completing the proof.
\qed

\subsection{Proofs for Section \ref{sec: limit}}
\label{sec: proof limit}

We will prove Propositions \ref{lem: kitagawa} and \ref{prop: density}, Theorem \ref{thm: limit}, Lemma \ref{lem: nondegeneracy}, the claim in Remark \ref{lem: simple case}, Theorem \ref{prop: bootstrap consistency}, and Proposition \ref{prop: pointwise}.

\medskip

6.2.1.~\textit{Proof of Proposition \ref{lem: kitagawa}}. 
The proof can be simplified when $\cY$ is compact and convex and $\rho$ is sufficiently regular, since in that case, the results from \cite{kitagawa2019convergence,bansil2022quantitative} are directly applicable. Since $\cY$ may be unbounded and $\rho$ may have discontinuities in our setting, we need additional work. For the reader's convenience, we provide a (mostly) self-contained proof.

\medskip

\underline{Step 1}. We first establish regularity of the dual objective function $\Phi(\cdot,\bm{q})$ together with its strong concavity in certain directions when $\bm{q} \in \cQ_+$. 
Set $\bm{G}(\vz) = \big( G_i(\vz)  \big)_{i=1}^N = \big( R(C_i(\vz)) \big)_{i=1}^N$.
Lemma \ref{lem: quant_stability} implies that the map $\vz \mapsto \bm{G}(\vz) $ is continuous, so 
the set $\cZ_+ \coloneqq\big\{ \vz : G_i(\vz) > 0, \ \forall i \in \{ 1,\dots, N \} \big\}$ is open. For every $\bm{q} \in \cQ$, the map $\R^N \ni \vz \mapsto \Phi(\vz,\bm{q})$ is clearly concave. Furthermore, for every fixed $i \in \{1,\dots,N \}$ and $\vz \in \R^N$, 
\[
\frac{\partial}{\partial z_i'} \min_{1 \le j \le N} \big( \|\vy-\vx_j\|^2/2 - z_j' \big) \Big|_{\vz'=\vz} 
=
\begin{cases}
-1 & \text{if} \ \vy \in \inte(C_i(\vz)), \\
0 & \text{if} \ \vy \in \bigcup_{j \ne i}\inte(C_j(\vz)).
\end{cases}
\]
The dominated convergence theorem then yields that $\Phi(\cdot,\bm{q})$ is partially differentiable with gradient
\[
\nabla_{\vz} \Phi(\vz,\bm{q}) = \bm{q} - \bm{G}(\vz).
\]
Since the gradient is continuous in $\vz$, $\Phi(\cdot,\bm{q})$ is continuously differentiable. 
The set of optimal solutions $\cZ^*(\bm{q})$ to the dual problem (\ref{eq: dual}) is convex and agrees with $\langle \vone \rangle^{\bot} \cap \{ \vz : \nabla_{\vz} \Phi(\vz,\bm{q}) = \vzero \}$.
The duality theory for OT problems (cf. Theorem 6.1.5 in \cite{ambrosio2005}) guarantees that $\cZ^*(\bm{q})$ is nonempty for every $\bm{q} \in \cQ$. 

Next, we observe that the mapping $\vz \mapsto \bm{G}(\vz)$ is differentiable on $\cZ_+$. Indeed, for every $i \in \{ 1,\dots, N \}$, $\vh \in \R^N$, and $t > 0$, we have
\[
\begin{split}
    G_i(\vz + t\vh) - G_i(\vz) &= R\big(C_i(\vz + t\vh) \setminus C_i(\vz)\big) - R\big(C_i(\vz) \setminus C_i(\vz  +t \vh)\big) \\&=
    \sum_{j \neq i}\Big \{  R\big(C_i(\vz  +t \vh) \cap C_j(\vz)\big) - R\big(C_i(\vz) \cap C_j(\vz + t \vh)\big) \Big \}.
\end{split}
\]
Arguing as in the proof of Theorem \ref{thm: Hadamard_diff} (ii), we see that $G_i$ is Hadamard differentiable (and hence Fr\'{e}chet differentiable) at $\vz \in \cZ_+$ with derivative
\[
\R^N \ni \vh \mapsto \sum_{j \ne i} \frac{h_i-h_j}{\|\vx_i-\vx_j\|} R^+\big ( C_i(\vz) \cap C_j (\vz) \big), 
\]
that is,
\begin{equation}
\frac{\partial G_i(\vz)}{\partial z_j}=  
\begin{cases}
 - \frac{1}{\|\vx_i-\vx_j\|} R^+\big ( C_i(\vz) \cap C_j (\vz) \big) & \text{if} \ i \ne j, \\
 \sum_{k \ne i} \frac{1}{\|\vx_i-\vx_k\|} R^+\big ( C_i(\vz) \cap C_k (\vz) \big) & \text{if} \ i = j.
\end{cases}
\label{eq: hessian}
\end{equation}
Denote the Jacobian matrix of $\bm{G}$ by $\nabla \bm{G}$. For every $\vz \in \cZ_{+}$, the $N \times N$ matrix $\nabla \bm{G}(\vz)$ is symmetric positive semidefinite (as $-\nabla \bm{G}$ agrees with the Hessian of the concave function $\Phi(\cdot,\bm{q})$) with smallest eigenvalue $0$ corresponding to eigenvector $\vone$. 

Now, by Lemma \ref{lem: strict convexity} below, for every fixed $\vz \in \cZ_+$, the second smallest eigenvalue of $\nabla \bm{G}$ is bounded away from zero in a small neighborhood of $\vz$, i.e., for sufficiently small $\delta > 0$,
\begin{equation}
\inf_{\vz' : \| \vz'-\vz \| < \delta} \inf_{\vv \in \langle \vone \rangle^{\bot}} \frac{\vv^{\intercal} \nabla \bm{G}(\vz') \vv}{\| \vv \|^2} > 0.
\label{eq: eigenvalue}
\end{equation}
In particular, the linear mapping $\nabla \bm{G}(\vz) : \langle \vone \rangle^{\bot} \to \langle \vone \rangle^{\bot}$ is isomorphic.

\medskip

\underline{Step 2}. Now, we shall prove the claims of the proposition. 
For every $\bm{q} \in \cQ_+$, since $\cZ^*(\bm{q}) \subset \langle \vone \rangle^{\bot} \cap \cZ_+$, $\Phi(\cdot,\bm{q})$ is strictly concave on $\cZ^*(\bm{q})$, which implies that $\cZ^*(\bm{q})$ is a singleton, i.e., $\cZ^*(\bm{q}) = \{ \vz^*(\bm{q}) \}$. 
Next, we shall verify Hadamard differentiability of $\vz^*(\cdot)$ at $\vp \in \cQ_+$, which leads to the second claim of the proposition. 
Pick any sequence $t_n \downarrow 0$. Let $\vh_n$ be  a sequence in $\langle \vone \rangle^\bot$ such that $\vh_n \to \vh$. Set $\vz_n = \vz^*(\bm{q} + t_n\vh_n) \in \langle \vone \rangle^{\bot} \cap \cZ_+$ (for large $n$). By construction, we have
\[
\bm{G}(\vz_n) - \bm{G}(\vz^*) = t_n \vh_n =  O(t_n).
\] 
Since $\bm{G}|_{\langle \vone \rangle^{\bot} \cap \cZ_+}$ is one-to-one and continuous, the only possible cluster point of the sequence $\vz_n$ is $\vz_{\infty}$. By Lemma \ref{lem: compactness} below, the sequence $\vz_n$ is bounded, so we have
$\vz_n \to \vz_{\infty}$  by the compactness argument. By differentiability of $\bm{G}$ and the estimate in  \eqref{eq: eigenvalue}, we further obtain $\| \vz_n - \vz^* \| = O(t_n)$. 
Now, observe that
\[
t_n \vh_n = \bm{G}(\vz_n) - \bm{G}(\vz^*) = \nabla \bm{G}(\vz^*) (\vz_n - \vz^*) + o(t_n).
\]
Since $\nabla \bm{G}(\vz^*) (\vz_n - \vz^*) \in \langle \vone \rangle^{\bot}$ (as $\vz_n- \vz^* \in \langle \vone \rangle^{\bot}$), which entails $t_n \vh_n - o(t_n) \in \langle \vone \rangle^{\bot}$ above,
we have
\[
\frac{\vz_n - \vz^*}{t_n} \to \big (\nabla \bm{G}(\vz^*)|_{\langle \vone \rangle^{\bot}} \big)^{-1} \vh. 
\]
 The limit is linear in $\vh$. 
The preceding argument shows that  $\vz^*(\cdot)$ is Hadamard differentiable at $\vp$, completing the proof.
\qed

It remains to prove the following lemmas used in the preceding proof.

\begin{lemma}
\label{lem: strict convexity}
Set $\mathsf{C}_{\vx} = \max_{1 \le i<j \le N} \| \vx_i - \vx_ j\|$. 
Under Assumptions \ref{asp: cont} and \ref{asp: poincare}, for every $\varepsilon > 0$ and $\vz \in \R^N$ with $\min_{1 \le i \le N} G_i(\vz) \ge \varepsilon$, the second smallest eigenvalue of $\nabla \bm{G}(\vz)$ is at least $\frac{8\varepsilon}{N^4 \mathsf{C}_{\vx}\mathsf{C}_{\mathrm{P}}}$. 
\end{lemma}

\begin{proof}
The lemma essentially follows from \cite[Theorem 4.3]{bansil2022quantitative} with $DG = -\nabla \bm{G}$ under our setting. The basic idea  is to regard $\nabla \bm{G}(\vz)$ as the graph Laplacian of a weighted graph over the vertex set $\{1,\dots,N\}$ with the weighted adjacent matrix $\bm{A}=(a_{ij})_{1 \le i,j \le N}$ with $a_{ij} = \| \vx_i-\vx_j\|^{-1} R^{+}\big(C_i(\vz) \cap C_j(\vz) \big)$ for $i \ne j$. Using Cheeger's inequality and the assumption that $\min_{1 \le i \le N} G_i(\vz) \ge \varepsilon$, one can verify that the graph constructed by removing edges with small weights  is still connected, and then lower bound the second smallest eigenvalue of the corresponding graph Laplacian via diameter (Theorem 4.2 in \cite{mohar1991eigenvalues}). 
Since we do no assume compactness of the support of $R$, we briefly sketch the required modifications. 
Cheeger's inequality in \cite[Lemma 4.3]{bansil2022quantitative} follows by our Assumption \ref{asp: poincare}; see \cite[Lemma 2.2]{milman2007role} and \cite{bobkov1997iso}. See also the discussion after Assumption \ref{asp: poincare}. The proof of \cite[Proposition 4.4]{bansil2022quantitative} goes through with $2C_\nabla$ replaced by $\mathsf{C}_{\vx}$. 
Recalling that the derivative expression (\ref{eq: hessian}) holds everywhere on $\cZ_+$ and following the proof of \cite[Theorem 4.3]{bansil2022quantitative}, we obtain the result. 
\end{proof}

\begin{lemma}
\label{lem: compactness}
Under Assumption \ref{asp: cont}, the set $\langle \vone \rangle^{\bot} \cap \{ \vz : \min_{1 \le i \le N} G_i(\vz) \ge \varepsilon \}$ is compact for every fixed $\varepsilon > 0$. 
\end{lemma}

\begin{proof}
It suffices to show that the above set is bounded.
Pick any $\vz \in \langle \vone \rangle^{\bot} \cap \{ \vz : \min_{1 \le i \le N} G_i(\vz) \ge \varepsilon \}$. Assume without loss of generality that $z_1 = \min_{1 \le i \le N} z_i$ and $z_N = \max_{1 \le i \le N}z_i$. 
By strong duality, for $\vY \sim R$, 
\[
\begin{split}
\underbrace{\frac{1}{2} \E \big [\| \vY - T_{\vz}(\vY) \|^2 \big]}_{\ge 0} &= \langle \vz, \bm{G}(\vz) \rangle + \int \min_{1 \le i \le N} \left ( \frac{1}{2}\|\vy-\vx_i\|^2 - z_i \right ) \, dR(\vy) \\
&\le \langle \vz, \bm{G}(\vz) \rangle + \int  \left ( \frac{1}{2}\|\vy-\vx_N\|^2 - z_N \right ) \, dR(\vy) \\
&\le  ( z_1-z_N ) \underbrace{\min_{1 \le i \le N} G_i(\vz)}_{\ge \varepsilon} + \max_{1 \le i \le N} \int \frac{1}{2}\| \vy-\vx_i \|^2 \, dR(\vy), 
\end{split}
\]
so $z_N-z_1 \le \bar{\mathsf{C}} \coloneqq\varepsilon^{-1}\big( \E[\|\vY\|^2] + \max_{1 \le i \le N}\|\vx_i\|^2\big)$. 
This implies $z_1 \le z_i \le z_1 +\bar{\mathsf{C}}$ for every $i \ne 1$. 
Now, since $\vz \in \langle \vone \rangle^{\bot}$, we have $z_1 = -\sum_{i \ne 1}z_i$, so $\sum_{i \ne 1}z_i \le - (N-1) \sum_{i \ne 1}z_i + (N-1)\bar{\mathsf{C}}$, i.e., $z_1 = -  \sum_{i \ne 1} z_i \ge  -(1-N^{-1})\bar{\mathsf{C}}$. Likewise, we have $z_N = -\sum_{i \ne N} z_i \le (1-N^{-1})\bar{\mathsf{C}}$. Conclude that $|z_i| \le (1-N^{-1})\bar{\mathsf{C}}$ for every $i \in \{1,\dots,N \}$. 
\end{proof}

6.2.2.~\textit{Proof of Proposition \ref{prop: density}}. \ 
\underline{Part (i)}. The function $g$ is convex and onto $[0,\infty)$. 
The latter follows from the assumption that $\sum_{j=1}^N \beta_{ij} > 0$ for every $i \in \{ 1,\dots, N\}$. 
Hence, for every open interval $(a,b) \subset [0,\infty)$, the inverse image $g^{-1}((a,b))$ is nonempty and open, so that $\Prob(V \in (a,b)) = \Prob\big(\vW_{-N} \in g^{-1}((a,b)) \big) > 0$, yielding that the support of $V$ agrees with $[0,\infty)$. If $g(\vw_{-N}) = 0$, then $w_i = w_j$ for some $i \ne j$, so the level set $g^{-1}(\{ 0 \})$ has Lebesgue measure zero. 
This implies that the distribution function of $V$ is continuous at the left endpoint of the support of $V$. Now, noting that $g$ is convex, the claim of Part (i) follows from Theorem 11.1 in \cite{Davydov1998}.

\medskip
\underline{Part (ii)}.
Set $h (\vw) = \sum_{1 \le i < j \le N} \beta_{ij} |w_i - w_j|$ for $\vw = (w_1,\dots,w_N)^{\intercal}$. Then
\[
\nabla g(\vw_{-N}) = \underbrace{\begin{pmatrix}
1 & 0 & \cdots & 0 & -1 \\
0 & 1 & \cdots & 0 & -1 \\
\vdots & \vdots & \ddots & \vdots & \vdots \\
0 & 0 & \cdots & 1 & -1
\end{pmatrix}
}_{(N-1) \times N}
\nabla h(\vw), \quad \vw_{-N} \in \bigcup_{\sigma} E_\sigma, 
\]
with $w_{-N} = -\sum_{i=1}^{N-1}w_i$. 
The matrix in front of $\nabla h$ is isomorphic from $\langle \vone \rangle^{\bot} \subset \R^{N}$ onto $\R^{N-1}$. 
On the set $w_{\sigma(1)} > \cdots > w_{\sigma(N)}$, the partial derivatives of $h$ are given by
\[
\frac{\partial h}{\partial w_{\sigma(i)}} = \left (-\sum_{j=1}^{i-1}  + \sum_{j=i+1}^N \right ) \beta_{\sigma(i),\sigma(j)}.
\]
Since $\nabla h(\vw) \in \langle \vone \rangle^{\bot}$ and $\| \nabla h(\vw) \| \ge \sum_{j=2}^N \beta_{\sigma(1),\sigma(j)} > 0$ by assumption, we have 
$
\inf_{E_\sigma} \| \nabla g \| > 0.
$
Since $\{ E_\sigma \}_{\sigma}$ gives a partition of $\R^{N-1}$ up to Lebesgue negligible sets, we have 
\[
\mathrm{essinf}_{\R^{N-1}} \| \nabla g \| > 0.
\]
Since $g$ is Lipschitz, we may apply the coarea formula, Theorem 3.13 in \cite{evans1991measure}, to conclude that
\[
\frac{d}{dv} \int_{\{ g \le v\}} \phi_{\bm{\Sigma}}\mathbbm{1}_{E_\sigma} \, d\calH^{N-1} = \int_{\{ g = v\}} \frac{\phi_{\bm{\Sigma}} \mathbbm{1}_{E_\sigma}}{\| \nabla g \|} \, d\calH^{N-2}
\]
for Lebesgue almost every $v \in [0,\infty)$. The density formula now follows from the fact that $\| \nabla g \| = \mathsf{C}(B_\sigma)$ on $E_\sigma$.
\qed

\medskip

6.2.3.~\textit{Proof of Theorem \ref{thm: limit}}. \
\underline{Part (i)}. 
The first claim follows from combining the CLT for $\hat{\vz}_n$ in (\ref{eq: CLT}), the stability results in Theorem \ref{thm: Hadamard_diff}, and the extended delta method (Lemma \ref{lem: functional delta method}). 

The second claim follows from Proposition \ref{prop: density} with $\beta_{ij} = \| \vx_i-\vx_j \|^{s-1} R^+(D_{ij})$ and $\bm{\Sigma}$ being the upper left $(N-1) \times (N-1)$ submatrix of $\bm{B}\bm{A}\bm{B}^\intercal$.
We note that $\sum_{j \ne i} \beta_{ij} \ge \min_{j \ne i} \| \vx_i - \vx_j \|^{s-1} \sum_{j \ne i} R^+(D_{ij}) = \min_{j \ne i}\|\vx_i-\vx_j\|^{s-1} R^+(\partial C_i)> 0$ by Cheeger's inequality, and $\bm{\Sigma}$ is nonsingular as $\bm{B}\bm{A}\bm{B}^{\intercal}$ is isomorphic from $\langle \vone \rangle^{\bot}$ onto $\langle \vone \rangle^{\bot}$, so Proposition \ref{prop: density} applies.

For the final claim concerning moment convergence of $\Upsilon \big( \sqrt{n}\| \hat{\vT}_n - \vT^* \|_{L^s(R)}^s\big )$, it suffices to verify its uniform integrability. Since $\Upsilon$ has polynomial growth, it suffices to show that for every $k \in \NN$, $\sup_{n \in \NN} \E\big [ \big ( \sqrt{n}\| \hat{\vT}_n - \vT^* \|_{L^s(R)}^s \big)^k \big ] < \infty$. Arguing as in (\ref{eq: delta Lip2}), we have $\| \hat{\vT}_n - \vT^* \|_{L^s(R)}^s \le \max_{i \ne j} \| \vx_i-\vx_j \|^s \sum_i R\big ( C_i(\hat{\vz}_n) \setminus C_i (\vz^*) \big)$. Now, by Theorem 1.3 in \cite{bansil2022quantitative} and symmetry, we have $\sum_i R\big ( C_i(\hat{\vz}_n) \setminus C_i (\vz^*) \big) \le 2N \| \hat{\vp}_n - \vp \|_1$ with $\| \cdot \|_1$ denoting the $\ell^1$-norm (\cite{bansil2022quantitative} assume compactness of the support of $R$ but the proof of their Theorem 1.3 goes through under our setting).  Hence, $ \big ( \sqrt{n}\| \hat{\vT}_n - \vT^* \|_{L^s(R)}^s \big)^k$ can be bounded above, up to a constant independent of $n$,  by $\big (\sqrt{n} \| \hat{\vp}_n-\vp \|_1 \big)^k$ whose expectation is bounded in $n$ by a simple application of the Marcinkiewicz-Zygmund inequality (cf. Theorem 10.3.2 in \cite{chow2003probability}) or the concentration inequality for $\| \hat{\vp}_n - \vp \|_1$ \cite{weissman2003inequalities}.

\underline{Part (ii)}. The proof is analogous to Part (i) and omitted for brevity. 
\qed

\medskip 

6.2.4.~\textit{Proof of Lemma \ref{lem: nondegeneracy}}. \ 
For notational simplicity, we omit the dependence on~$\bm{\varphi}$. Since $a_{ji} = -a_{ij}$, we have 
$\sum_{1 \le i < j \le N} (W_i-W_j) a_{ij} = \sum_{i \ne j} W_i a_{ij} = \sum_{i=1}^N W_i \tilde{a}_i$. Hence, the variance $\sigma^2$ is
$
\sigma^2 = \tilde{\bm{a}}^{\intercal} \bm{B}\bm{A}\bm{B}^\intercal \tilde{\bm{a}},
$
where $\tilde{\bm{a}} = (\tilde{a}_1,\dots,\tilde{a}_N)^{\intercal}$. Since $\bm{B}\bm{A}\bm{B}^\intercal$ is isomorphic from $\langle \vone \rangle^\bot$ onto itself, $\sigma^2 = 0$ if and only if $\tilde{a}_1 = \cdots = \tilde{a}_N$. 
\qed

\medskip 

6.2.5.~\textit{Proof of the claim in Remark \ref{lem: simple case}}. \ 
Set $\bm{G}(\vz) = \big( G_i(\vz) 
 \big)_{i=1}^N = \big( R(C_i(\vz)) \big)_{i=1}^N$ as in the proof of Proposition \ref{lem: kitagawa}. Since $\langle \vx_i-\vx_j, \vy \rangle = b_{ij} = \|\vx_i\|^2/2-z_i^* - (\|\vx_j\|^2/2-z_j^*)$ for $\vy \in D_{ij}$, we have 
\[
\sum_{j \ne i} \frac{1}{\| \vx_i - \vx_j \|} \int_{D_{ij}} \langle \vx_i-\vy_j,\vy \rangle \rho (\vy) \, dR(\vy) = \big (\nabla G_i(\vz^*) \big)^{\intercal}\vh,
\]
where $\vh = \big(\|\vx_1\|^2/2-z_1^*,\dots,\|\vx_N\|^2/2-z_N^*\big)^{\intercal}$.
In view of Remark \ref{rem: rank} and the proof of the preceding lemma, we have
\[
\sigma_{\bm{\varphi}}^2 = \vh^{\intercal} \bm{A} \vh
=\Var_{P}\big(\| \cdot \|^2/2 - \psi \big),
\]
completing the proof.
\qed

\medskip

6.2.6.~\textit{Proof of Theorem \ref{prop: bootstrap consistency}}.
By the conditional CLT for the bootstrap (cf. Theorem 23.4 in \cite{vanderVaart1998asymptotic}), the conditional law of $\sqrt{n}(\hat{\vp}_n^{B}-\hat{\vp}_n)$ given the sample converges weakly to $\cN(\vzero,\bm{A})$ in probability.
The delta method for the bootstrap (cf. Theorem 23.5 in \cite{vanderVaart1998asymptotic}) yields that the conditional law of $\sqrt{n}(\hat{\vz}_n^B-\hat{\vz}_n)$ given the sample converges weakly to $\cN(\vzero,\bm{\Sigma})$ in probability, where $\bm{\Sigma} = \bm{B}\bm{A}\bm{B}^{\intercal}$.
Given this, the first claim of Part (ii) follows from another application of the delta method for the bootstrap.
For the second claim of Part (ii), from (\ref{eq: lipschitz linear functional}) and Theorem 1.3 in \cite{bansil2022quantitative}, we have $|\langle \bm{\varphi},\hat{\vT}_n^B - \hat{\vT}_n\rangle_{L^2(R)}| \le 2N\| \| \bm{\varphi} \| \|_{\infty} \max_{i < j} \| \vx_i - \vx_j \| \| \hat{\vp}_n^B - \hat{\vp}_n \|_{1}$ and the conditional fourth moment of $\sqrt{n}\| \hat{\vp}_n^B - \hat{\vp}_n \|_{1}$ is bounded in probability by direct computations. So the second claim of Part (ii) follows by Lemma 2.1 in \cite{kato2011note}.

For the rest, we focus on proving Part (i). 
We first observe that
\[
\begin{pmatrix}
\sqrt{n}(\hat{\vz}_n^B-\hat{\vz}_n) \\
\sqrt{n}(\hat{\vz}_n-\vz^*)
\end{pmatrix}
\stackrel{d}{\to}
\cN \left ( \vzero, \begin{pmatrix} \bm{\Sigma} & \vzero \\ \vzero & \bm{\Sigma} 
\end{pmatrix}
\right)
\]
unconditionally (cf. Problem 23.8 in \cite{vanderVaart1998asymptotic}), which implies that
\[
\begin{pmatrix}
\sqrt{n}(\hat{\vz}_n^B-\vz^*) \\
\sqrt{n}(\hat{\vz}_n-\vz^*)
\end{pmatrix}
\stackrel{d}{\to}
\cN \left ( \vzero, \begin{pmatrix} 2\bm{\Sigma} & \bm{\Sigma} \\ \bm{\Sigma} & \bm{\Sigma} 
\end{pmatrix}
\right)
.
\]
From this, the extended delta method (Lemma \ref{lem: functional delta method}) yields 
\[
\sqrt{n}\|\hat{\vT}_n^B - \hat{\vT}_n \|_{L^s(R)}^s = \sqrt{n}\delta_s(\hat{\vz}_n^B,\hat{\vz}_n) = [\delta_s]_{(\vz^*,\vz^*)}'\big( \sqrt{n}(\hat{\vz}_n^B-\vz^*),\sqrt{n}(\hat{\vz}_n-\vz^*) \big) + \varepsilon_n
\]
with $\varepsilon_n \to 0$ in probability. From the expression (\ref{eq: h_deriv}) of the derivative, we see that
\[
[\delta_s]_{(\vz^*,\vz^*)}'\big( \sqrt{n}(\hat{\vz}_n^B-\vz^*),\sqrt{n}(\hat{\vz}_n-\vz^*) \big) = \underbrace{[\delta_s]_{(\vz^*,\vz^*)}'\big( \sqrt{n}(\hat{\vz}_n^B-\hat{\vz}_n),\vzero \big)}_{=: S_n^B}.
\]
By the continuous mapping theorem, the conditional law of $S_n^B$ given the sample converges weakly to the limit law in (\ref{eq: limit law}) in probability. Let $S$ be a random variable following the limit law in (\ref{eq: limit law}). We will show that
\[
\sup_{g \in \mathsf{BL}_1(\R)} \Big |\E\big[ g(S_n^B + \varepsilon_n) \mid \vX_1,\dots,\vX_n \big] - \E[g(S)] \Big | \to 0
\]
in probability, which leads to the conclusion of Part (i). For every $g \in \mathsf{BL}_1(\R)$, 
\[
\big|g(S_n^B + \varepsilon_n) - g(S_n^B)\big| \le 2 \wedge |\varepsilon_n|.
\]
By Markov's inequality, we obtain
\[
\E\big[ 2 \wedge |\varepsilon_n| \,\big|\, \vX_1,\dots,\vX_n \big] \to 0
\] 
in probability. Now, we have
\[
\begin{split}
&\sup_{g \in \mathsf{BL}_1(\R)}
\Big|\E\big[ g(S_n^B + \varepsilon_n) \,\big|\, \vX_1,\dots,\vX_n \big] - \E[g(S)] \Big | \\
&\quad \le \sup_{g \in \mathsf{BL}_1(\R)}
\Big|\E\big[ g(S_n^B) \,\big|\, \vX_1,\dots,\vX_n \big] - \E[g(S)] \Big |  + \E\big[ 2 \wedge |\varepsilon_n| \,\big|\, \vX_1,\dots,\vX_n \big],
\end{split}
\]
and both terms on the right-hand side converge to zero in probability.
\qed

\medskip

6.2.7.~\textit{Proof of Proposition \ref{prop: pointwise}}. \ 
Since $K \subset \interior \big(C_i(\vz^*)\big)$ is compact, we have
    \[
    \varepsilon_0 = \min_{j \ne i} \inf_{\vy \in K} \big \{ \langle \vx_i - \vx_j, \vy \rangle - b_{ij} \big\} > 0.
    \]
    For every $j \ne i$,
    \[
    \begin{split}
    \langle \vx_i - \vx_j, \vy \rangle - b_{ij} (\hat{\vz}_n) &= \langle \vx_i - \vx_j, \vy \rangle -b_{ij} + (\hat{z}_{n,i}-z_i^*) - (\hat{z}_{n,j}-z_j^*) \\
    &\ge \varepsilon_0 - 2\max_{1 \le k \le N} |\hat{z}_{n,k} - z_{k}^*|,
    \end{split}
    \]
    so that we have
    \be
    \label{eq: superconsistency_1}
    \begin{split}
    &\PP\big( \hat{\vT}_n(\vy) = \vT^*(\vy),
  \forall \vy \in K\big) \ge \Prob \Big ( K \subset \inte\big(C_i(\hat{\vz}_n)\big) \Big) \\
  &\quad \ge \Prob \left ( \max_{1 \le k \le N} |\hat{z}_{n,k} - z_{k}^*| <  \varepsilon_0/2 \right ) \ge \Prob \left ( \| \hat{\vz}_n - \vz^* \| <  \varepsilon_0/2 \right ).  
  \end{split}
    \ee

    We derive a concentration inequality for $\| \hat{\vz}_n - \vz^* \|$. 
   Set $\hat{p}_{n,(1)} = \min_{1 \le j \le N} \hat{p}_{n,j}$. The union bound and Hoeffding's inequality (cf. \cite{boucheron2013concentration}) yield
\[
       \PP\left ( \hat{p}_{n,(1)} < \delta_0 \right )\leq \sum_{j=1}^N\PP\big ( \hat p_{n,j} < \delta_0 \big ) \le \sum_{j=1}^N \Prob \big(\hat{p}_{n,j} - p_j < -\delta_0\big) 
       \leq N e^{-2n\delta_0^2}.
\]
Set $\bm{G}(\vz) = \big( G_i(\vz) 
 \big)_{i=1}^N = \big( R(C_i(\vz)) \big)_{i=1}^N$ as in the proof of Proposition \ref{lem: kitagawa}. Suppose $\hat{p}_{n,(1)} \ge \delta_0$ holds. From the proof of Proposition \ref{lem: kitagawa}, we know that
 \[
\frac{d}{dt} \vz^*\big( t\hat{\vp}_n^{-N} + (1-t) \vp^{-N} \big) = \Big ( \nabla \bm{G}\big(\vz^*\big(t \hat{\vp}_n^{-N} + (1-t)\vp^{-N} \big)  \big)\big|_{\langle \vone \rangle^{\bot}}\Big)^{-1} \big (\hat{\vp}_n - \vp \big). 
 \]
 Since $\min_{1 \le j \le N} \big(t \hat{p}_{n,j} +(1-t)tp_{j} \big) \ge \delta_0$ for $t \in [0,1]$, using Lemma \ref{lem: strict convexity}, we have
\begin{equation}
\begin{split}
       \|\hat \vz_n - \vz^*\| 
       &=\left \| \int_0^1 \frac{d}{dt} \vz^*\big(t \hat{\vp}_n^{-N} + (1-t) \vp^{-N} \big) \, dt \right \| \\
       &\le \frac{N^4 \mathsf{C}_{\vx}\mathsf{C}_{\mathrm{P}}}{8 \delta_0} \| \hat{\vp}_n-\vp \| \le \frac{N^4 \mathsf{C}_{\vx}\mathsf{C}_{\mathrm{P}}}{8 \delta_0} \| \hat{\vp}_n-\vp \|_1, 
       \end{split}
       \label{eq: z_stability}
\end{equation}
   where $\| \cdot \|_{1}$ denotes the $\ell^1$-norm. 
Now, by \cite{weissman2003inequalities}, the following concentration inequality holds for $\|\hat \vp_n - \vp\|_{1}$,
   \begin{equation}
   \PP \left ( \|\hat \vp_n - \vp\|_{1} \geq t \right ) \leq (2^N-2)e^{-nt^2/2}, \ t > 0.
   \label{eq: concentration}
   \end{equation}
   Plugging \eqref{eq: z_stability} into \eqref{eq: superconsistency_1}, and using \eqref{eq: concentration}, we obtain
   \begin{align*}
       \PP\big( \hat{\vT}_n(\vy) = \vT^*(\vy),
  \forall \vy \in K\big)
  &\ge 1 - \PP\left ( \hat{p}_{n,(1)} < \delta_0 \right ) - \PP \left ( \|\hat \vp_n - \vp\|_{1} \geq  \frac{4\varepsilon_0 \delta_0}{N^4\mathsf{C}_{\vx} \mathsf{C}_{\mathrm{P}} }\right )\\
  &\geq 1 - N e^{-2n \delta_0^2}  - (2^N-2) e^{-\frac{8n \varepsilon_0^2 \delta_0^2}{N^8 \mathsf{C}_{\vx}^2 \mathsf{C}_{\mathrm{P}}^2}}.
   \end{align*}
   The final claim follows from the Borel-Cantelli lemma.
\qed

\subsection{Proofs for Section \ref{sec: applications}}
\label{sec: proof Sec4}

We provide proofs of Corollaries \ref{prop: L1 confidence set}, \ref{cor: bootstrap maximum correlation}, and \ref{prop: confidence band}.

\medskip

6.3.1.~\textit{Proofs of Corollaries \ref{prop: L1 confidence set} and \ref{cor: bootstrap maximum correlation}}. \ 
Both corollaries directly follow from Theorem \ref{prop: bootstrap consistency} and Lemma 23.3 in \cite{vanderVaart1998asymptotic}, upon noting that the limit law has a continuous distribution function in each case.
\qed

\medskip

6.3.2.~\textit{Proof of Corollary \ref{prop: confidence band}}. \ 
Suppose Assumptions \ref{asp: cont} and \ref{asp: poincare} hold.  
Let $\vY \sim R$ be independent of the sample. Define the event $A = \big \{ \sqrt{n}\| \hat{\vT}_n - \vT^* \|_{L^1(R)}  \le \hat{\tau}_{n,1-\alpha/2} \big \}$. Then, we have 
\[
\begin{split}
&1-\int \Prob \big ( \vT^*(\vy) \in \cC_{n,1-\alpha}(\vy) \big) \, dR(\vy)  \\&= \Prob \big ( \vT^*(\vY) \notin \cC_{n,1-\alpha}(\vY) \big) \\&=\Prob \left ( \sqrt{n}\| \hat{\vT}_n (\vY) - \vT^*(\vY) \| > \hat{\tau}_{n,1-\alpha/2}\cdot \frac{2}{\alpha} \right ) \\
&\le \Prob \left ( \left \{ \sqrt{n}\| \hat{\vT}_n (\vY) - \vT^*(\vY) \| > \hat{\tau}_{n,1-\alpha/2} \cdot \frac{2}{\alpha} \right \} \cap A \right ) + \Prob (A^c).
\end{split}
\]
By construction, $\Prob(A^c) = \alpha/2 + o(1)$. Furthermore, on the event $A$,  Markov's inequality yields 
\[
\begin{split}
&\Prob \left ( \left \{ \sqrt{n}\| \hat{\vT}_n (\vY) - \vT^*(\vY) \| > \hat{\tau}_{n,1-\alpha/2} \cdot \frac{2}{\alpha} \right \} \,\middle|\, \vX_1,\dots,\vX_n \right ) \\
&\le \frac{\sqrt{n}\| \hat{\vT}_n - \vT^* \|_{L^1(R)}}{\hat{\tau}_{n,1-\alpha/2}} \cdot \frac{\alpha}{2} \le \frac{\alpha}{2}.
\end{split}
\]
Conclude that 
\[
1-\int \Prob \big ( \vT^*(\vy) \in \cC_{n,1-\alpha}(\vy) \big) \, dR(\vy) \le \alpha+o(1). 
\]
\qed

\subsection{Proofs for Section \ref{sec: dual holder}}
\textcolor{black}{We provide proofs of Proposition \ref{prop: impossibility}, Theorem \ref{thm: dual_holder_limit}, and Proposition \ref{prop: asymptotic efficiency}.}

\medskip

6.4.1.~\textit{Proof of Proposition \ref{prop: impossibility}}. \ 
\textcolor{black}{If $r_n= o(n^{1/4})$, then $\big \| r_n\big(\hat{\vT}_n - \vT^*\big) \big \|_{L^2(R)}^2 \to 0$ in probability by Theorem \ref{thm: limit} (i). Consider the case where $n^{1/4} = O(r_n)$. 
Suppose on the contrary that $r_n\big(\hat{\vT}_n - \vT^*\big)$ were convergent in distribution in $L^2(R;\R^d)$, $r_n\big(\hat{\vT}_n - \vT^*\big) \stackrel{d}{\to} \mathbb{U}$ in $L^2(R;\R^d)$. In view of Theorem \ref{thm: limit} (i), $r_n$ must be exactly of order $n^{1/4}$. For simplicity, take $r_n = n^{1/4}$.  Since $L^2(R;\R^d)$ is  separable, $\mathbb{U}$ is tight by Ulam's theorem, so for every $\varepsilon, \delta > 0$, there exists finite $J$ such that $\Prob \left ( \sum_{j > J} \langle \bm{\varphi}_j,\mathbb{U} \rangle_{L^2(R)}^2 > \delta \right ) \le \varepsilon$ by Lemma 1.8.1 in \cite{van1996weak}. However, by Theorem \ref{thm: limit} (ii), for every $j \in \NN$, $n^{1/4} \langle \bm{\varphi}_j, \hat{\vT}_n - \vT^* \rangle_{L^2(R)} \to 0$ in probability, so $\sum_{j=1}^J \langle \bm{\varphi}_j,\mathbb{U} \rangle_{L^2(R)}^2 = 0$ a.s. This implies that $\| \mathbb{U} \|_{L^2(R)}^2 = \sum_{j=1}^\infty \langle \bm{\varphi}_j,\mathbb{U} \rangle_{L^2(R)}^2= 0$ a.s., which contradicts the conclusion of Theorem \ref{thm: limit} (i).}
\qed

\medskip

6.4.2.~\textit{Proof of Theorem \ref{thm: dual_holder_limit}}.
\ 
\textcolor{black}{\underline{Part (i)}. 
Since $\R^N$ is finite dimensional, it suffices to show that $\Gamma$ is Hadamard differentiable at $\vz^*$. 
We first observe that
\begin{align*}
    \|\Gamma(\vz') - \Gamma(\vz)\|_{\BB^\ast} 
    &= \sup_{\bm\varphi \in \BB_1} \left | \sum_{i \neq j} \int_{C_i(\vz)\cap C_j(\vz')} \langle \bm\varphi, \vx_i - \vx_j \rangle \, dR \right |\\
    &\leq \left ( \sup_{\bm\varphi \in \BB_1} \big \|  \|\bm\varphi \| \big \|_\infty \right) \max_{i \neq j} \|\vx_i - \vx_j\| \sum_{i} R(C_i(\vz) \Delta C_i(\vz')).
\end{align*}
Lemma \ref{lem: quant_stability} then implies that $\Gamma$ is locally Lipschitz. Hence, it suffices to verify Gateaux differentiability of $\Gamma$.
The proof is analogous to that of Theorem \ref{thm: Hadamard_diff}~(ii), and we follow the notation used there. Observe that
\[
\Gamma(\vz + t \vh)(\bm \varphi) - \Gamma(\vz) (\bm \varphi)= \sum_{1 \le i \neq j \le N} \int_{C_i(\vz^*) \cap C_j(\vz^* + t\vh)} \langle \vx_j - \vx_i, \bm{\varphi}(\vy) \rangle \ dR(\vy).
\]
Pick any $i,j \in \{ 1,\dots, N \}$ with $i \ne j$, 
and consider the case where $h_j > h_i$ (the $h_j < h_i$ case is analogous). Arguing as in the proof of Theorem \ref{thm: Hadamard_diff}~(ii) and using the fact that  the functions in $\BB_1$ are uniformly bounded, we see that the expansion
\[
\int_{C_i(\vz^*) \cap C_j(\vz^* + t\vh)} \langle \vx_j - \vx_i, \bm{\varphi}(\vy) \rangle \, dR(\vy) = \int_{\tilde D_{ij}(\vh,t)} \langle \vx_j - \vx_i, \bm{\varphi}(\vy) \rangle \, dR(\vy) + O(t^2).
\]
holds uniformly over $\bm \varphi \in \BB_1$, and the first term on the right hand side reduces to
\[
 t \int_0^{\frac{h_j - h_i}{\|\vx_i - \vx_j\|}} \left ( \int_{D_{ij}} \langle \vx_j - \vx_i,\bm{\varphi}(\vy + t\tau \vv) \rangle  \rho(\vy + t\tau \vv) \,d\cH^{d-1}(\vy)\right ) \,d\tau. 
\]
Now, using the fact that the functions in $\BB_1$ are uniformly bounded and uniformly equicontinuous, we see that the above integral can be expanded as
\[
\begin{split}
\frac{t(h_j - h_i)}{\|\vx_i - \vx_j\|} \int_{D_{ij}}  \langle \vx_j - \vx_i, \bm{\varphi}(\vy) \rangle \rho(\vy) \, d\cH^{d-1}(\vy) + o(t)
\end{split}
\]
uniformly over $\bm \varphi \in \BB_1$. Hence, $\Gamma$ is Gateaux differentiable with derivative (in the direction $\vh$) given by
\[
\bm{\varphi} \mapsto \sum_{i < j}\frac{h_j - h_i}{\|\vx_i - \vx_j\|} \int_{D_{ij}}  \langle \vx_j - \vx_i, \bm{\varphi}(\vy) \rangle \rho(\vy) \, d\cH^{d-1}(\vy) = \sum_{i=1}^N h_ib_i^*(\bm{\varphi}).
\]
}

\textcolor{black}{\underline{Part (ii)}. This follows from the CLT for $\hat{\vz}_n$ and the extended delta method.}
\qed

\medskip

6.4.3.~\textit{Proof of Proposition \ref{prop: asymptotic efficiency}}. \ 
\textcolor{black}{\underline{Part (i)}. 
We first show that $\vT_n(h) = \Gamma (\vz^*(\vp_{n,\vh}))$ is regular. Recall that $\Gamma$ is Fr\'echet differentiable at $\vz^* = \vz^*(\vp)$ with derivative $\vh \mapsto \sum_{i=1}^N h_i b_i^* = \Gamma_{\vz^*}'(\vh)$ and $\vz^*(\cdot)$ is Hadamard differentiable at $\vp$ tangentially to $\langle \vone \rangle^\bot$ with derivative $\vh \mapsto \bm{B}\vh$. The chain rule for Hadamard differentiable maps implies that the composition map $\Gamma \circ \vz^*(\cdot)$ is Hadamard differentiable at $\vp$ tangentially to $\langle \vone \rangle^\bot$ with derivative $\Gamma_{\vz^*}' (\bm{B} \vh)$. Since $\sqrt{n}(\vp_{n,\vh}-\vp) = \vp \odot \vh = (p_ih_i)_{i=1}^N \in \langle \vone \rangle^\bot$, the parameter sequence $\vT_n (\vh)$ satisfies that $\sqrt{n}(\vT_n (\vh) - \vT_n(\vzero)) \to  \Gamma_{\vz^*}' \big(\bm{B} (\vp \odot \vh) \big) =: \dot \vT(\vh)$, which is linear in $\vh$ and continuous from $H$ into $\BB^*$.}

\textcolor{black}{
Next, we wish to show that $\hat{\vT}_n$ is regular. Hadamard differentiability of $\vz^*(\cdot)$ at $\vp$ implies that, under $P_{n,\vzero}$, 
\[
\Big (\sqrt{n}(\hat{\vz}_n - \vz^*),\log \frac{dP_{n,\vh}}{dP_{n,0}} \Big) \stackrel{d}{\to} (\vW,\Lambda) \sim \cN \left ( \binom{\vzero}{-\frac{\|\vh\|_{\vp}^2}{2}}, \begin{pmatrix} \bm{B}\bm{A}\bm{B}^\intercal & \bm{B} (\vp \odot \vh) \\ (\vp \odot \vh)^{\intercal} \bm{B}^\intercal & \| \vh \|_{\vp}^2  \end{pmatrix}\right ).
\]
Combining Hadamard differentiability of $\Gamma$ at $\vz^*$, we have
\[
\Big (\sqrt{n}(\hat{\vT}_n - \vT_n(\vzero)),\log \frac{dP_{n,\vh}}{dP_{n,0}} \Big) \stackrel{d}{\to} (\Gamma_{\vz^*}'(\vW),\Lambda) \quad \text{in $\BB^* \times \R$}
\]
under $P_{n,\vzero}$. Now,
by Le Cam's third lemma \cite[Theorem 3.10.7]{van1996weak}, we have
\[
 \sqrt{n}(\hat{\vT}_n - \vT_n(\vzero)) \stackrel{d}{\to} \Gamma_{\vz^*}'(\bm{B}(\vp \odot \vh)+\vW) \stackrel{d}{=} \Gamma_{\vz^*}'(\bm{B}(\vp \odot \vh))+ \mathbb{G}
\]
under $P_{n,\vh}$, so that
\[
 \sqrt{n}(\hat{\vT}_n - \vT_n(\vh)) = \sqrt{n}(\hat \vT_n - \vT_n(\vzero)) -\underbrace{\sqrt{n}(\vT_n (\vh) - \vT_n(\vzero))}_{=\Gamma_{\vz^*}' (\bm{B} (\vp \odot \vh)) + o(1)} \stackrel{d}{\to}  \mathbb{G} 
\]
under $P_{n,\vh}$. This establishes regularity of $\hat{\vT}_n$.}

\textcolor{black}{
To prove the final claim of Part (i), from Theorem 3.11.2 in \cite{van1996weak}, it suffices to verify that for every $b^{**} \in \BB^{**}$ (the dual of $\BB^*$), $b^{**} (\mathbb{G})$ has a centered Gaussian distribution with variance $\sup_{h \in H: \| \vh \|_{\vp} = 1}|(b^{**} \circ \dot{\vT})(\vh)|^2$.
Recalling $\bm{A} = \mathrm{diag}\{ p_1,\dots,p_N \} - \vp\vp^\intercal$, we have $(b^{**} \circ \dot{\vT})(\vh) = \langle \bm{B}(\vp \odot \vh), \bm \beta \rangle = \langle \vp \odot \vh, \bm{B}^\intercal \bm \beta \rangle = \langle \vh,\bm{B}^\intercal \bm \beta \rangle_{\vp} = \vh^\intercal \bm{A} \bm{B}^\intercal \bm \beta$ for $\vh \in H$ with $\bm \beta = (b^{**}(b_i^*))_{i=1}^N$. Maximizing $|\vh^\intercal \bm{A} \bm{B}^\intercal \bm \beta|^2$ w.r.t. $\vh \in H$ with $\| \vh \|_{\vp}^2 = \vh^\intercal \bm{A} \vh = 1$ gives $\bm{\beta}^\intercal \bm{B} \bm{A} \bm{B}^\intercal \bm{\beta}$, which agrees with the variance of  
 $b^{**} (\mathbb{G}) = \sum_{i=1}^N W_i b^{**}(b_i^*) = \bm \beta^\intercal \vW$. Hence, Theorem 3.11.2 in \cite{van1996weak} applies, and the final claim of Part (i) follows.}

\textcolor{black}{\underline{Part (ii)}.
The first claim follows from Part (i) combined with  Theorem 3.11.5 in \cite{van1996weak}. For the moment convergence, since $I$ is finite, it suffices to verify that for every $\vh \in H$ and $k \in \NN$, $\sup_n \E_{\vh} \left [  \| \sqrt{n} (\hat{\vT}_n - \vT_n(\vh)) \|_{\BB^*}^k \right] < \infty$.
Observe that
    \begin{align*}
    \| \sqrt{n} (\hat{\vT}_n - \vT_n(\vh)) \|_{\BB^\ast} &= \sup_{\bm \varphi \in \BB_1} \left |  \sqrt{n} \sum_{i \neq j} \int_{C_i(\hat{\vz}_n) \cap C_j(\vz^\ast (\vp_{n,\vh}))} \langle \bm\varphi(\vy), \vx_i - \vx_j \rangle \,\rho(\vy) \,d d\vy \right | \\
    &\le O(1) \cdot   \sqrt{n} \sum_{i \neq j} R\big(C_i(\hat{\vz}_n) \cap C_j(\vz^\ast (\vp_{n,\vh})) \big) \\
    &\le O(1) \cdot \sqrt{n}\| \hat{\vp}_n-\vp_{n,\vh}\|_{1},
    \end{align*}
    where the final inequality follows from Theorem 1.3 in \cite{bansil2022quantitative}.
An application of the Marcinkiewicz-Zygmund inequality yields that $\sup_n \E_{\vh}\big [ \big ( \sqrt{n}\| \hat{\vp}_n-\vp_{n,\vh}\|_{1} \big)^k\big ] < \infty$, completing the proof.}
\qed

\section{Concluding remarks}
\label{sec: conclusion}

In this paper, we have established limit theorems for the integral error and linear functionals of the empirical OT map in the semidiscrete setting. The main ingredients of the proof are new stability estimates of these functionals with respect to the dual potential vector, whose derivation requires a careful analysis of the facial structures of the Laguerre cells. For both functionals, we have also established the consistency of the nonparametric bootstrap. These results enable constructing confidence sets/bands and lay the groundwork for principled statistical inference for (functionals of) the OT map in the semidiscrete setting. \textcolor{black}{Finally, we have shown that, while the empirical OT map does not possess nontrivial weak limits in $L^2(R)$, it satisfies a CLT in a dual H\"{o}lder space, and established its asymptotic efficiency for estimating the OT map when viewed as elements of the said Banach space.}

\begin{funding}
The second author was supported by NSF grants  CCF-2046018, and DMS-2210368, and the IBM Academic Award.
The third author was partially supported by NSF grants DMS-1952306, DMS-2014636, and DMS-2210368.
\end{funding}

\bibliographystyle{imsart-number}
\bibliography{main.bib}

\end{document}